\documentclass[12pt]{amsart}

\usepackage{layout} 
\usepackage[top=2cm, bottom=2cm, left=2cm, right=2cm]{geometry} 
\usepackage[english]{babel}
\usepackage[utf8]{inputenc}
\usepackage[T1]{fontenc}
\usepackage{amsmath}
\usepackage{amssymb}
\usepackage{mathrsfs}
\usepackage{amsthm}
\usepackage{enumerate} 
\usepackage{hyperref}
\usepackage{enumitem}
\usepackage{comment}
\usepackage{mathtools}
\usepackage{esint}
\usepackage{tikz}
\usepackage{faktor}
\usepackage{appendix}
\usepackage{color}
\usepackage[capitalize]{cleveref}

\numberwithin{equation}{section}
\theoremstyle{thmstyleone}%
\newtheorem{theorem}{Theorem}[section]
\newtheorem{proposition}[theorem]{Proposition}
\newtheorem{lem}[theorem]{Lemma}
\newtheorem{cor}[theorem]{Corollary}
\newtheorem{thmx}{Theorem}

\theoremstyle{thmstyletwo}%
\newtheorem{example}[theorem]{Example}
\newtheorem{remark}[theorem]{Remark}

\theoremstyle{thmstylethree}%
\newtheorem{definition}[theorem]{Definition}

\title[Einstein Type Systems on Complete Manifolds]{Einstein Type Systems on Complete Manifolds}
\date{\today}

\author{Rodrigo Avalos}
\address[Rodrigo Avalos]{Eberhard Karls Universität Tübingen, Fachbereich Mathematik, Auf der Morgenstelle 10, 72076 Tübingen, Germany}
\email{rodrigo.avalos@mnf.uni-tuebingen.de}

\author{Jorge Lira}
\address[Jorge Lira]{Mathematics Department, Universidade Federal do Ceará, Parquelândia, Fortaleza, 60020-181, Cear\'a, Brazil}
\email{jorge.lira@mat.ufc.br}

\author{Nicolas Marque}
\address[Nicolas Marque]{Institut \'Elie Cartan de Lorraine, Université de Lorraine}
\email{nicolas.marque@univ-lorraine.fr}

\begin{document}
	
	\maketitle
	
	\begin{abstract}
		In this paper, we study  the coupled Einstein constraint equations on complete manifolds through the conformal method, focusing on non-compact manifolds with flexible asymptotics. This is physically well-motivated by standard cosmological space-times with non-compact Cauchy hypersurfaces, which favour general bounded geometry manifolds rather than a specific model for infinity. First, we prove an existence criterion on complete manifolds with appropriate barrier functions  for physically well-motivated coupled systems. Then, in the bounded geometry case, we build barrier functions and thus show existence. We also prove an existence result on compact manifolds with boundary for a wider family of coupled systems.
	\end{abstract}

	

\section{Introduction}\label{sec1}

In this article we analyse existence results for relativistic initial data via the conformal method \cite{CB-ConformalMethod.1,York1,York2}, 
 focusing on constructions on complete non-compact manifolds without special asymptotic conditions. This work is motivated by well-established programs devoted to the analysis of geometric partial differential equations (PDEs) on such complete manifolds, exemplified by \cite{AlbanseRigoli.1,AlbanseRigoli.2,GunPig,GunPig2,MR2962687}, and by recent advances in the analysis of the coupled Einstein constraint equations (ECE) of general relativity (GR),  initiated by \cite{Holst1,MaxwellFarCMC}. Within GR a space-time is defined as a Lorentzian manifold $(V^{n+1},\bar{g})$ satisfying the Einstein equations:
\begin{align}\label{EinsteinEqs}
\mathrm{Ric}_{\bar{g}} - \frac{R_{\bar{g}}}{2}\bar{g} + \Lambda\bar{g}=T(\bar{g},\bar{\psi}),
\end{align} 
where $\mathrm{Ric}_{\bar{g}}$ and $R_{\bar{g}}$ stand for the Ricci tensor and scalar curvature respectively; $\Lambda$ stands for the \emph{cosmological constant}, and $T$ stands for the \emph{energy-momentum} tensor field associated with some physical model.  Such $T$ will typically depend both on the space-time metric $\bar{g}$ and some collection of physical fields, here collectively denoted by $\bar{\psi}$. Demanding \emph{physically reasonable} causality conditions imposes topological restrictions on $V^{n+1}$. In particular, we always assume \emph{global hyperbolicity} which excludes backward travels in time and imposes that $V^{n+1}\cong M^n\times \mathbb{R}$ \cite{SanchezSplittingThm,SanchezSplittingThm2,GerochSplittingThm}.

In the above context, an initial data set for GR is given by the manifold $M^n$ together with  the induced Riemannian metric $g$ and its initial time derivative, 
 essentially given in terms of the extrinsic curvature $K$ of $M^n$ as an embedded hypersurface, which we define according to the convention
\begin{align*}
K(X,Y)\doteq \langle \mathbb{II}(X,Y), \mathcal{N} \rangle_{\bar{g}} \text{ for all } X,Y\in \Gamma(TM).
\end{align*}
Above $\mathbb{II}(X,Y) \doteq \left(  \overline{\nabla}_XY  \right)^\perp$ stands for the second fundamental form of $(M^n,g)\hookrightarrow (V^n,\bar{g})$ and $\mathcal{N}$ for the future pointing unit normal to $M$. The Gauss-Codazzi equations for hypersurfaces imply that the initial data $(M^n,g,K)$ must satisfy the Einstein constraint equations:
\begin{align}\label{ECE}
\begin{split}
R_g-\vert K \vert^2_g+(\mathrm{tr}_gK)^2 - 2\Lambda&=2\epsilon,\\
\mathrm{div}_gK-d\mathrm{tr}_gK&=J,
\end{split}
\end{align}
where above $\epsilon\doteq T(\mathcal{N},\mathcal{N})$ denotes the energy density induced by the matter fields, while $J\in \Gamma(T^{*}M)$ given by $J(X) \doteq -T(\mathcal{N},X)$ is the momentum density associated with these fields. In most cases of interest, these equations are also sufficient conditions for $(M^n,g,K)$ to admit evolution into a space-time satisfying (\ref{EinsteinEqs}) (see Y. Choquet-Bruhat's \cite{CB-WellPosedness1,CB-WellPosedness2} or self-contained and updated reviews in \cite{MR2473363,RingstromBook}). This strongly motivates the study of (\ref{ECE}) and the understanding of the associated space of solutions, which are topics connected with classic curvature prescription problems in geometric analysis.

The method best understood to deal with (\ref{ECE}) is the \emph{conformal-method}, which transforms it into a determined system of elliptic PDEs by splitting $(g,K)$ according to the choices
\begin{align}\label{ConformalData1}
\begin{split}
g=\phi^{\frac{4}{n-2}}\gamma,\;\; K=\phi^{-2}\left(\pounds_{\gamma,\mathrm{conf}}X + U \right)+ \frac{\tau}{n}g.
\end{split}
\end{align}
Above $(M^n,\gamma)$ is a fixed Riemannian manifold and 
\begin{align*}
\pounds_{\gamma, \mathrm{conf}}X&\doteq \pounds_{X}\gamma - \frac{2}{n}\mathrm{div}_{\gamma}X\: \gamma, \text{ for all } X\in \Gamma(TM)
\end{align*} 
is the conformal Lie derivative whose kernel consists of the conformal Killing fields (CKFs) of $\gamma$. Also, in (\ref{ConformalData1}), $U$ stands for a $(0,2)$ symmetric TT tensor field (\emph{i.e.} $\mathrm{tr}_{\gamma}U=0$ and $\mathrm{div}_{\gamma}U=0$) and $\tau\doteq \mathrm{tr}_{\gamma}K$ denotes the \emph{mean curvature} of the initial data set $(M,g,K)$. Direct computations show that (\ref{ConformalData1}) transforms (\ref{ECE}) into an elliptic system on $(\phi,X)$: 
\begin{equation}\label{Conformal-GaussCod}
\begin{aligned}
&a_n\Delta_{\gamma}\phi - R_{\gamma}\phi + \vert \tilde{K} \vert^2_{\gamma} \phi^{-\frac{3n-2}{n-2}} + \left(\frac{1-n}{n}\tau^2 + 2\epsilon  \right)\phi^{\frac{n+2}{n-2}}=0,\\
&\Delta_{\gamma,\mathrm{conf}}X - \left(\frac{n-1}{n}\nabla\tau + J\right)\phi^{\frac{2n}{n-2}}=0,
\end{aligned}
\end{equation}
where  $a_n\doteq \frac{4(n-1)}{n-2}$, while $\tilde K \doteq \pounds_{\gamma, \mathrm{conf}}X + U$ and $\Delta_{\gamma,\mathrm{conf}}:\Gamma(TM) \rightarrow \Gamma(T^{*}M)$ is an elliptic operator, known as the conformal Killing Laplacian (CKL), and defined by
\begin{align*}
\Delta_{\gamma,\mathrm{conf}}X\doteq \mathrm{div}_{\gamma}(\pounds_{\gamma, \mathrm{conf}}X).
\end{align*} 
One can then choose a given physical model which determines the sources $\epsilon$ and $J$ and yields an elliptic system with unknowns $(\phi,X)$ and geometric data $\mathcal{I}\doteq (\gamma,\tau,U)$. The scalar equation in (\ref{Conformal-GaussCod}) is referred to as the \emph{Lichnerowicz equation}, the vector equation is the \emph{momentum equation}.

The system \eqref{Conformal-GaussCod} decouples if $\epsilon,J=0$ and $\tau=\text{const}.$, which is known as the CMC (constant mean curvature) case. The momentum constraint then reduces to a selection of a CKF, and all of the analysis is centred on the corresponding Lichnerowicz equation. In this context  \cite{IsenbergCMC} provided a complete classification of the smooth CMC solutions on closed manifolds. Since then, several refinements have been obtained. In particular, the decoupled constraint system has been analysed on non-compact manifolds, including asymptotically Euclidean manifolds (AE)  \cite{CBIY,Maxwell1,MaxwellRoughAE}, asymptotically cylindrical manifolds \cite{ChruscielMazzeo1,ChruscielMazzeo2}, asymptotically hyperbolic (AH) manifolds \cite{ChruscielCMCHyperbolic,IsenbergCMCHyperbolic}. Recently, the decoupled Lichnerowicz equation has been analysed on general complete manifolds in \cite{AlbanseRigoli.1,AlbanseRigoli.2} where both existence and uniqueness results were obtained. Furthermore, low regularity results have been established, for instance in \cite{Maxwell1,MaxwellRoughAE,MaxwellRoughClosed,HolstLichCompact,CB2004,CBIP,arxivrodrigo}, and  some non-vacuum situations have been incorporated. In particular, the system still decouples under a CMC condition if the momentum density scales appropriately under conformal transformations, that is when $J=\phi^{-\frac{2n}{n-2}}\tilde{J}$, where $\tilde{J}$ is a datum constructed from the conformal data $\mathcal{I}$. 

In contrast to the above description, whenever the system is coupled far less is known. Some near CMC results are known to hold through implicit function arguments (see  \cite{CB2004,IsenbergMoncrief}) while the first far-from-CMC results case were due to \cite{Holst1}. This last result appeals to a  fixed point argument, which was modified in \cite{MaxwellFarCMC} to account for the vacuum case, previously left out. These two papers triggered several advances in the analysis of the coupled Gauss-Codazzi system (\ref{ECE}), such as those of \cite{DiltsIsenberMazzeoAE,Gicquaud2,Gicquaud1,HolstAE,HolstFarCMCWithBoundary,Nguyen,Premoselli1,Premoselli2,Valcu}, where some important non-compact manifolds (namely, AE and AH manifolds) were analysed and  some model sources (in particular scalar fields) were incorporated. Furthermore, some non-uniqueness issues have been made clear \cite{Premoselli2,MaxwellModelProblem,Maxwell-KasnerST,MaxwellDriftModel,IsenbergMurchadha,Walsh-nonuniq}, which is a feature which does not occur in the CMC case and has further motivated some variations of the conformal method such as \cite{MaxwellDriftModel}.\footnote{We further refer the reader to \cite{Gicquaud_2025} for recent results on how to restore uniqueness under certain additional conditions.}

The main objective of this paper is to analyse existence results for the coupled Gauss-Codazzi system (\ref{Conformal-GaussCod}) on general complete manifolds, without a specific asymptotic structure. This follows the spirit of the work done by G. Albanese and M. Rigoli in \cite{AlbanseRigoli.1,AlbanseRigoli.2}, but in the perspective of the developments commented above for the coupled system. This procedure will consist of two parts: a general existence criterion which relies on the availability of appropriate barrier functions, and then  explicit constructions for these barriers. Since our analysis will be sensitive to the specific non linearities present in our problem, we will focus on a fairly general and physically well-motivated situation, which is that of energy momentum sources with contributions from a perfect fluid and an electromagnetic field. In such a case, our system takes the following form (see Example \ref{020320211748}):
\begin{equation}\label{EMSystem-CompleteManifolds}
\begin{aligned}
&\!\!\!\!a_n\Delta_\gamma \phi - R_\gamma \phi + \left\vert \tilde K(X) \right\vert^2_\gamma \phi^{-\frac{3n-2}{n-2} } - \frac{n-1}{n} \tau^2 \phi^{\frac{n+2}{n-2} }+ 
2 \epsilon_1  \phi^{\frac{n+2}{n-2}} + 2 \epsilon_2 \phi^{-3} + 2 \epsilon_3 \phi^{\frac{n-6}{n-2} } =0, \\
&\!\!\!\! \Delta_{\gamma, \mathrm{conf}} X - \frac{n-1}{n} \nabla \tau \phi^{\frac{2n}{n-2} } - \omega_1\phi^{2\frac{n + 1}{n-2}} + \omega_2 =0,
\end{aligned}
\end{equation}
Above, the functions $\epsilon_1,\epsilon_2$ and $\epsilon_3$ represent the energy contributions of the physical sources involved, while $\omega_1$ and $\omega_2$ stand for the corresponding momentum densities. 

There are strong physical motivations to construct complete non-compact initial data sets with no fixed asymptotic model. Indeed, most standard cosmological models (such as  FLRW) appear to have non-compact Cauchy hypersurfaces modelling \emph{open universes}.  Moreover, FLRW models are CMC with $\tau_0\neq 0$, which resonates well with our results in Theorem \ref{ExistenceNonVacuum}.

As announced, we will first derive the following general existence criterion, using global barrier functions (see Definition \ref{GlobalBarriersDefn}).

\begin{thmx}\label{ExistenceCriteriaThm}
Let $(M^n, \gamma)$ be a  Riemannian manifold, $n\ge3$, with $\gamma$ a complete smooth metric with no global CKF. Consider the system (\ref{EMSystem-CompleteManifolds}) with coefficients satisfying  
\begin{equation} \label{FunctHypMain} 
R_{\gamma}, \epsilon_1, \epsilon_2, \epsilon_3,  \vert U \vert^2,  \tau^2 \in L^p_{\mathrm{loc}}(M) \text{ and } \omega_1,\omega_2, \nabla \tau \in L^{2}(M,dV_{\gamma})\cap L^p_{\mathrm{loc}}(M),
\end{equation}
for $p>n$.
and assume such system admits a pair of compatible global barrier functions $\phi_{-},\phi_{+}\in W^{2,p}_{loc}$, with $0< \phi_- \le \phi_+ \le m < \infty$, and that the first eigenvalue of the conformal Killing Laplacian satisfies $$\lambda_{1, \gamma, \mathrm{conf} }>0.$$ Then \eqref{EMSystem-CompleteManifolds} admits a $W^{2,p}_{\mathrm{loc}}$ solution $(\phi,X)$.
\end{thmx}

Global barriers were  first introduced in the analysis of the ECE in \cite{Holst1} and are tailored to provide uniform estimates for solutions to the momentum equation in \eqref{EMSystem-CompleteManifolds}, whenever $\phi$ is  replaced by a fixed $\varphi$ such that $\phi_{-}\leq \varphi\leq \phi_{+}$. This, in turn, allows one to produce solutions along a compact exhaustion of $M$ via a Schauder fixed point theorem, which is then coupled with a diagonal extraction scheme to generate a global solution. The necessary uniform controls on $X$ will rely on the first eigenvalue of the CKL, defined by
\begin{align}
\lambda_{1, \gamma, \mathrm{conf} } \doteq \inf_{u \in C^{\infty}_c(M, TM)\backslash \{ 0 \} }  \frac{ \int_M  \left\vert \pounds_{\gamma, \mathrm{conf} } u \right\vert^2 dV_\gamma }{ \int_M \vert u \vert^2 dV_\gamma }.
\end{align}
Positivity of $\lambda_{1, \gamma, \mathrm{conf} }$ allows one to guarantee the existence of solutions  to linear equations involving the CKL 
on a complete manifold $M$, with good estimates  on compacts (see Lemma \ref{050720211748}).

Concerning the functional hypotheses in the above theorem, the $L^p_{loc}$-conditions imposed in (\ref{FunctHypMain}) demand merely local integrability conditions. On the other hand, the momentum sources as well as $\nabla\tau$ are imposed with a global restriction, which forces a control at infinity. These arise from a necessity to have global controls on solutions for the momentum constraint, and a similar control is necessary in our strategy for constructing global supersolutions. It is worth noticing that the conditions in (\ref{FunctHypMain}) clearly allow for a wide range of mean curvature behaviours, allowing for far-from-CMC initial data. In particular, the condition $\nabla \tau\in L^2(M,dV_{\gamma})$ can be interpreted as a \emph{near CMC-condition at infinity}. Given the flexibility of the asymptotic structure of $M$, this can be seen as a substantial improvement on current existence theory for the conformally formulated ECE.

In order to construct global barrier functions for (\ref{EMSystem-CompleteManifolds}), we need to impose some further conditions on the global geometry and on the coefficients of the systems. In particular, in the following theorem establishing our main result, we will assume that $(M,\gamma)$ has \emph{bounded geometry} (see Appendix \ref{BoundedGeometry} for detailed definitions).

\begin{thmx}\label{ExistenceNonVacuum}
Let $(M^n, \gamma)$ be a smooth Riemannian manifold of bounded geometry $n\ge3$  and $p>n$. We make the following assumptions:
\begin{align}\label{FunctinalHypot}
\begin{split}
&R_\gamma, \epsilon_1, \epsilon_2, \epsilon_3,  \vert U \vert^2,  \tau^2 \in L^p_{\mathrm{loc}}(M) \text{ and }  \omega_1,\omega_2, d\tau \in L^{2}(M,dV_{\gamma}) \cap L^p(M,dV_{\gamma}),\\
& \lambda_{1,\mathrm{conf}} >0,\\
&a \doteq R_\gamma + \frac{n-1}{n} \tau^2 \in L^\infty (M), \,  \, a \ge a_0 >0,
\end{split}
\end{align}
for some constant $a_0$. Assume further that:
\begin{equation}\label{EnergyHypot}
\left\{ \begin{aligned}
&\epsilon_2 + \epsilon_3>0  \text{ if } n \le 6\\
& \epsilon_2 >0 \text{ if } n >6.
\end{aligned} \right.
\end{equation}
Then, there exists $C(n,M, \gamma, \lambda_{1, \mathrm{conf}})$  such that if
\begin{equation}\label{MeanCurvHypot}
\begin{split}
\left\vert R_\gamma \right\vert +& \max \left( \| d\tau \|_{L^{2}(M)}, \| d\tau \|_{L^{p}(M)}\right)  + \max \left( \| \omega_1 \|_{L^{2}(M)},\| \omega_1 \|_{L^{p}(M)}\right) \\
&+ \max\left(\| \omega_2 \|_{L^{2}(M)}, \| \omega_2 \|_{L^{p}(M)} \right) + \vert  U\vert+ \epsilon_1 + \epsilon_2 + \epsilon_3  \le C \tau^2,
\end{split}
\end{equation}
then  \eqref{EMSystem-CompleteManifolds} admits a $W^{2,p}_{\mathrm{loc}}$ solution $(\phi,X)$. In addition, if  \eqref{EnergyHypot} is replaced by uniform lower-bounds given by a positive constant $\varepsilon_0>0$,  then the physical metric $g$ is also complete.
\end{thmx}

Manifolds of bounded geometry are manifolds with positive injectivity radius and whose curvature tensor has bounded covariant derivatives. The analysis of the coupled system (\ref{EMSystem-CompleteManifolds}) on general manifolds of bounded geometry can be seen as a natural developing step in the current existence theory associated with the ECE. This also falls along the lines of recent work done on related problems in the context of bounded geometry, such as the Yamabe problem \cite{YamabeEqBG}, and, more generally, we refer the reader to \cite{BoundaryValueRegBG,LaplacianBG,https://doi.org/10.1002/mana.201300007,Shubin,Shubin2,GreeneBG} and references therein for a review of analytic developments in this context. Bounded geometry  also seems to be a reasonable physical hypothesis, further supported by the existence of analytic tools which would allow one to evolve such initial data sets (see, for instance, \cite[Theorem 4.14, Appendix III]{MR2473363}).

Let us furthermore notice that (\ref{MeanCurvHypot}) can be understood as a mean curvature restriction, given $\gamma$ and the energy-momentum sources. Since $\tau\in W^{1,p}_{loc}$ with $p>n$, then $\tau\in C^0_{loc}$, and since $d\tau\in L^2(M,dV_{\gamma})\cap L^p(M,dV_{\gamma})$, then, in an appropriate $L^p$-sense, $d\tau\rightarrow 0$ at infinity. Therefore, assuming $M$ has only one end to simplify the discussion, if we regard $\tau_0$ as an asymptotic value of $\tau$, we suggest decomposing $\tau=\tau_0+\tilde{\tau}$, with $\tilde{\tau}\in W^{1,p}(M,dV_{\gamma})$ being the freely prescribed conformal datum, and leaving $\tau_0$ to be chosen so as to satisfy (\ref{MeanCurvHypot}). In such a case, (\ref{MeanCurvHypot}) fixes a minimum possible value for  $\tau_0$, given the conformal data $(\gamma,\tilde{\tau},  U, \epsilon_1,\epsilon_2,\epsilon_3,\omega_1,\omega_2)$. This slight modification in the conformal decomposition allows for far-from-CMC solutions to (\ref{EMSystem-CompleteManifolds}) with large conformal data $(\gamma,\tilde{\tau},  U, \epsilon_1,\epsilon_2,\epsilon_3,\omega_1,\omega_2)$. To the best of our knowledge, this is a novelty, since even on closed manifolds far-from-CMC initial data typically force a trade-off with smallness conditions on the remaining coefficients of the system (see, for instance, \cite{Holst1,MaxwellFarCMC,Gicquaud2} and a quantitative explicit version of this can be consulted in \cite[Theorem 1]{Gicquaud_2025}).

The above comments about the flexibility of (\ref{MeanCurvHypot}) align well with the objectives set up prior to Theorem \ref{ExistenceCriteriaThm}: obtaining initial data sets without the need of strong decaying conditions. In contrast, if one wanted initial data sets with $\tau\rightarrow 0$ at infinity (as one typically does for isolated gravitational systems), then (\ref{MeanCurvHypot}) would become very restrictive,  demanding $d\tau=0$ and decoupling the system. Nevertheless, for such situations it would be much better to actually start with an AE model for infinity and obtain far-from-CMC constructions along the lines of papers such as \cite{DiltsIsenberMazzeoAE,HolstAE,arxivrodrigo}. Therefore, we regard Theorem \ref{ExistenceNonVacuum} (as well as Theorem \ref{ExistenceVacuum} below) as complementing these references. They also appear as physically complementary, where the cited references deal with isolated physical systems while we are  motivated by open cosmological scenarios.



With the above comments in mind,  (\ref{EnergyHypot}) could be regarded as the most restrictive condition from a mathematical standpoint, since it avoids the special case of vacuum. This motivates the following result, which accounts for vacuum initial data, where the construction of the subsolution is modified by appealing to general results on Yamabe-type equations by \cite{MR2962687}.

\medskip
\begin{thmx}\label{ExistenceVacuum}
Let $(M^n, \gamma)$ be a smooth Riemannian manifold of bounded geometry,  $n\ge3$, $p>n$ and assume (\ref{FunctinalHypot}) holds. Let $r=d_{\gamma}(p,\cdot):M\rightarrow\mathbb{R}$ denote the distance function to a given point $p\in M$. Letting $H,A,B$ be real numbers, $A,B>0$, we make the following additional assumptions:
\begin{equation}\label{YamabeCondtions}
 \begin{aligned}
& \mathrm{Ric}_{\gamma} \ge -(n-1) H^2(1+r^2),\\
& R_\gamma \ge -A, \\
& \vert \tau \vert \ge B >0 \text{ outside a compact set},
& \lambda_1^{-a_n\Delta_\gamma +  R_\gamma}(B_0)>0 \\
& \lambda_1^{-a_n\Delta_\gamma +  R_\gamma}(M)<0, \\
\end{aligned} 
\end{equation}
where $B_0= \{ x \in M \, : \, \tau(x)=0 \}.$ Then, there exists $C(n,M, \gamma, \lambda_{1, \mathrm{conf}})$  such that if (\ref{MeanCurvHypot}) holds then  \eqref{EMSystem-CompleteManifolds} admits a $W^{2,p}_{\mathrm{loc}}$ solution $(\phi,X)$.
\end{thmx}

The above theorem allows for vacuum initial data sets and complements our previous result. Following \cite[Chapter 6]{MR2962687}, given a non-empty open set $\Omega\subset M$, let us recall the definition
\begin{align}
\label{definitionvapsurOmega}
\lambda_1^{-a_n\Delta_{\gamma}+R_{\gamma}}(\Omega)=\Big\{ \inf_{\varphi\in W^{1,2}_0(\Omega)}\int_{\Omega}\left(a_n\vert \nabla \varphi \vert^2_{\gamma}+R_{\gamma}\varphi^2 \right)dV_{\gamma} \: :\: \int_{\Omega}\varphi^2dV_{\gamma}=1\Big\}.
\end{align}
In Theorem \ref{ExistenceVacuum}, we use the hypothesis $\lambda_1^{-a_n\Delta_\gamma + R_\gamma}(B_0)>0$ as a basis to employ results on Yamabe-type equations available in \cite[Chapter 6]{MR2962687}. As discussed in \cite[Section V]{DiltsIsenberg16}, this condition is a recurrent feature appearing also in \cite{maxwell2015yamabeclassificationprescribedscalar,gicquaud2019prescribednonpositivescalar,Gicquaud_2022}.

In the proof of Theorem \ref{ExistenceCriteriaThm} we will need to produce a sequence of solutions along an exhaustion of $M$ by compact subsets, under appropriate boundary conditions. Although the final existence of a solution for \eqref{EMSystem-CompleteManifolds} on the non-compact manifold $M$ is highly sensitive to the specific couplings and non-linearities of the system, the associated boundary value problems on compact sets are far less so. In fact, existence criteria analogous to those of Theorem \ref{ExistenceCriteriaThm} can be presented on a compact manifold with boundary for a far broader class of systems generalising \eqref{EMSystem-CompleteManifolds} and including examples such as charged fluids (see Example \ref{020320211748}). These being interesting systems on their own, in Section \ref{sectionPDEgen} (Theorem \ref{theoexistcompact}) we will broaden the scope of our analysis to present existence criteria for such a larger class of systems on compact manifolds with boundary. 

Although the general approach for Theorems \ref{ExistenceCriteriaThm} and \ref{theoexistcompact} is conceptually similar (both barrier-based existence results), the technicalities surrounding Theorem \ref{ExistenceCriteriaThm} are far more subtle. Among other things, in order to account for coefficients with asymptotics as flexible as possible, we appeal to slightly different notions of barriers than those of Theorem \ref{theoexistcompact}. In the latter case, we introduce \emph{strong} global barriers, tailored for the proof of this theorem and also to the kind of barrier constructions for systems like \eqref{EMSystem-CompleteManifolds} and \eqref{fullsystem} that one could adapt for instance from \cite{Holst1}. These differences will become apparent as we move from Section \ref{sectionPDEgen} to Section \ref{SectionEinsteinSyst}, where key ideas already showcased in Section \ref{sectionPDEgen} are adapted and  refined.

Taking all of the above into account, this paper will be organised as follows. First, in Section \ref{sectionPDEgen} we will analyse Einstein-type elliptic systems on compact manifolds with boundary. In Section \ref{SectionEinsteinSyst} we will prove Theorem \ref{ExistenceCriteriaThm}. Then, in Section \ref{section210720211556} we will proceed to construct the appropriate barrier functions for the application of Theorem \ref{ExistenceCriteriaThm} and prove Theorems \ref{ExistenceNonVacuum} and \ref{ExistenceVacuum}. Finally, in Appendix \ref{appendixlinearanalysisellipticboundaryvalueproblem} we review some properties of linear operators involved in our analysis on compact manifolds with boundary, while in Appendix \ref{BoundedGeometry} we present several useful definitions and results associated with manifolds of bounded geometry and linear elliptic operators one can define on them.

\section{Einstein-type systems on a compact manifold with boundary}\label{sectionPDEgen}


\medskip
{\noindent \textbf{Remark concerning the analytic setting}}: Below, the associated PDE problems will be posed for equations with coefficients of different degrees of Sobolev regularity $W^{k,p}(M)$. In the case of the trace-values of solutions on the boundary $\Sigma$, the associated interpolating spaces $W^{k-1/p,p}(\Sigma)$ are given in terms of the trace spaces described in \cite{Adams}[Chapters V, VII]. The analysis of the linear operators associated with (1.5) on such compact manifolds with boundary is reviewed in Appendix \ref{appendixlinearanalysisellipticboundaryvalueproblem}.

\subsection{Einstein-type systems}
Let us consider $(M^n,  \gamma)$  a compact Riemannian manifold with smooth boundary $\partial M \doteq \Sigma $ and $\gamma \in W^{2,p}(M)$, $p>n$. On this manifold, let us consider an equation in $M$ of unknown $\Psi = \left( \phi, Y \right) \in \Gamma(E)$, where $E$ is a  vector bundle  of shape $\left(M\times \mathbb{R}\right)\oplus \left[\oplus_{j=1}^l T^{k_j}_{r_j} M \right]$. This equation has the following form:

\begin{equation}
\label{300320201529}
\left\{
\begin{aligned}
\Delta_\gamma \phi &= \sum_I a^0_I(Y) \phi^I  \text{ in } M\\
L^i (Y^i) &= \sum_J a^i_J(Y) \phi^J \text{ } \forall i \in 1 \dots l \text{ in } M \\
\phi &= u \text{ on } \Sigma \\
Y^i &= v^i \text{ on } \Sigma.
\end{aligned}
\right.
\end{equation}
where
\begin{itemize}
\item
$\left(L^i, B^i \right)$ are linear elliptic operators with Dirichlet conditions, acting as maps $W^{2,p}(M) \rightarrow L^p(M) \times W^{2-\frac{1}{p},p}\left(\Sigma\right)$, $Y^i\mapsto (L^iY^i,B^iY^i)$, where $B^i:W^{2,p}(M)\to W^{2-\frac{1}{p},p}\left(\Sigma\right)$ denotes the trace map associated with sections of the tensor bundle $T^{k_j}_{r_j} M $. We assume they are invertible for $p>n$,
\item
$\left( a^\alpha_T  \right)^{\alpha=0 \dots, l}_{ T = I, J} :\, W^{2,p} \left( M \right) \rightarrow L^p(M)$  are maps which can depend on $Y$ and $\nabla Y$,
\item
$u$, $v^i \in W^{2- \frac{1}{p}, p}(\Sigma)$,
\end{itemize}
with $p>n.$
We further make the following hypotheses on the operators: 
\begin{itemize}
\item
\emph{Boundedness: }
There exists $\rho_0>0$ such that for all $I$,  $\exists f_I \in L^p( M)$ and constants $C^\alpha$ such that denoting $\overline{B}_{\rho_0}$  a closed ball of radius $\rho_0$ centred at the origin:
\begin{equation}
\label{300320201610}
\forall Y^i \in \overline{B}_{\rho_0} \subset W^{2,p}(M),  \left\vert a^0_I (Y)\right\vert \le f_I, \, \|a^\alpha_I(Y)\|_{L^p(M)} \le C^\alpha \text{ for } \alpha = 0 \dots l.
\end{equation}
\item
\emph{Continuity:}  if $Y_k \rightarrow Y$ in $C^1(M)$,
\begin{equation}
\label{300320201614bis}
\forall \alpha = 0 \dots l  \quad \begin{aligned}  a^\alpha_I(Y_k) &\rightarrow a^\alpha_I(Y) \text{ in } L^p(M). 
\end{aligned} 
\end{equation}
\item \emph{Boundedness of the inverse: } In addition to the invertibility of the operators $(L^i,B^i)$, we assume they satisfy the following estimates:
\begin{equation}
\label{ellipticL}
\!\!\!\!\!\!\!\!\left\| Y\right\|_{W^{2,p} \left( M \right) } \le C^i \left( \left\| L^i Y \right\|_{L^p (M)} + \left\| B^i Y \right\|_{W^{2-\frac{1}{p}, p} (\Sigma_1 ) } \right).
\end{equation}
\end{itemize}

\begin{definition}
\label{confeinstsyscomp}
We will call a system of the form \eqref{300320201529} satisfying \eqref{300320201610}, \eqref{300320201614bis} and \eqref{ellipticL} on a compact Riemannian manifold $(M, \gamma)$  a \emph{conformal Einstein-type system}.
\end{definition}

These kind of systems were introduced by the first two authors (see \cite{arxivrodrigo}), and their definition was shaped to encompass and mimic the Lichnerowicz equation with electromagnetic and momentum constraints as described in the following example.
\begin{example}\label{020320211748}
Let us consider the Einstein-Maxwell constraints    associated with the  space-time field equations with sources modelling a charged fluid. They are obtained by coupling \eqref{ECE} with the system of constraints induced by a space-time electromagnetic $2$-form $\bar{F}$, given by  
\begin{align}\label{constraints}
\begin{split}
&\mathrm{div}_gE=\tilde{q}\doteq q\bar{N}u^0\vert_{t=0},\\
&dF=0,
\end{split}
\end{align}
where, in addition to the conventions used in  introduction,  $E\doteq \bar{F}(\cdot,\mathcal{N})\vert_{t=0}$ is the $1$-form  associated with the electric-field  of the charged fluid, 
$\bar{N}$ represents the lapse function associated with the orthogonal splitting of the Lorentzian metric, $u$ is the velocity field  of the charged fluid and $F\in \Omega^2(M)$ is the induced $2$-form that arises from restricting  $\bar{F}$ to  vectors tangent to $M$. Since the magnetic constraint $dF=0$ is always decoupled from the rest of the system, we consider that such a choice of a closed $2$-form $F$ has been made.

Using the conformal method on \eqref{ECE}, setting $E = \phi^{-\frac{2n}{n-2}} \tilde E$ and $\tilde{E}=\nabla f + \vartheta\in \Gamma(T^{*}M)$,  $\mathcal{\vartheta}$ a datum and $f$ an unknown, one can transform the above constraints into the following system posed for $(\phi,X,f)\in \Gamma(\mathcal{M})$, where $\mathcal{M}=\left(M\times \mathbb{R}\right)\oplus TM\oplus \left(M\times\mathbb{R}\right)$ (see \cite[(9)-(10)]{arxivrodrigo}).
\begin{equation}\label{fullsystem}
\!\!\!\!\!\!\!\!\!\!\!\!\left\{
\begin{aligned}
&\Delta_{\gamma}\phi - c_n R_{\gamma}\phi + c_n \vert \tilde{K} \vert^2_{\gamma}\phi^{-\frac{3n-2}{n-2}} + c_n\left(2\epsilon_1  - \frac{n-1}{n}\tau^2\right)\phi^{\frac{n+2}{n-2}} + 2c_n\epsilon_2\phi^{-3} + 2c_n\epsilon_3\phi^{\frac{n-6}{n-2}} =0,\\
&\Delta_{\gamma, \mathrm{conf}}X  -  \frac{n-1}{n}\nabla\tau \phi^{\frac{2n}{n-2}} - \omega_1\phi^{2\frac{n + 1}{n-2}} + \omega_2 = 0,\\
&\Delta_{\gamma}f=\tilde{q}\phi^{\frac{2n}{n-2}}
\end{aligned}
\right.
\end{equation}
where, given the mass density of the fluid $\mu$:
\begin{align*}
\epsilon_1=\mu\left( 1 + \vert \tilde{u} \vert^2_{\gamma} \right) \; , \; &\epsilon_2=\frac{1}{2}\vert\tilde{E}\vert^2_{\gamma} \; , \; \epsilon_3=\frac{1}{4}\vert\tilde{F}\vert^2_{\gamma},\\
{\omega_1}_k=\mu\left( 1 + \vert\tilde{u}\vert^2_{\gamma} \right)^{\frac{1}{2}}\tilde{u}_k \; , \; &{\omega_2}_{k}=F_{ik}\tilde{E}^i \; ,\; \tilde{q}=q(1+\vert\tilde{u}\vert^2_{\gamma})^{\frac{1}{2}}.
\end{align*}
The  third equation  is called the \emph{electric constraint} and $f$ is the electric potential. Then, prescribing Dirichlet boundary values for $(\phi,X,f)$ along $\Sigma$ yields a system
of the form of (\ref{300320201529}) and, once we declare the free parameters in appropriate function spaces, one can show show that the above system is of conformal Einstein type, since elliptic estimates follow from (\ref{MaxwellsEstimate.1})-(\ref{MaxwellsEstimate.2}), and the estimates (\ref{ellipticL}) follow again from Theorem \ref{CKLIsoThm}. We refer the reader to \cite[Lemma 1]{arxivrodrigo} for the remaining details. System \eqref{EMSystem-CompleteManifolds}, investigated in Section \ref{SectionEinsteinSyst}, can be seen as the $q=0$ case.
\end{example}

\subsection{$W^{2,p}$ existence}
We will use a fixed point scheme relying on \emph{strong global  barrier functions}:

\begin{definition}
\label{strigglobalbarrierscomp}
A conformal Einstein-type system on a compact Riemannian manifold $(M, \gamma)$ admits  strong global barrier functions $\phi_-$ and $\phi_+$ if there exist two real numbers $l$ and $m$ such that $0<l \le \phi_- \le \phi_+ \le m$, and  real numbers $\left( M_{Y^i} \right)_{i=1 \dots r}$, $M_{Y^i} \le \rho_0$, such that,  if we denote $\overline{B}_{M_{Y^i}}$ the closed ball of radii $M_{Y^i}$ in $W^{2,p}(M)$ and $\overline{B}_{M_Y} = \times_i \overline{B}_{M_{Y^i}}$, then the following three conditions are satisfied:

\begin{equation} \label{invariantcondition} \forall{ \phi_- \le \varphi \le \phi_+}, \, Z \in \overline{B}_{M_Y}, \,  Y \in W^{2,p}(M),  
\left\{
\begin{aligned}
L^i (Y^i) &= \sum_J a^i_J(Z) \varphi^J \text{ in } M \\
Y^i &= v^i \text{ on } \Sigma
\end{aligned}
\right. 
\implies Y \in \overline{B}_{M_{Y}},
\end{equation}
\begin{equation}
\label{300320201549}
\forall{Y \in \overline{B}_{M_{Y}}}\  \quad \left\{ \begin{aligned} \Delta_\gamma \phi_- &\ge \sum_I a^0_I(Y) \phi_-^I  \text{ in } M \\ 
\phi_- &\le u \text{ on } \Sigma, \end{aligned} \right.
\end{equation}
and
\begin{equation}
\label{300320201549}
\forall {Y \in \overline{B}_{M_{Y}}}  \quad \left\{ \begin{aligned} \Delta_\gamma \phi_+ &\le \sum_I a^0_I(Y) \phi_+^I  \text{ in } M\\ 
\phi_+ &\ge u \text{ on } \Sigma. \end{aligned} \right.
\end{equation}
\end{definition}

To simplify notations, we will denote $\mathcal{P}$ the left-hand part of \eqref{300320201529}  and  $F(\phi, Y) \doteq \left(h_Y(\phi), h^i_Y(\phi), u, v^i \right)$ its right-hand part, containing all the non-linearities. In the same manner, given  $a \in L^p (M)$, we will denote $\mathcal{P}_{a}$ and $F_{a}$ the shifted operators:
\begin{equation}
\label{300320201647}
\begin{aligned}
h^a_Y ( \phi) &= h_Y(\phi) -a \phi =  \sum_I a^0_I (Y) \phi^I  - a \phi \\
\end{aligned}
\end{equation}
\begin{equation}
\label{300320201704}
\mathcal{P}_{a} \,: \, \left\{
 \begin{aligned}
W^{2,p}(M) &\to L^p(M) \times W^{2- \frac{1}{p}, p} (\Sigma) \\
\Psi=(\phi, Y)& \mapsto \left( \Delta_\gamma \phi - a \phi, L^i (Y^i) , B_0\phi , B^i(Y^i)\right),
\end{aligned}
 \right.
\end{equation}
and 
\begin{equation}
\label{300320201720}
\mathrm{F}_{a}\,: \, \left\{
 \begin{aligned}
L^\infty(M) \times W^{2,p}(M) &\rightarrow L^p(M) \times W^{2- \frac{1}{p}, p} (\Sigma) \\
\Psi=(\phi, Y)& \mapsto \left(h^a_Y(\phi) , h^i_Y(\phi),u, v^i \right).
\end{aligned}
 \right.
\end{equation}

By invertibility assumption on the operators $L^i$ and Theorem \ref{ShiftedRoughLapBelIso}, both $\mathcal{P}$ and $\mathcal{P}_{a}$ are invertible with bounded inverses $\mathcal{P}^{-1}$ and $\mathcal{P}^{-1}_{a}$. 

Given a pair of strong global barrier functions  on $M$, in order to apply maximum principles, we will choose $a$ such that $h^a_Y$  are non increasing in $\phi$.
Given any $Y \in \times_i \overline{B}_{M_{Y^i}}$ one has, thanks to \eqref{300320201610}:
\begin{equation}
\label{300320201647}
\begin{aligned}
\left\vert \partial_\phi h_Y(\phi) \right\vert &\le \left\vert\sum_I a^0_I (Y) I \phi^{I-1} \right\vert \le \sum_I \vert I\vert f_I \sup_{l\le \phi \le m} \phi^{I-1}.
\end{aligned}
\end{equation}
Thus, if we set 
\begin{equation}
\label{300320201647}
\begin{aligned}
h^a_Y ( \phi) &= h_Y(\phi) -a \phi 
\end{aligned}
\end{equation}
with 
\begin{equation} \label{aandbforcomp} \begin{aligned}
a \in L^p (M)  \text{ s.t. } &a > \sum_I I f_I \sup_{l\le y \le m } y^{I-1}  
\end{aligned}\end{equation}
one has $\partial_\phi h^a_Y (\phi)
<0$ on $[l,m]$ ($l$ and $m$ being the lower and upper bounds in Definition \ref{strigglobalbarrierscomp}), meaning that $h^a_Y$ 
 is a non-increasing function on $[l,m]$ for all $Y \in \overline{B}_{M_Y}$.

We will show the existence of a solution to the system \eqref{300320201529} using a Schauder fixed point theorem (see \cite[Corollary B.3, Appendix B, Chap. 14]{TaylorBook3}):

\begin{theorem}
\label{fixpointschauder}
Let $B$ be a Banach space and $U\subset B$ a closed and convex subset. Let $\mathcal{F}: \, U \rightarrow U$ be a continuous map such that $\overline{\mathcal{F}(U)}$ is compact. Then $\mathcal{F}$ has a fixed point.
\end{theorem}

From this, one can show:

\begin{theorem}
\label{theoexistcompact}
Let $(M^n, \gamma)$ be a compact manifold with boundary, let $p>n$ and consider a conformal Einstein-type system \eqref{300320201529} on $M$. Assume that there exists a pair of strong global barrier functions $0<l\le \phi_- \le \phi_+\le m$. Then the system admits a solution $\psi= (\phi, Y) \in W^{2,p}(M)$ with $\phi >0$.
\end{theorem}

\begin{proof}

Let $\mathcal{U}_0 = \{ \psi = (\phi, Y) \in W^{2,p}(M) \text{ s.t. } \phi_- \le \phi \le \phi_+ , \, Y \in \overline{B}_{M_Y} \}$, and  define, for $a$ as in \eqref{aandbforcomp}, the solution map:
$$\mathcal{F}_{a} :  \left\{
\begin{aligned}
& W^{2,p}(M) \rightarrow W^{2,p}(M) \\
& \psi = \left( \phi, Y \right) \mapsto \mathcal{F}_{a} = \mathcal{P}_{a}^{-1} \circ F_{a} (\psi)
\end{aligned} \right.$$
One may notice that the shift  only happens on the equation in $\phi$. Thus, if  $(\phi, Y) \doteq \mathcal{F}_a \left(\varphi, Z \right)$, $Y$ still satisfies
$$\left\{
\begin{aligned}
L^i (Y^i) &= \sum_J a^i_J(Z) \varphi^J \text{ in } M \\
Y^i &= v^i \text{ on } \Sigma.
\end{aligned} \right.$$
Consequently, if $(\varphi, Z) \in \mathcal{U}_0$, \eqref{invariantcondition} ensures $Y  \in \overline{B}_{M_Y}$.

\medskip
\underline{ {\bf Claim 1:} $\mathcal{F}_a$ is bounded and continuous on $\mathcal{U}_0$ }

Let $(\varphi, Z) \in \mathcal{U}_0$ and let $(\phi, Y ) \doteq \mathcal{F}_a(\varphi, Z)$. Then, $\mathcal{P}_a (\phi, Y)= F_a(\varphi, Z)$ and by invertibility of $(\Delta_\gamma -a, B_0)$ (see Theorem \ref{ShiftedRoughLapBelIso}) there exists a constant $C$ depending on $M, \gamma, \Sigma$ such that:
\begin{equation} \label{phibounded}\begin{aligned}
\| \phi \|_{W^{2,p}(M)} ¨
&\le C \big(\|a\|_{L^p(M)} m +  \sum_I \left\| a^0_I (Z) \right\|_{L^p(M)} \max ( l^I, m^I ) + \| u \|_{W^{2-\frac{1}{p}, p}(\Sigma)} \big) \\
&\le C \big( \|a\|_{L^p(M)} m+ \sum_I C^0 \max ( l^I, m^I ) + \| u \|_{W^{2-\frac{1}{p}, p}(\Sigma)} \big)  \\ 
&\le C_0(M, \gamma, \Sigma, (M_Y^i)_i, l,m),
\end{aligned}\end{equation}
using the boundedness hypotheses \eqref{300320201610} and the fact that $Z \in \overline{B}_{M_Y}$. In addition, considering $(\varphi_1, Z_1), (\varphi_2, Z_2)  \in \mathcal{U}_0$,  $(\phi_1, Y_1 ) \doteq \mathcal{F}_a(\varphi_1, Z_1 ),$  $(\phi_2, Y_2 ) \doteq \mathcal{F}_a(\varphi_2, Z_2)$, by linearity of $\mathcal{P}_a$, $(\phi_1- \phi_2, Y_1- Y_2)$ solve 
$$\left\{ \begin{aligned}
 &\left( \Delta_\gamma - a \right)( \phi_1 - \phi_2) = h_a(\varphi_1) - h_a (\varphi_2) \text{ in } M \\
& L^i (Y^i_1- Y^i_2) = \sum_J a^i_J (Z_1) \varphi_1^J -a^i_J (Z_2) \varphi_2^J  \text{ in } M\\
& \phi_1 - \phi_2= 0  \text{ on } \Sigma \\
& Y^i_1 - Y^i_2 = 0 \text{ on } \Sigma. 
\end{aligned} \right.$$

Once more, by the invertibility of $(\Delta_\gamma -a, B_0)$ there exists a constant $C$ depending on $M, \gamma, \Sigma$ such that:
$$ \label{pseudoLipschiptz1}\begin{aligned}
\| \phi_1 - \phi_2 \|_{W^{2,p}(M)}  
&\le C \big( \|a\|_{L^p(M)} \| \varphi_1 -\varphi_2 \|_{L^\infty(M)} + \sum_I \| a^0_I(Z_1) - a^0_I(Z_2)\|_{L^p(M)} \| \varphi_1^I  \|_{L^\infty(M)}  \\ 
& + \sum_I\|  a^0_I(Z_2)\|_{L^p(M)} \| \varphi_1^I -\varphi_2^I \|_{L^\infty(M)}\big) \\ 
&\le C \big( \|a\|_{L^p(M)} \| \varphi_1 -\varphi_2 \|_{L^\infty(M)} + \sum_I \| a^0_I(Z_1) - a^0_I(Z_2)\|_{L^p(M)} \max ( l^I, m^I )   \\ 
& + C^0 \| \varphi_1^I  -\varphi_2^I \|_{L^\infty(M)}\big),
\end{aligned}$$
using the boundedness hypotheses \eqref{300320201610} and the controls $Z_1, Z_2 \in \overline{B}_{M_Y}$. Now,  since $0<l \le \phi_- \le \varphi_1, \varphi_2\le \phi_+ \le m$, for all $I$ the maps $t\mapsto t^I$ are Lipschitz on $[l,m]$ and there exists a constant $C(M, \gamma, \Sigma, (M_{Y^i})_i,  l, m)$ such that:
\begin{equation} \label{pseudoLipschiptz1}
\begin{aligned}
\| \phi_1 - \phi_2 \|_{W^{2,p}(M)}  & \le C \big( \| \varphi_1 - \varphi_2 \|_{L^{\infty}(M)} +   \sum_I \| a^0_I(Z_1) - a^0_I(Z_2)\|_{L^p(M)}\big).
\end{aligned}
\end{equation}

Similarly  the invertibility of $L^i$ (given by \eqref{ellipticL}) ensures there exists a constant $C$ depending on $M, \gamma, \Sigma$ such that for all $i$: 
$$\begin{aligned}
\| Y_1^i - Y_2^i\|_{W^{2,p}(M)}  
&\le C \big(\sum_J \| a^i_J(Z_1) - a^0_J(Z_2)\|_{L^p(M)} \| \varphi_1^J  \|_{L^\infty(M)} + \|  a^i_J(Z_2)\|_{L^p(M)} \| \varphi_1^J -\varphi_2^J \|_{L^\infty(M)} \big),
\end{aligned}$$
which once again yields the existence of a constant $C(M, \gamma, \Sigma, (M_{Y^i})_i,  l, m)$ such that for all $i$:
\begin{equation} \label{pseudoLipschiptz2}
\begin{aligned}
\| Y_1^i - Y_2^i\|_{W^{2,p}(M)}  &  \le C \big(  \sum_J \| a^i_J(Z_1) - a^i_J(Z_2)\|_{L^p(M)}+ \| \varphi_1  -\varphi_2 \|_{L^\infty(M)} \big).
\end{aligned}
\end{equation}
Given the continuity hypothesis \eqref{300320201614bis}, the estimates \eqref{pseudoLipschiptz1} and \eqref{pseudoLipschiptz2} combine with Sobolev embeddings and ensure that $\mathcal{F}_a$ is continuous on $\mathcal{U}_0$.

Let us define $\tilde C_0 \doteq C_0+ \|\phi_- \|_{W^{2,p}(M)}$ with $C_0$ the final constant in \eqref{phibounded}, and consider $\mathcal{U} \doteq \{ \psi = (\phi, Y) \in W^{2,p}(M) \text{ s.t. } \phi_- \le \phi \le \phi_+ , \, \phi \in \overline{B}_{\tilde C_0}, \, Y \in \overline{B}_{M_Y} \}$.  It is straightforwardly a nonempty (since $(\phi_-, 0) \in \mathcal{U}$) closed convex subset of $W^{2,p}(M)$. 

\bigskip
\underline{ {\bf Claim 2:}  $\mathcal{F}_a(\mathcal{U}) \subset \mathcal{U}$. }

Consider  $(\varphi, Z)\in \mathcal{U}$,  $(\phi, Y ) \doteq \mathcal{F}_a(\varphi, Z)$. As detailed above Claim 1, $Y \in \overline{B}_{M_Y}$ by \eqref{invariantcondition}, while \eqref{phibounded} ensures $\phi \in \overline{B}_{C_0} \subset \overline{B}_{\tilde C_0}$. There just remains to show that $\phi_- \le \phi \le \phi_+$. We can write:
$$ \left\{ \begin{aligned}
& \left( \Delta_\gamma - a \right) \left( \phi - \phi_- \right) \le h^a_{Z} (\varphi) - h^a_{Z} \left( \phi_- \right) \le 0, \text{ in } M\\
& \phi - \phi_-  \ge 0 \text{ on } \Sigma,
\end{aligned}
\right.
$$
since  $ \phi_-$ is a strong global subsolution, $Z \in \overline{B}_{M_Y}$, $\varphi  \ge \phi_-$ by hypothesis, and $h^a_{Z}$ is non-increasing by construction. By the maximum principle given in Proposition \ref{MaxPrinciplesLaplace} one thus has $\phi\ge \phi_-$. Using the same reasoning to compare $\phi $ to the supersolution $\phi_+$ we conclude that $\phi  \le \phi_+$, which proves the claim.

\bigskip
\underline{ {\bf Claim 3:} $\mathcal{F}_a$ is continuous on $\mathcal{U}$ and $\mathcal{F}_a(\mathcal{U})$ is compact. }

Since $\mathcal{F}_a$ is continuous on $\mathcal{U}_0$ and $\mathcal{U}\subset \mathcal{U}_0$, the first part of the claim is clear. To show that $\mathcal{F}_a(\mathcal{U})$ is compact, and since $\mathcal{U} \subset W^{2,p}(M)$ which is metric, it is enough to show sequential compactness. Let us thus consider $\{\varphi_k, Z_k\}_{k=0}^\infty \subset \mathcal{U}$ and $(\phi_k, Y_k) = \mathcal{F}_a (\varphi_k, Z_k)$.  By definition $\varphi_k \in \overline{B}_{\tilde C_0}$ and $Z_k \in \overline{B}_{M_Y}$, the sequence $\{\varphi_k, Z_k\}_{k=0}^\infty $ is thus uniformly bounded in $W^{2,p}(M)$. By weak sequential compactness, and the compactness of the Sobolev embeddings in $W^{2,p} \hookrightarrow C^1, L^\infty$ for $p>n$,
 there exists an extraction $f :\, \mathbb{N} \rightarrow \mathbb{N}$ and $(\varphi,Z) \in \overline{B}_{\tilde C_0} \times \overline{B}_{M_Y}$ such that $(\varphi_{f(k)}, Z_{f(k)} ) \rightharpoonup (\varphi, Z)$ weakly in $W^{2,p}(M)$, $\varphi_{f(k)} \rightarrow \varphi$ strongly in $L^\infty(M)$ and $Z_{f(k)} \rightarrow Z$ strongly in $C^1(M)$.
By continuity hypotheses \eqref{300320201614bis}, for all $\alpha$ and $ I$ we have $a^{\alpha}_I(Z_{f(k)}) \rightarrow a^{\alpha}_I(Z)$ in $L^p(M)$. In addition,  $L^\infty$ convergence yields $\phi_- \le \varphi \le \phi_+$, meaning that $(\varphi, Z) \in \mathcal{U }$. Then $(\phi, Y) = \mathcal{F}_a(\varphi, Z) \in \mathcal{U}$ by Claim 2. Applying \eqref{pseudoLipschiptz1} and \eqref{pseudoLipschiptz2} with $\varphi_1 = \varphi_{f(k)}$, $\varphi_2 = \varphi$, $Z_1 = Z_{f(k)}$, $Z_2 = Z$ for $k \in \mathbb{N}$ yields:
$$\begin{aligned}
\| \phi_{f(k)} - \phi \|_{W^{2,p}(M)}  + \| Y_{f(k)} - Y\|_{W^{2,p}(M)} &\le C\big( \sum_\alpha \sum_I \| a^\alpha_I(Z_{f(k)}) - a^\alpha_I(Z)\|_{L^p(M)}  + \| \varphi_{f(k)}  -\varphi \|_{L^\infty(M)} \big),
\end{aligned}$$
which shows that the sequence $\{\phi_{f(k)}, Y_{f(k)}\}_{k=0}^\infty $  converges in $W^{2,p}(M)$ towards $\mathcal{F}_a(\varphi, Z) \in \mathcal{U}$. 

\medskip
One can then apply Theorem \ref{fixpointschauder} to show that $\mathcal{F}_a$ has a fixed point in $ \mathcal{U} \subset W^{2,p}(M)$ and thus that \eqref{300320201529} has a $W^{2,p}(M)$ solution with $\phi\ge \phi_->0$.
\end{proof}

\section{Einstein-type systems on a complete manifold}\label{SectionEinsteinSyst}

In this section, we analyse existence results for an Einstein-type system on a complete manifold $(M^n, \gamma)$. Since this analysis will be sensitive to  the types of coupling and non-linearities in the system, we will work with \eqref{EMSystem-CompleteManifolds}, recalled here:
\begin{equation}
\label{systoncomplphiX}
\left\{
\begin{aligned}
&a_n\Delta_\gamma \phi -  R_\gamma \phi + \left\vert \tilde K(X) \right\vert^2_\gamma \phi^{-\frac{3n-2}{n-2} } -  \frac{n-1}{n} \tau^2 \phi^{\frac{n+2}{n-2} }+ 
2 \epsilon_1  \phi^{\frac{n+2}{n-2}} + 2  \epsilon_2 \phi^{-3} + 2  \epsilon_3 \phi^{\frac{n-6}{n-2} } =0, \\
& \Delta_{\gamma, \mathrm{conf}} X - \frac{n-1}{n} \nabla \tau \phi^{\frac{2n}{n-2} } - \omega_1\phi^{2\frac{n + 1}{n-2}} + \omega_2  =0,
\end{aligned}
\right.
\end{equation}
where given $p>n$ we assume, 
\begin{equation} \label{hypcomplLp}
\gamma \in C^{\infty}(M), \, R_\gamma, \epsilon_1, \epsilon_2, \epsilon_3,\vert U \vert^2,  \tau^2 \in L^p_{\mathrm{loc}}(M) \text{ and } \omega_1, \omega_2, \nabla \tau \in L^{2}(M)\cap L^p_{\mathrm{loc}}(M). 
\end{equation}

\begin{definition}\label{GlobalBarriersDefn}
We say that the system \eqref{systoncomplphiX} on a complete manifold $(M, \gamma)$ admits \emph{global barrier functions} on $M$ if there exists $m \in \mathbb{R}$ and $0<\phi_- \le \phi_+ \le m < + \infty$ such that:
$$\begin{aligned} &\forall \phi_- \le \varphi \le \phi_+, \quad \forall Y \in L^2(M)\cap W^{2,p}_{\mathrm{loc}} (M) \text{ s.t. } \Delta_{\gamma, \mathrm{conf} } Y - \frac{n-1}{n} \nabla \tau \varphi^{\frac{2n}{n-2}}  - \omega_1\varphi^{2\frac{n + 1}{n-2}} + \omega_2  =0 \text{ then } \\
&a_n \Delta_\gamma \phi_- \ge   R_\gamma \phi_- - \left\vert \tilde K(Y) \right\vert^2_\gamma \phi_-^{-\frac{3n-2}{n-2} } +  \frac{n-1}{n} \tau^2 \phi_-^{\frac{n+2}{n-2} }- 
2  \epsilon_1  \phi^{\frac{n+2}{n-2}}_- - 2 \epsilon_2 \phi_-^{-3} - 2  \epsilon_3 \phi_-^{\frac{n-6}{n-2} }   \text{ in } M \\
&a_n \Delta_\gamma \phi_+ \le  R_\gamma \phi_+ - \left\vert \tilde K(Y) \right\vert^2_\gamma \phi_+^{-\frac{3n-2}{n-2} } +  \frac{n-1}{n} \tau^2 \phi_+^{\frac{n+2}{n-2} }-
2  \epsilon_1  \phi^{\frac{n+2}{n-2}}_- - 2 \epsilon_2 \phi_+^{-3} - 2  \epsilon_3 \phi_+^{\frac{n-6}{n-2} }   \text{ in } M
\end{aligned}$$
\end{definition}

Below, we will rely on an estimate of the first eigenvalue for the CKL in order to produce solutions to the momentum constraint on $M$. To that end, we introduce the following definition:

\begin{definition}
Let $(M^n, \gamma)$  be a complete manifold. The first eigenvalue of the conformal Killing Laplacian is defined as:
$$\lambda_{1, \gamma, \mathrm{conf} } \doteq \inf_{u \in C^{\infty}_c(M, TM)\backslash \{ 0 \} }  \frac{ \int_M  \left\vert \pounds_{\gamma, \mathrm{conf} } u \right\vert^2 dV_\gamma }{ \int_M \vert u\vert^2 dV_\gamma }.$$
\end{definition}

\begin{lem}\label{Xagainstf}
Assume that on a complete manifold $(M^n, \gamma)$ the first eigenvalue of the conformal Killing Laplacian satisfies $$\lambda_{1, \gamma, \mathrm{conf} }>0.$$  Then for  any $f\in L^p_{\mathrm{loc}}(M) \cap L^2(M)$, there exists an $ L^2(M) \cap W^{2,p}_{\mathrm{loc}}(M)$ solution to the equation 
\begin{equation}\label{deltagammaconff} \Delta_{\gamma, \mathrm{conf}} X = f \end{equation} 
satisfying, for any pair of compact sets with smooth boundary, $K \subset \subset K' \subset \subset  M$:
\begin{equation}\label{XSol-UniformEstimatef}
\begin{aligned}
\left\| X \right\|_{W^{2,p} (K) } &\le \frac{C(n, \gamma, K, K' )}{\lambda_{1, \gamma, \mathrm{conf}} }\left( \left\| f \right\|_{L^2 (M)}  + \| f \|_{L^{p}(K')} \right)
\end{aligned}
\end{equation}
and 
\begin{equation}\label{XSol-UniformEstimatefL2}
\begin{aligned}
\left\| X \right\|_{L^2(M) } &\le \frac{2}{\lambda_{1, \gamma, \mathrm{conf}} } \left\| f \right\|_{L^2 (M)}. 
\end{aligned}
\end{equation}
\end{lem}

\begin{proof}
\underline{\textbf{Step 1: } Existence of  a compact exhaustion }\newline
Let us fix  $(\Omega_k)$ a compact exhaustion of $M$ with smooth boundaries, and compact subdivisions  $(U_k)$ and $(V_k)$ such that:
$$\Omega_{k-1} \subset \subset U_k \subset \subset V_k \subset \subset \Omega_k.$$

Such a compact exhaustion follows from the metric completeness  of the manifold and a classical regularization procedure to ensure the smoothness of the boundary (see \cite[Remark 3.2]{AlbanseRigoli.2}).

\bigskip
\underline{\textbf{Step 2: } Solution and estimate on $\Omega_k$ }\newline
Fix $k \in \mathbb{N}$. Theorem \ref{CKLIsoThm} ensures the operator $(\Delta_{\gamma, \mathrm{conf}}, B_0) :\, W^{2,p}(\Omega_k) \rightarrow L^p(\Omega_k) \times W^{2-\frac{1}{p}, p}(\partial \Omega_k)$ is an isomorphism, and thus that there exists $X_k \in W^{2,p}(\Omega_k)$ such that
$$
\left\{
\begin{aligned}
 \Delta_{\gamma, \mathrm{conf}}X_k &= f \text{ in } \Omega_k \\
 X_k &= 0 \text{ on } \partial \Omega_k.
\end{aligned}
\right.
$$

Using interior elliptic estimates on $U_s \subset V_s$ for $s<k$ then yields:
\begin{equation}
\label{050720211559}
\| X_k \|_{W^{2,2}(U_s)}  \le C(U_s, V_s, \gamma) \left( \left\| f \right\|_{L^2 (V_s)} + \| X_k \|_{L^2(V_s)} \right).
\end{equation}

In addition, given the boundary condition $X_k = 0$, we can estimate:
\begin{equation} \label{050720211549}
\begin{aligned}
\left\| \pounds_{\gamma, \mathrm{conf} } X_k \right\|_{L^2(\Omega_k)}^2 &= -2 \int_{ \Omega_k}  \left\langle X_k , \Delta_{\gamma, \mathrm{conf}} X_k \right\rangle d V_\gamma \\
&=  2 \left\vert \int_{ \Omega_k}  \left\langle X_k , f \right\rangle d V_\gamma \right\vert \\
&\le 2 \|X_k \|_{L^2(\Omega_k)} \left\| f \right\|_{L^2(\Omega_k)}.
\end{aligned}
\end{equation}
Further, since  $X_k \in W^{1,2}_0(\Omega_k)$ we can approach $X_k$ in the $W^{1,2}$ topology by compactly supported smooth vector fields and deduce, if $X_k \neq 0$:
$$\lambda_{1, \gamma, \mathrm{conf} } \le \frac{ \left\| \pounds_{\gamma, \mathrm{conf} } X_k \right\|_{L^2(\Omega_k)}^2}{\left\|  X_k \right\|_{L^2(\Omega_k)}^2 }.$$ 
Using the hypothesis $\lambda_{1, \gamma, \mathrm{conf} }>0$, we rephrase the previous inequality as:
\begin{equation}
\label{050720211551}
\left\|  X_k \right\|_{L^2(\Omega_k)}^2 \le \frac{1}{ \lambda_{1, \gamma, \mathrm{conf} } }  \left\| \pounds_{\gamma, \mathrm{conf} } X_k \right\|_{L^2(\Omega_k)}^2.
\end{equation}
In addition \eqref{050720211551} stands if $X_k=0$ and is thus true for all $X_k$.

Injecting  \eqref{050720211549} into  \eqref{050720211551} then yields:
$$\left\|  X_k \right\|_{L^2(\Omega_k)}^2 \le \frac{2}{ \lambda_{1, \gamma, \mathrm{conf} } }   \|X_k \|_{L^2(\Omega_k)} \left\| f \right\|_{L^2(\Omega_k)},$$
which ensures the uniform estimate on $V_s$ for $s<k$:
\begin{equation}
\label{050720211554}
\| X_k \|_{L^2(V_s) } \le \| X_k \|_{L^2(\Omega_k)} \le  \frac{2}{ \lambda_{1, \gamma, \mathrm{conf} } }    \left\| f \right\|_{L^2(M)}.
\end{equation}
Inserting \eqref{050720211554} into \eqref{050720211559} then yields:
\begin{equation}
\label{050720211611}
\begin{aligned}
\| X_k \|_{W^{2,2}(U_s)}  &\le  \frac{C(n, U_s, V_s, \gamma)}{ \lambda_{1, \gamma, \mathrm{conf} } }  \left( \left\| f \right\|_{L^2(V_s)} +    \left\| f \right\|_{L^2(M)} \right) \\& \le \frac{C(n, U_s, V_s, \gamma)}{ \lambda_{1, \gamma, \mathrm{conf} } }  \left( \left\| f \right\|_{L^p(V_s)} +    \left\| f \right\|_{L^2(M)} \right)
\end{aligned}
\end{equation}
since $p\ge 2$ and $V_s$ is compact.

\bigskip
\underline{\textbf{Step 3: } Bootstrap from \eqref{050720211611} }\newline
We intend to bootstrap \eqref{050720211611} to obtain, for any $s<k$:
$$\begin{aligned}
\left\| X_k \right\|_{W^{2,p} (\Omega_{s-1}) } &\le \frac{C(n, \gamma, \Omega_{s-1}, \Omega_s)}{\lambda_{1, \gamma, \mathrm{conf}} }   \left[   \|f\|_{L^p(\Omega_s)} + \|f \|_{L^2(M)}  \right].
\end{aligned}$$
Let us denote $\{p_j\}_{j=0}^\infty $ the sequence defined by induction as:
$p_0=2, \, \frac{1}{p_{j+1}} = \frac{1}{p_j} - \frac{2}{n}$, and let $j_{\mathrm{max}}$ be the first integer for which $p_{j_{\mathrm{max}}} = \min \left(p, \frac{n p_{j_{\mathrm{max}}-1}}{n-2 p_{j_{\mathrm{max}}-1}} \right) \ge \frac{n}{2}$ if $n>3$, and $0$ if $n=3$.  In the latter case, since $\Omega_{s-1} \subset U_s$ \eqref{050720211611} is already the desired estimate, and else if $j_{\mathrm{max}} >0$,  consider  compacts  $ \Omega_{s-1}  \subset \subset U^{j_{\mathrm{max}}-1}_s \subset \subset \dots \subset \subset U_s =U^0_s$. We will show by induction that for all $i \le j_{\mathrm{max}}$, one has
\begin{equation} \label{050720211612} \begin{aligned} \left\| X_k \right\|_{W^{2,p_i} (U^i_s) } &\le \frac{C(n, \gamma, U^1_s, \dots U^i_s, V_s )}{\lambda_{1, \gamma, \mathrm{conf}} } \left[   \|f\|_{L^p(V_s)} + \|f\|_{L^2(M)} \right]. \end{aligned} \end{equation}
Inequality \eqref{050720211611} shows that \eqref{050720211612} stands for $i=0$. If we assume that it stands for $i< j_{\mathrm{max}}$, then applying Sobolev estimates on the \emph{compact} set $U^i_s$ yields:
\begin{equation}
\label{050720211626}
\begin{aligned}
\left\| X_k \right\|_{L^{p_{i+1}}( U^i_s)} &\le C(U^i_s, n) \left\| X_k \right\|_{W^{2,p_i} (U^i_s) }  \\
&\le \frac{C(n, \gamma, U^1_s, \dots U^i_s, V_s )}{\lambda_{1, \gamma, \mathrm{conf}} }  \left[   \|f\|_{L^p(V_s)} + \|f \|_{L^2(M)}  \right],
\end{aligned}
\end{equation}
while $L^p$ interior estimates yield:
\begin{equation}
\label{050720211626bis}
\begin{aligned}
\left\| X_k \right\|_{W^{2,p_{i+1}} (U^{i+1}_s) }& \le  {C(n, \gamma, U^1_s, \dots U^{i+1}_s )}\left[ \left\|f \right\|_{L^{p_{i+1} }(U^i_s)} + \left\| X_k\right\|_{L^{p_{i+1}} (U^i_s) }  \right] \\
&\le  {C(n, \gamma, U^1_s, \dots U^{i+1}_s )}  \left[   \|f\|_{L^p(V_s)} + \|f \|_{L^2(M)}  \right],
\end{aligned}
\end{equation}
which yields \eqref{050720211612} for $i+1$ and proves the result for $p_{j_{\mathrm{max}}}$ when $j_{\mathrm{max}}>0$, and since the estimate derives naturally from \eqref{050720211611} when $j_{\mathrm{max}}=0$, it stands true in all cases. Once \eqref{050720211612} is known for $j_{\mathrm{max}}$,  Sobolev embeddings ensure that $X_k \in  L^p (U^{j_{\mathrm{max}}}_s)$, and interior estimates again show that $X_k$ satisfies:
\begin{equation}
\begin{aligned}
\left\| X_k \right\|_{W^{2,p} (\Omega_{s-1}) } &\le \frac{C(n, \gamma, \Omega_{s-1}, U^1_s, \dots, U^{j_{\mathrm{max}}}_s )}{\lambda_{1, \gamma, \mathrm{conf}} }   \left[   \|f\|_{L^p(V_s)} + \|f \|_{L^2(M)}    \right].
\end{aligned}
\end{equation}
Since we chose the sequence of $\left( U^{j}_s\right)$, we can fix them once and for all for any $\Omega_{s-1}$ and deduce:
\begin{equation}
\label{050720211639}
\begin{aligned}
\left\| X_k \right\|_{W^{2,p} (\Omega_{s-1}) } &\le \frac{C(n, \gamma, \Omega_{s-1})}{\lambda_{1, \gamma, \mathrm{conf}} }   \left[   \|f\|_{L^p(\Omega_s)} + \|f \|_{L^2(M)}  \right].
\end{aligned}
\end{equation}
Notably, estimate \eqref{050720211639} does not depend on $k>s$.

\bigskip
\underline{\textbf{Step 4: } Diagonal extraction}\newline
Thanks to Steps 1,2 and 3, we can produce a sequence of solutions $\{X_k\}_{k=1}^{\infty}$ on the compact exhaustion $\{\Omega_k\}_{k=1}^{\infty}$ of $M$, uniformly bounded in $W^{2,p}$ on interior compacts. We can then extract from $\{X_k\}_{k=1}^{\infty}$ a subsequence $\{X_{k_l(\Omega_1)} \}_{l =1}^\infty$ which converges weakly in $W^{2,p}(\Omega_1)$ to some $\bar{X}_1\in W^{2,p}(\Omega_1)$ and strongly in $W^{1,p}(\Omega_1)$ to some $\tilde{X}_1\in W^{1,p}(\Omega_1)$. Using that strong convergence implies weak convergence, one deduces that $\tilde{X}_1=\bar{X}_1 \in W^{2,p}(\Omega_1)$ and,
integrating by parts  on $\Omega_1$, we find
\begin{align}
\label{weakequationafterdiagonalextraction1}
-\int_{\Omega_1}\langle \pounds_{\gamma, \mathrm{conf}}X_{k_l(\Omega_1)},\nabla v\rangle_{\gamma}dV_{\gamma}=\int_{\Omega_1}\langle f,v\rangle_{\gamma}dV_{\gamma} \:\:\:\: \forall\: v\in C^{\infty}_0(\Omega_1).
\end{align}
By strong $W^{1,p}(\Omega_1)$-convergence of the sequence to $\bar{X}_1$, the left-hand side above passes to the limit and then integrating by parts this limit, we find
\begin{align}
\label{weakequationafterdiagonalextraction2}
\int_M\langle \Delta_{\gamma, \mathrm{conf}}\bar{X}_1-f,v\rangle_{\gamma}dV_{\gamma}=0 \: \forall\: v\in C^{\infty}_0(\Omega_1),
\end{align}
implying that $\Delta_{\gamma,\mathrm{conf}}\bar{X}_1-f=0$ in $\Omega_1$, which means that $\bar{X}_1$ satisfies \eqref{deltagammaconff} on $\Omega_1$. In addition $\{X_{k_l(\Omega_1)}\}_{l=1}^{\infty}$ satisfies \eqref{050720211639} for $s<k_l(\Omega_1)$ and thus, by a similar reason to the one given above, we then extract a subsequence $\{X_{k_l(\Omega_2)}\}_{l=1}^{\infty}$ which converges toward $\bar{X}_2$ solving \eqref{deltagammaconff} on $\Omega_2$. Consequently, $\bar{X}_2{\vert_{\Omega_1}}= \bar{X}_1$. Similarly, we build $\{X_{k_l(\Omega_t)}\}_{l=1}^{\infty}$ converging toward $\bar{X}_t$, solving \eqref{deltagammaconff} on $\Omega_t$, and extending the previous solutions. Thus, given $y\in M$ such that $y\in \Omega_t$ for some $t\in \mathbb{N}$, setting $X(y)\doteq \bar{X}_l(y)$ for any $l\geq t$ gives a well-defined $W^{2,p}_{loc}(M)$ function which satisfies \eqref{050720211639} on $M$, and such that the diagonal extraction $X_{k_l(\Omega_l)} \rightharpoonup X$ in $W^{2,p}(\Omega_i)$ for any $i$. Compactness of Sobolev embeddings then ensures that the convergence is strong in $W^{1,p}_{\mathrm{loc}}(M)$ and thus that  $\left\{ X_{k_l(\Omega_l)} \right\}_{l=1}^\infty$ is Cauchy in $W^{1,p}(K)$ for any compact with smooth boundary $K\subset\subset M$.

Moreover, since the right-hand side of the equation does not depend on $X$, for any integers $l$, $s$ and any compacts with smooth boundaries $K \subset \subset K' \subset \subset \Omega_l \cap \Omega_s$, $\Delta_{\gamma, \mathrm{conf}} \left( X_{k_l(\Omega_l)} - X_{k_s(\Omega_s)} \right)=0$ on $K$,  and $L^p$ interior estimates ensure
$$\|X_{k_l(\Omega_l)} - X_{k_s(\Omega_s)} \|_{W^{2,p}(K)} \le C(K, K') \|X_{k_l(\Omega_l)} - X_{k_s(\Omega_s)} \|_{W^{1,p}(K')}.$$
Thus the $X_{k_l(\Omega_l)}$ are $W^{2,p}$-Cauchy on compacts.  One may then pass to the limit in \eqref{050720211639} obtaining the uniform estimate \eqref{XSol-UniformEstimatef} on the $\{\Omega_s\}$. In addition, thanks to \eqref{050720211554} one can check  that for any fixed $s$ 
$$\left\|  X_{k_l(\Omega_s)} \right\|_{L^2(\Omega_s)} \le \frac{2}{ \lambda_{1, \gamma, \mathrm{conf} } }    \left\| f \right\|_{L^2(M)}. $$
Taking this inequality to the limit in $l$ ensures, thanks to the strong $L^2_{\mathrm{loc}}$ convergence, that:
$$\left\|  {X} \right\|_{L^2(\Omega_s)} \le \frac{2}{ \lambda_{1, \gamma, \mathrm{conf} } }    \left\| f \right\|_{L^2(M)}. $$
 Taking the limit as $\Omega_s \nearrow M$ ensures the  limit $X \in L^2(M)$ with \eqref{XSol-UniformEstimatefL2}. 

To obtain  \eqref{XSol-UniformEstimatef} on arbitrary compacts $K\subset \subset K'$ one can use interior $L^2$ estimates and reproduce the bootstrap of Step 3 to go from \eqref{050720211611} to the proper estimate. 
\end{proof}

We will now prove a lemma inspired by \cite{GunPig}, which will  grant injectivity and continuity of the solution map associated with (\ref{deltagammaconff}).
 
\begin{lem}\label{Pigola}
Let $(M^n,\gamma)$ be a complete smooth Riemannian manifold. Then, there is a constant $C=C(n)$ such that the following estimate holds for all $X\in W_{ \mathrm{loc} }^{2,2}(M) \cap L^2 (M)$ such that $\Delta_{\gamma,\mathrm{conf}}X\in L^2(M)$:
\begin{align}\label{Pigola.1}
\Vert \pounds_{\gamma, \mathrm{conf}}X\Vert^2_{L^2(M)}\leq C\Vert \langle X,\Delta_{\gamma, \mathrm{conf}}X \rangle_{\gamma}\Vert_{L^1(M)}, 
\end{align}
\end{lem}
\begin{proof}
Consider  $X\in W_{ \mathrm{loc} }^{2,2}(M) \cap L^2 (M)$ and a function $\varphi\in C^{\infty}_0(M)$, $\varphi\geq 0$, and notice that, in a distributional sense:
\begin{align*}
\mathrm{div}_{\gamma}(\varphi^2\pounds_{\gamma,\mathrm{conf}}X(X,\cdot))
&=2\varphi \pounds_{\gamma,\mathrm{conf}}X(X,\nabla \varphi)+ \varphi^2\langle \Delta_{\gamma, \mathrm{conf}}X, X\rangle_{\gamma} + \frac{1}{2}\varphi^2 \vert \pounds_{\gamma,\mathrm{conf}}X\vert^2_{\gamma}.
\end{align*}
Integrating the above equation we find
\begin{align*}
\frac{1}{2}\int_{M}\varphi^2 \vert \pounds_{\gamma,\mathrm{conf}}X\vert^2_{\gamma}dV_{\gamma}=-\int_M2\varphi \pounds_{\gamma,\mathrm{conf}}X(X,\nabla \varphi)dV_{\gamma} - \int_M\varphi^2\langle \Delta_{\gamma, \mathrm{conf}}X, X\rangle_{\gamma}dV_{\gamma}.
\end{align*}
Now, apply the pointwise estimates almost everywhere $\pounds_{\gamma,\mathrm{conf}}X(X,\nabla \varphi)\leq C(n) \vert\pounds_{\gamma,\mathrm{conf}}X\vert_{\gamma}\vert X\vert_{\gamma} \vert \nabla \varphi\vert_{\gamma}$, and then, given $\epsilon>0$, apply Young's inequality $ab\leq \frac{\epsilon a^2}{2} + \frac{b^2}{2\epsilon}$ to get
\begin{align*}
\int_M\varphi\vert \pounds_{\gamma,\mathrm{conf}}X\vert_{\gamma}\vert X \vert_{\gamma} \vert \nabla \varphi\vert_{\gamma}dV_{\gamma}\leq \frac{\epsilon}{2}\int_M\varphi^2\vert\pounds_{\gamma,\mathrm{conf}}X\vert^2_{\gamma}dV_{\gamma} + \frac{1}{2\epsilon}\int_M\vert X \vert^2_{\gamma} \vert \nabla \varphi \vert^2_{\gamma}dV_{\gamma}.
\end{align*}
Therefore, we find
\begin{align*}
\begin{split}
\frac{1}{2}\int_{M}\varphi^2 \vert \pounds_{\gamma,\mathrm{conf}}X\vert ^2_{\gamma}dV_{\gamma}&\leq C(n)\frac{\epsilon}{2}\int_M\varphi^2\vert\pounds_{\gamma,\mathrm{conf}}X\vert^2_{\gamma}dV_{\gamma} + \frac{C(n)}{2\epsilon}\int_M\vert X \vert^2_{\gamma} \vert \nabla \varphi\vert^2_{\gamma}dV_{\gamma} \\
&+ \int_M\varphi^2\vert \langle \Delta_{\gamma,\mathrm{conf}}X, X\rangle_{\gamma}\vert dV_{\gamma}.
\end{split}
\end{align*}
Picking $\epsilon$ sufficiently small, we can absorb the first term in the right-hand side into the left-hand side, so as to find a fixed constant $C>0$ such that 
\begin{align}
\int_{M}\varphi^2 \vert \pounds_{\gamma,\mathrm{conf}}X \vert^2_{\gamma}dV_{\gamma}&\leq C\left( \int_M \vert X \vert^2_{\gamma} \vert \nabla \varphi\vert^2_{\gamma}dV_{\gamma} + \int_M\varphi^2\vert \langle \Delta_{\gamma, \mathrm{conf}}X, X\rangle_{\gamma}\vert dV_{\gamma}\right).
\end{align}
Now, using the above inequality along a sequence of first order cut-off functions $\{\varphi_j\}_{j=1}^{\infty}$ (which exists since $M$ is complete, see  B. Güneysu's \cite[Theorem 2.2]{Gun} or  M. Shubin's \cite[Proposition 4.1]{SHUBIN200192}), using monotone and dominated convergence, one finds the desired estimate.
\end{proof}

From this, we deduce the following corollary:
\begin{cor}\label{applicationLipschitz}
If $(M^n, \gamma)$ is smooth complete, satisfies $\lambda_{1,\gamma, \mathrm{conf}} >0$ and admits no  non-trivial  global CKF, then for any $f \in L^p_{\mathrm{loc}}(M) \cap L^2(M)$ there exists a unique $X_f \in L^2(M) \cap W^{2,p}_{\mathrm{loc}}(M)$ satisfying \eqref{deltagammaconff}-\eqref{XSol-UniformEstimatefL2}. 
In addition, the linear map $\mathcal{P}:\, f\in L^p_{\mathrm{loc}}(M) \cap L^2(M) \rightarrow X_f \in W^{2,p}_{\mathrm{loc}}(M) \cap L^2(M)$ satisfies, for all $f_1$, $f_2$:

\begin{equation}
\label{L2LipschitzXf}
\| X_{f_1} - X_{f_2} \|_{L^2(M)} \le \frac{2}{\lambda_{1,\gamma,\mathrm{conf}}} \|f_1-f_2\|_{L^2(M)},
\end{equation}
and  for all $K\subset \subset K' \subset \subset M$ with smooth boundary
\begin{equation}
\label{LpLipschitzXf}
\| X_{f_1} - X_{f_2} \|_{W^{2,p}(K)} \le \frac{C(n, \gamma, K, K')}{\lambda_{1,\gamma,\mathrm{conf}}} \left( \| f_1-f_2\|_{L^p(K')}+   \|f_1-f_2\|_{L^2(M)} \right).
\end{equation}
\end{cor}
\begin{proof}
Lemma \ref{Xagainstf} ensures that for any $f \in L^p_{\mathrm{loc}}(M) \cap L^2(M)$ there exists a $ L^2(M) \cap W^{2,p}_{\mathrm{loc}}(M)$ solution of \eqref{deltagammaconff} satisfying \eqref{XSol-UniformEstimatef} and \eqref{XSol-UniformEstimatefL2}. In addition if $X$ and $Y$ are two such solutions, $X-Y \in L^2(M) \cap W^{2,p}_{\mathrm{loc}}(M) \subset L^2(M) \cap W^{2,2}_{\mathrm{loc}}(M)$ and satisfies $\Delta_{\gamma, \mathrm{conf}}(X-Y) = 0$. Then \eqref{Pigola.1} ensures: $\|\pounds_{\gamma, \mathrm{conf}} (X-Y) \|^2_{L^2(M)} \le 0$. Since $(M, \gamma)$ admits no non-trivial global CKF, $X- Y = 0$, which shows uniqueness.

The map $\mathcal{P}$ is then well-defined, and can be shown to be linear in a direct manner.  Finally, if $f_1$, $f_2\in L^p_{\mathrm{loc}}(M) \cap L^2(M)$, then $\Delta_{\gamma, \mathrm{conf}} \left( X_{f_1}-X_{f_2} \right) = f_1- f_2$, and thus $X_{f_1}-X_{f_2} = X_{f_1-f_2}$ by uniqueness. Applying \eqref{XSol-UniformEstimatef}-\eqref{XSol-UniformEstimatefL2} to $X_{f_1}-X_{f_2}$ then yields \eqref{L2LipschitzXf}-\eqref{LpLipschitzXf}.
\end{proof}

Let us now ensure that the right-hand side of the second equation of \eqref{systoncomplphiX} depends continuously on the choice of the conformal factor.

\begin{lem}
\label{lemmacontrolrighthandside}
Let $(M^n, \gamma)$  be a  complete smooth Riemannian manifold, $\tau, \omega_1, \omega_2$ satisfying the controls in \eqref{hypcomplLp}  and  $m>0$ a constant.  Then the map $\mathcal{G}: \varphi \in L^\infty(M) \cap \{ 0 <\varphi \le m \} \mapsto \frac{n-1}{n} \nabla \tau \varphi^{\frac{2n}{n-2}} + \omega_1 \varphi^{2\frac{n+1}{n-2}} - \omega_2 \in L^p_{\mathrm{loc}}(M) \cap L^2(M)$ is well defined and for all  $0< \varphi \le m$ satisfies:
\begin{equation}
\label{GboundL2}
\begin{aligned}
\| \mathcal{G} ( \varphi) \|_{L^2(M)} &\le C \left( \left\| \nabla \tau \right\|_{L^2 (M)}   m^{\frac{2n}{n-2} }  + \left\| \omega_1 \right\|_{L^2(M) } m^{2\frac{n+1}{n-2}}    +  \left\| \omega_2 \right\|_{L^2(M) }  \right),
\end{aligned}
\end{equation}
\begin{equation}
\label{GisLipschitzL2}
\| \mathcal{G}(\varphi_1) - \mathcal{G}(\varphi_2) \|_{L^2(M)} \le C\left( m,\| \nabla \tau \|_{L^{2}(M)}, \| \omega_1 \|_{L^{2}(M)}     \right) \| \varphi_1 - \varphi_2\|_{L^\infty(M)}
\end{equation}
and for all compacts $K \subset \subset M$
\begin{equation}
\label{GboundL^p}
\begin{aligned}
\| \mathcal{G} ( \varphi) \|_{L^p(K)} &\le C  \left( \left\| \nabla \tau \right\|_{L^p (K)}   m^{\frac{2n}{n-2} }  + \left\| \omega_1 \right\|_{L^p(K) } m^{2\frac{n+1}{n-2}}    +  \left\| \omega_2 \right\|_{L^p(K) }  \right),
\end{aligned}
\end{equation}
\begin{equation}
\label{GisLipschitzLp}
\| \mathcal{G}(\varphi_1) - \mathcal{G}(\varphi_2) \|_{L^p(K)} \le C\left( m,\| \nabla \tau \|_{L^{p}(K)}, \| \omega_1 \|_{L^{p}(K)}     \right) \| \varphi_1 - \varphi_2\|_{L^\infty(M)}.
\end{equation}
\end{lem}
\begin{proof}
The $L^2$ and $L^p$ controls given by  \eqref{hypcomplLp} combined with the $C^1$ bound of $t\mapsto t^I$ for $I\ge 1$ on $(0 , m]$ ensure that $\mathcal{G}$ satisfies \eqref{GboundL2}-\eqref{GisLipschitzLp}.
\end{proof}

Combining Corollary \ref{applicationLipschitz} and Lemma \ref{lemmacontrolrighthandside} yields:

\begin{cor}\label{050720211748}
Assume that on a complete smooth manifold $(M^n, \gamma)$  the system \eqref{systoncomplphiX} admits  two global barrier functions $\phi_{\pm}$, that $\lambda_{1, \gamma, \mathrm{conf} }>0$, and that $\gamma$ admits no non-trivial global CKFs. Then, for any $\varphi \in L^\infty(M)$ such that $\phi_- \le \varphi \le \phi_+ \le m$ there exists a unique $L^2(M) \cap W^{2,p}_{\mathrm{loc}}(M)$ solution $X_\varphi$ to the equation
\begin{equation}
\label{050720211652}
\Delta_{\gamma, \mathrm{conf}} X_\varphi - \frac{n-1}{n} \nabla \tau \varphi^{\frac{2n}{n-2}}  - \omega_1\varphi^{2\frac{n + 1}{n-2}} + \omega_2  =0.
\end{equation} 
Moreover, for any pair of compact sets with smooth boundary $K \subset \subset K' \subset \subset  M$ one has the following control on $X_\varphi$:
\begin{equation}\label{XSol-UniformEstimate}
\begin{aligned}
\left\| X_\varphi \right\|_{W^{2,p} (K) } &\le \frac{C(n, \gamma, K )}{\lambda_{1, \gamma, \mathrm{conf}} } \left[ \left( \left\| \nabla \tau \right\|_{L^2 (M)}  + \| \nabla \tau \|_{L^{p}(K')} \right) m^{\frac{2n}{n-2} }  \right. \\ & \left. + \left(\left\| \omega_1 \right\|_{L^2(M) } + \| \omega_1 \|_{L^{p}(K')}  \right) m^{2\frac{n+1}{n-2}}    + \left(  \left\| \omega_2 \right\|_{L^2(M) } + \| \omega_2 \|_{L^{p}(K')}  \right) \right]\\
\left\| X_\varphi \right\|_{L^2(M) } &\le \frac{C(n, \gamma )}{\lambda_{1, \gamma, \mathrm{conf}} } \left[  \left\| \nabla \tau \right\|_{L^2 (M)}  m^{\frac{2n}{n-2} }   + \left\| \omega_1 \right\|_{L^2(M) }  m^{2\frac{n+1}{n-2}}    +  \left\| \omega_2 \right\|_{L^2(M) }  \right].
\end{aligned}
\end{equation}
In addition, for any $\phi_- \le \varphi_1, \varphi_2 \le \phi_+ \le m$:
\begin{equation}\label{XSol-UniformEstimateLipschitz}
\begin{aligned}
\left\| X_{\varphi_1} - X_{\varphi_2} \right\|_{W^{2,p} (K) } &\le C\left( m,\| \nabla \tau \|_{L^{p}(K')}, \| \omega_1 \|_{L^{p}(K')}     \right) \| \varphi_1 - \varphi_2\|_{L^\infty(M)}, \\
\left\| X_{\varphi_1} - X_{\varphi_2} \right\|_{L^2(M) } &\le C\left( m,\| \nabla \tau \|_{L^{2}(M)}, \| \omega_1 \|_{L^{2}(M)}     \right) \| \varphi_1 - \varphi_2\|_{L^\infty(M)}.\\
\end{aligned}
\end{equation}
\end{cor}

\begin{remark}
Let us briefly contrast the above analysis concerning solutions to the momentum constraint on a complete manifold with the type of hypotheses given in Section \ref{sectionPDEgen} for the ``momentum''-type equations. While in Section \ref{sectionPDEgen} existence of solutions to the linear problems, and moreover a continuous inverse for the associated linear operator, are consequences of a general theory of regular elliptic boundary value problems under suitable boundary conditions (see, for instance, Appendix \ref{appendixlinearanalysisellipticboundaryvalueproblem}), the parallel existence, uniqueness and continuous dependence of solutions to (\ref{deltagammaconff}) relies on a diagonal extraction scheme which is sensitive to the shape of the equation.
\end{remark}

We will apply Corollary \ref{050720211748} with $ \varphi \in L^\infty(M)$ given as an extension of a $L^\infty$ function on a compact subset, as described in the following lemma:
\begin{lem}
\label{extensionlemma}
Let  $\Omega$ be a compact subset of $(M^n, \gamma)$, a complete manifold with smooth boundary. Assume the system \eqref{systoncomplphiX} admits  two global barrier functions $\phi_{\pm}$. Then, the extension map
$$ \mathcal{H} : \, \left\{  \begin{aligned} L^\infty(\Omega)  &\rightarrow L^\infty(M)\\ 
 \phi & \mapsto  \varphi = \left\{ \begin{aligned}& \phi \text{ on } \Omega \\ & \frac{\phi_+ + \phi_-}{2} \text{ on } M \backslash \Omega \end{aligned} \right. \end{aligned} \right. $$
is Lipschitz, and if $\phi_- \le \phi \le \phi_+ \implies  \phi_- \le \mathcal{H}( \phi) \le \phi_+$.
\end{lem}
\begin{proof}
The last implication is straightforward while if $\phi_1, \phi_2 \in L^\infty(\Omega)$, since $\varphi_1 = \varphi_2 =  \frac{\phi_+ + \phi_-}{2}$ on $M \backslash \Omega$ one has:
\begin{equation}
\label{Lipschitzphivarphi}
\| \varphi_1 - \varphi_2 \|_{L^\infty(M)} \le \|\varphi_1- \varphi_2\|_{L^\infty( \Omega)} + \|\varphi_1- \varphi_2\|_{L^\infty( M \backslash \Omega)} = \|\phi_1- \phi_2\|_{L^\infty( \Omega)}.
\end{equation}
\end{proof}

This preparatory work  enables us obtain a solution of the momentum equation over the whole manifold for which the global barrier inequalities (Definition \ref{GlobalBarriersDefn}) apply. This will allow us to set up a fixed point theorem similar to the one in Theorem \ref{theoexistcompact} when analysing (\ref{systoncomplphiX}) on an arbitrary smooth $\Omega\subset \subset M$. We cannot directly apply Theorem \ref{theoexistcompact} in such a situation since our global barrier hypotheses do not grant that on such $\Omega$ (\ref{invariantcondition})-(\ref{300320201549}) are satisfied. The shift from Definition \ref{strigglobalbarrierscomp} to Definition \ref{GlobalBarriersDefn} in this section is tailored to accommodate flexible asymptotics on a general non-compact manifolds where direct a priori estimates for the momentum constraint may not be available and to impose, whenever possible, only local integrability conditions on coefficients. 

\begin{lem}\label{existenceoncompact}
Assume that on a complete smooth manifold $(M^n, \gamma)$, $n\ge 3$, with no global CKF,  the system \eqref{systoncomplphiX}  satisfies \eqref{hypcomplLp} and admits  two global barrier functions $0< \phi_- \le \phi_+ \le m < \infty$, and that $\lambda_{1, \gamma, \mathrm{conf} }>0.$ Then, for any compact subset with smooth boundary $\Omega$, \eqref{systoncomplphiX} admits a  solution $(\phi,X)\in W^{2,p}(\Omega)$ where, with the same notations as in the previous lemmas,  $X$ is globally defined on $M$ and given by $X = \mathcal{P}\circ \mathcal{G} \circ \mathcal{H} (\phi) \in W^{2,p}_{\mathrm{loc}} (M) \cap L^2 (M)$. 
\end{lem}
\begin{proof}
Let  $\Omega$ be a compact subset of $M^n$ with smooth boundary and $ \mathcal{U}_ 0 = \{ \phi \in W^{2,p}(\Omega) \text{ s.t. } 0 < \phi_- \le \phi \le \phi_+\}.$ Since $\Omega$ is compact, there exists $l>0$ such that $\phi_-\ge l >0$ on $\Omega$.

\vspace{0.25cm}

\underline{\textbf{Step 1: } Solutions to the linear problems and the fixed point map}\newline
Using the notations of Corollary \ref{050720211748} and Lemmas \ref{lemmacontrolrighthandside}-\ref{extensionlemma}, for any $\phi \in \mathcal{U}_0$ there exists  $X_\phi = \mathcal{P}\circ \mathcal{G} \circ \mathcal{H} (\phi)$ a $W^{2,p}_{\mathrm{loc}} (M) \cap L^2 (M)$  solution of the momentum equation \eqref{050720211652}. Considering $\Omega'$ a compact strictly containing $\Omega$ and applying \eqref{XSol-UniformEstimate} yields a uniform $W^{2,p}(\Omega)$ bound on $X_\phi$, since
 notably the right-hand side of the inequality does not depend on $\phi \in \mathcal{U}_0$. In addition, combining \eqref{XSol-UniformEstimateLipschitz} and \eqref{Lipschitzphivarphi} ensures that for all $\phi_1,\phi_2 \in \mathcal{U}_0$,
\begin{equation}
\begin{aligned}
\left\| X_{\phi_1}- X_{\phi_2} \right\|_{W^{2,p} (\Omega) } &\le C\left( m,\| \nabla \tau \|_{L^{p}(\Omega')}, \| \omega_1 \|_{L^{p}(\Omega')}     \right) \| \phi_1 - \phi_2\|_{L^\infty(\Omega)}.
\end{aligned}
\end{equation}

Consequently, since $\tilde K(X_\phi)=\pounds_{\gamma, \mathrm{conf}}X + U $ and Sobolev embeddings applied to \eqref{XSol-UniformEstimate} ensure that $\nabla X_\phi$ is uniformly bounded in $C^1(\Omega)$. 
Thus, along the same lines as in \eqref{300320201647}, we can introduce a large enough \emph{shift} function $a>0$ such that the map, $h_{\phi,a}: \, \mathcal{U}_0 \rightarrow L^p(\Omega)$, $\phi\in \mathcal{U}_0$, defined by
\begin{equation} \label{defdeh}
\begin{aligned}
h_ {\phi,a}(\psi)&\doteq  R_\gamma \psi -  \left\vert \tilde K(X_\phi) \right\vert^2_\gamma \psi^{-\frac{3n-2}{n-2} } +  \frac{n-1}{n} \tau^2 \psi^{\frac{n+2}{n-2} }- 
2 \epsilon_1  \psi^{\frac{n+2}{n-2}} - 2  \epsilon_2 \psi^{-3} - 2  \epsilon_3 \psi^{\frac{n-6}{n-2} }- a \psi.
\end{aligned}
\end{equation}
is a non-increasing function of $\psi$.
Along the lines of Theorem \ref{theoexistcompact}, using the boundedness of $t\mapsto t^I$ on $[l,m]$ and (\ref{XSol-UniformEstimate}), we find a constant $C>0$ depending on $\|R_\gamma\|_{L^p(\Omega)}$ , $\|a\|_{L^p(\Omega)}, \|U\|_{L^p (\Omega)}$,  $\| \epsilon_1  \|_{L^p (\Omega)}$, $\| \epsilon_2 \|_{L^p (\Omega)}$, $\| \epsilon_3  \|_{L^p (\Omega)}$, $\| \tau^2 \|_{L^p (\Omega)}$,$l$, and $m$, such that for all  $\phi, \psi \in \mathcal{U}_0$:
\begin{equation}
\label{uniformcontrolhphiapsi}
\|h_{\phi, a} (\psi)\|_{L^p(\Omega)} \le C.
\end{equation}

In addition, for $\phi_1, \psi_1, \phi_2, \psi_2 \in \mathcal{U}_0$, using that $t \mapsto t^I$ is Lipschitz on $[l,m]$ for any $I$, there exists $C\left(\|R_\gamma\|_{L^p(\Omega)} , \|a\|_{L^p(\Omega)}, \|U\|_{L^p (\Omega)}, \| \epsilon_1  \|_{L^p (\Omega)}, \| \epsilon_2 \|_{L^p (\Omega)}, \| \epsilon_3  \|_{L^p (\Omega)}, \| \tau^2 \|_{L^p (\Omega)},l, m\right)$ such that:
$$\begin{aligned}
\| h_{\phi_1,a}(\psi_1) - h_{\phi_2,a}(\psi_2) \|_{L^p(\Omega)} \le C  \|\psi_1- \psi_2 \|_{L^\infty(\Omega)}  +
\left\| \left\vert \tilde K(X_{\phi_1}) \right\vert^2_\gamma \psi_1^{-\frac{3n-2}{n-2} }- \left\vert \tilde K(X_{\phi_2}) \right\vert^2_\gamma \psi_2^{-\frac{3n-2}{n-2} } \right\|_{L^p(\Omega)}.
\end{aligned}$$

Writing
\begin{equation} \label{controldifferenceofKtilde1}
\begin{aligned}
\left\| \left\vert \tilde K(X_{\phi_1}) \right\vert^2_\gamma \psi_1^{-\frac{3n-2}{n-2} }- \left\vert \tilde K(X_{\phi_2}) \right\vert^2_\gamma \psi_2^{-\frac{3n-2}{n-2} } \right\|_{L^p(\Omega)} &\le C l^{-\frac{3n-2}{n-2}}\left\| \left\vert \tilde K(X_{\phi_1}) \right\vert^2_\gamma - \left\vert \tilde K(X_{\phi_2}) \right\vert^2_\gamma  \right\|_{L^{p}(\Omega)} \\&  + C l^{- \frac{3n-2}{n-2} -1} \| \psi_1- \psi_2\|_{L^\infty(\Omega)}\\
&\le C' l^{-\frac{3n-2}{n-2}} (\|X_{\phi_1} - X_{\phi_2} \|_{W^{2,p}(\Omega)}  + l^{ -1} \| \psi_1- \psi_2\|_{L^\infty(\Omega)}),
\end{aligned}\end{equation}
by \eqref{XSol-UniformEstimate} and the Lipschitz nature of $t \mapsto t^{-\frac{3n-2}{n-2}}$, while controlling the $\tilde K$ terms in the following manner:
$$\begin{aligned}
\left\| \left\vert \tilde K(X_{\phi_1}) \right\vert^2_\gamma - \left\vert \tilde K(X_{\phi_2}) \right\vert^2_\gamma  \right\|_{L^{p}(\Omega)} &\le \left\| \tilde K(X_{\phi_1})  - \tilde K(X_{\phi_2}) \right\|_{L^{\infty}(\Omega)} \left\| \tilde K(X_{\phi_1})  + \tilde K(X_{\phi_2}) \right\|_{L^{p}(\Omega)} \\
&\le C(\Omega, \|U\|_{L^p(\Omega)}, \|X_{\phi_1}\|_{W^{1,p}(\Omega)}, \|X_{\phi_2}\|_{W^{1,p}(\Omega)} ) \|X_{\phi_1} - X_{\phi_2}\|_{W^{2,p}(\Omega)}.
\end{aligned}$$

 Applying \eqref{XSol-UniformEstimateLipschitz} and injecting this in \eqref{controldifferenceofKtilde1} ensures that there exists $C$ depending on $L^p(\Omega') \cap L^2(M)$ norms on the data and $l$, $m$ such that:
\begin{equation}
\label{haphiLipschitz}
\| h_{\phi_1,a}(\psi_1) - h_{\phi_2,a}(\psi_2) \|_{L^p(\Omega)} \le C \left( \|\phi_1- \phi_2\|_{L^\infty(\Omega)} + \| \psi_1 - \psi_2 \|_{L^\infty(\Omega)}\right).
\end{equation}

Let us now define the solution map associated with the Schrödinger operator $(a_n\Delta_{\gamma}-a,B_0)$ (see Theorem \ref{ShiftedRoughLapBelIso}):
$$\mathcal{F} :\, \left\{ \begin{aligned}
&\mathcal{U}_0 \rightarrow W^{2,p}(\Omega) \\
& \phi \mapsto \psi = (a_n\Delta_\gamma - a , B_0)^{-1}\left(h_{\phi,a}(\phi), \frac{\phi_-+\phi_+}{2} \right),
\end{aligned}\right.$$

Given $\phi \in \mathcal{U}_0$ and $\psi = \mathcal{F}(\phi)$, since $X_\phi$ solves the momentum equation for $\phi_-\le \phi \le \phi_+$, then 
$a_n\Delta_{\gamma} \phi_- - a \phi_- \ge h_{\phi,a}(\phi_-)$ on $\Omega$. Similarly to our analysis in Theorem \ref{theoexistcompact}, since $\phi\geq \phi_{-}$, our choice of shift function ``$a$''  and boundary data for the solution map $\mathcal{F}$ grant that $\phi_{-}\leq \psi\leq \phi_{+}$ and therefore that $\mathcal{F}:\mathcal{U}_0\to \mathcal{U}_0$ is well-defined, and moreover there is a constant $C(\Omega, \gamma, a)$ such that:
$$\begin{aligned}\| \mathcal{F}(\phi) \|_{W^{2,p}(\Omega)} 
&\le C \left( \|h_{\phi,a} (\phi)\|_{L^p(\Omega)} +\left\| \phi_+ \right\|_{W^{2,p}(\Omega)} + \left\| \phi_- \right\|_{W^{2,p}(\Omega)}    \right), \end{aligned}$$
which combined with \eqref{uniformcontrolhphiapsi} grants that \begin{equation}
\label{uniformboundmathcalF}
\|\mathcal{F}(\phi) \|_{W^{2,p}(\Omega)} \le C_0.
\end{equation}
for a constant $C_0$, which depends on the  $L^p(\Omega') \cap L^2(M)$ norms on the data, $l$, $m$, $\left\| \phi_+ \right\|_{W^{2,p}(\Omega)}$, $\left\| \phi_- \right\|_{W^{2,p}(\Omega)}$, but not on $\phi$.

Finally, if $\phi_1$, $\phi_2 \in \mathcal{U}_0$, since 
$$
\left\{\begin{aligned}
&a_n \Delta_\gamma \left( \mathcal{F}(\phi_1) - \mathcal{F}(\phi_2) \right) - a \left( \mathcal{F}(\phi_1) - \mathcal{F}(\phi_2) \right) = h_{\phi_1,a}(\phi_1) - h_{\phi_2,a}(\phi_2)   \text{ in } \Omega, \\
& \left( \mathcal{F}(\phi_1) - \mathcal{F}(\phi_2) \right)= 0  \text{ on } \partial \Omega,
\end{aligned}\right.$$
the invertibility of the Schrödinger operator  and \eqref{haphiLipschitz} for $\psi_1= \phi_1$ and $\psi_2= \phi_2$ ensure that there exists a constant $C$ depending on  $L^p(\Omega') \cap L^2(M)$ norms on the data, $\gamma$, $\Omega$ and $l$, $m$ such that:

\begin{equation}
\label{mathcalFLipschitzfinally}
\|\mathcal{F}(\phi_1) - \mathcal{F}(\phi_2)  \|_{W^{2,p}(\Omega)} \le C \|\phi_1-\phi_2\|_{L^\infty(\Omega)}.
\end{equation}

\vspace{0.25cm} 

\underline{\textbf{Step 2: } Fixed point theorem}\newline

Similarly to the last part of Theorem \ref{theoexistcompact}, we now consider $\mathcal{F}$ on $\mathcal{U}= \mathcal{U}_0 \cap \overline{B}_{ \tilde C_0}$, where $\tilde C_0 \doteq C_0+ \|\phi_- \|_{W^{2,p}(\Omega)} $ and $C_0$ is the constant in \eqref{uniformboundmathcalF}. This set is again closed and convex in $W^{2,p}$, and $\mathcal{F}:\mathcal{U}\to \mathcal{U}$ by the invariance of $\mathcal{F}$ over $\mathcal{U}_0$ just shown above, together with (\ref{uniformboundmathcalF}). Moreover, (\ref{mathcalFLipschitzfinally}) shows this map is Lipschitz. Finally, (sequential) precompactness is a direct consequence of (\ref{mathcalFLipschitzfinally}) which substitutes the appeal to (\ref{pseudoLipschiptz1})-(\ref{pseudoLipschiptz2}) in the parallel argument of Theorem \ref{theoexistcompact}. Thus, one can apply Theorem \ref{fixpointschauder} and deduce that $\mathcal{F}$ has a fixed point $\phi$, which satisfies:
\begin{equation} \label{eqoncompctphi}\left\{
\begin{aligned}
&a_n\Delta_\gamma \phi -  R_\gamma \phi+  \left\vert \tilde K(X_\phi) \right\vert^2_\gamma \phi^{-\frac{3n-2}{n-2} } - \frac{n-1}{n} \tau^2 \phi^{\frac{n+2}{n-2} } \\&+
2  \epsilon_1  \phi^{\frac{n+2}{n-2}} + 2  \epsilon_2 \phi^{-3} + 2 \epsilon_3 \phi^{\frac{n-6}{n-2} } =0 \text{ in } \Omega \\
&\phi = \frac{\phi_+ +\phi_-}{2} \text{ on } \partial \Omega, \end{aligned}
\right.\end{equation}
where $X_\phi$ is defined as in Step 1, and in particular solves
\begin{equation}\label{eqoncompctX}\Delta_{\gamma, \mathrm{conf}} X_\phi - \frac{n-1}{n} \nabla \tau \phi^{\frac{2n}{n-2} } - \omega_1\phi^{2\frac{n + 1}{n-2}} + \omega_2  =0 \text{ in } \Omega. \end{equation}

Together $(\phi, X_\phi)$ realize a solution of \eqref{systoncomplphiX} on $\Omega$.
\end{proof}

One can then apply a diagonal extraction scheme similar to that of Lemma \ref{Xagainstf} to obtain global solutions. This is a key difference with respect to barrier-based existence criteria of Theorem \ref{theoexistcompact} where a global solution on the (compact) manifold is directly obtained after the fixed point argument.

\begin{theorem}
\label{050720211748bis}
Assume that on a complete smooth manifold $(M^n, \gamma)$, $n\ge 3$, with no global CKF,  the system \eqref{systoncomplphiX}  satisfies \eqref{hypcomplLp} and admits  two global barrier functions $0< \phi_- \le \phi_+ \le m < \infty$, and that the first eigenvalue of the conformal Killing Laplacian satisfies $\lambda_{1, \gamma, \mathrm{conf} }>0.$ Then \eqref{systoncomplphiX} admits a  solution $(\phi,X)\in W^{2,p}_{\mathrm{loc}}(M)$.
\end{theorem}
\begin{proof}
Let us consider a compact exhaustion $\{\Omega_k\}_{k=1}^\infty$ of $M$ with smooth boundary and let $(\phi_k, X_k)$ be a solution to \eqref{systoncomplphiX} on $\Omega_k$, obtained with Lemma \ref{existenceoncompact}. By construction of $X_k$ as the solution to the momentum equation defined in Corollary \ref{050720211748}, it satisfies  \eqref{XSol-UniformEstimate} for arbitrary compacts $K \subset K' \subset \Omega_k$. In particular, one obtains $W^{2,p}$ controls on $X_k$ that depend on the chosen compacts $K$ and $K'$, but which are uniform in $k$.

Let us now consider $K \subset \subset K' \subset \subset K''\subset \Omega_k$. Since $\phi_{-} \in  W^{2,p}(K')\hookrightarrow  C^{0, \alpha}(K')$ and $\phi_->0$, there exists $l_{K'} = \min_{K'} \phi_->0$ such that on $K'$, $\phi_k \ge  \phi_-\ge l_{K'}$. Working in parallel to Step 1 in the proof of Lemma \ref{existenceoncompact} above, there is a constant $C>0$, depending on $\|R_\gamma\|_{L^p(K')},  \| U \|_{L^p (K')},  \| \epsilon_1  \|_{L^p (K')}, \| \epsilon_2 \|_{L^p (K')}, \| \epsilon_3  \|_{L^p (K')}$, $\| \tau^2 \|_{L^p (K')},l_{K'}$ and $m$ such that, for any $k$ and for $h_{\phi_k, 0}$ defined in \eqref{defdeh}:  
$$\|h_{\phi_k, 0} (\phi_k)\|_{L^p(K')} \le C \left( \|\nabla X_k\|_{L^p(K')} + 1\right),$$
which again coupled with (\ref{XSol-UniformEstimate}) applied between $K'$ and $K''$ provides a constant $C$ with similar dependences such that:
\begin{equation}
\label{uniformcontrolhphik0phik}
\|h_{\phi_k, 0} (\phi_k)\|_{L^p(K')} \le C.
\end{equation}

Interior elliptic estimates between $K$ and $K'$ together with (\ref{uniformcontrolhphik0phik}) and the uniform bound $\phi_k \le m$ yield the existence of a constant $C$ that once more depends on $L^p(K'') \cap L^2(M)$ norms on the data, $K$, $K'$, $K''$ and $l_{K'}$, $m$ such that 
\begin{equation}
\label{uniformcontrolphikoncompacts}
\| \phi_k\|_{W^{2,p} (K ) }  \le C.
\end{equation}
This control depends on the compacts, but not on $\Omega_k$, and is uniform in $k$. Then, \eqref{XSol-UniformEstimate} and  \eqref{uniformcontrolphikoncompacts} yield that $\left\{\psi_k \right\}_{k =1}^\infty= \left\{(\phi_k, X_k)\right\}_{k =1} ^\infty$ is uniformly $W^{2,p}$ bounded on any $\Omega_s$, $s<k$. We may then finish the proof by a diagonal extraction following closely the same arguments as in Step 4 of Lemma \ref{Xagainstf} and thus obtaining $\psi=(\phi,X)\in W^{2,p}_{loc}(M)$ solving \eqref{systoncomplphiX} around any point, and hence over all of $M$.

\end{proof}

\begin{remark}
Classical elliptic regularity ensures that more regularity in the data for Theorem \ref{050720211748bis} yields more regularity for the solution. In particular, smooth data will yield smooth solutions.
\end{remark}

\begin{remark}
\label{remarkbarrierfunctions}
We again stress our appeal to \emph{global barrier functions} in the above theorem, instead of  \emph{strong global barrier functions} in Section \ref{sectionPDEgen}. The difference lies in the space of admissible $X$ for which the barriers are sub and super solutions of the Lichnerowicz equation  (for all $X$ in  a $W^{2,p}$ ball in the first case, for all $X$ solving an equation in the second one) and the conditions required to solve the equation.
\end{remark}

\begin{remark}
It is interesting to investigate the flexibility of our scheme.  For Einstein-type systems \eqref{300320201529}, if the $Y$ equations are decoupled once $\phi$ is considered a datum, one can proceed as  in Lemma \ref{050720211748} and Theorem \ref{050720211748bis} under the appropriate spectral and integrability assumptions. 

If we no longer assume that the $Y$-system decouples but is ``triangular'' as in Example \ref{020320211748}, one can assume that $\tilde q \in L^2(M)$, that  $\lambda^{-\Delta_\gamma}_1>0$, and first find solutions $f_\varphi$, and then the corresponding $X_\varphi$ when $\nabla \tau$, $\omega_1$, $\tilde F_{ik} \mathcal{V}^k$,  and $\tilde F_{ik} \nabla^k f_\varphi$ lie in $L^2(M)$. This requires a uniform $L^2$ control on $\nabla f_\varphi$, and thus a uniform $W^{1,2}$ a priori estimate on $f_\varphi$ working as in \eqref{Pigola.2} and  \eqref{280720211228} below.
\end{remark}

\section{Building global barrier functions}\label{section210720211556}

\subsection{Barrier functions on Bounded Geometry}\label{subsecbarrfunct}
In this section we assume that $(M, \gamma)$ is a smooth complete Riemannian manifold of bounded geometry (see Appendix \ref{BoundedGeometry} for detailed definitions). We also assume that  $a \doteq R_{\gamma}+\frac{n-1}{n}\tau^2\in  L_{\mathrm{loc}}^{p}(M) \cap L^\infty(M)$, $\tau\in W^{1,p}_{\mathrm{loc}}$, with $p> n$,  and $a\geq a_0>0$ for some constant $a_0$. This condition can be interpreted as a restriction on the admissible mean curvatures, given $R_\gamma$.

Let us fix an exhaustion of $M$ by precompact sets $\{\Omega_k \}$ with smooth boundaries, two constants $0<c_{-} \leq c_+$, $\Lambda_{\pm}\in  L^{\infty}(M)\cap C^{\infty}(M)$ positive functions, $\Lambda_{+}\geq \Lambda_{-}$, and analyse the sequence of problems:
\begin{equation}
\label{subsolsBG.1}
\left\{
\begin{aligned}
a_n\Delta_{\gamma}\varphi_- - a\varphi_-&=c_{-}a - \Lambda_{-}  \text{ in } \Omega_k,\\
\varphi_-&=0 \text{ on } \partial\Omega_k,
\end{aligned}
\right.
\end{equation}

\begin{equation}\label{supersolBG.1}
\left\{
\begin{aligned}
a_n\Delta_{\gamma} \varphi_+ - a \varphi_+ &=c_{+}a - \Lambda_{+}, \text{ in } \Omega_k\\
\varphi_+&=0 \text{ on } \partial\Omega_k.
\end{aligned}
\right.
\end{equation}

From  \eqref{subsolsBG.1} and \eqref{supersolBG.1} we build two sequences of solutions $ \left\{ \varphi_k^- \right\}$ and $\left\{ \varphi_k^+ \right\}$ and  define  $u_k= \varphi_k^- +c_-$ and $v_k = \varphi_k^+ + c_+$, which are solutions to

\begin{equation} \label{subsolsBG.2}
\left\{ \begin{aligned}
a_n\Delta_\gamma u_k - a u_k &= -\Lambda_{-} \text{ in } \Omega_k, \\
u_k &= c_- \text{ on } \partial \Omega_k,
\end{aligned}
\right.
\end{equation}

\begin{equation} \label{supersol.2}
\left\{ \begin{aligned}
a_n\Delta_\gamma v_k - a v_k &= - \Lambda_{+} \text{ in } \Omega_k, \\
v_k &= c_+ \text{ on } \partial \Omega_k.
\end{aligned}
\right.
\end{equation}

\begin{lem}\label{SupersolUpperBound}
Under the above conditions there is a constant $c>0$, depending on the fixed quantities $c_{\pm},\Lambda_{\pm},a_{0}$ such that
\begin{align}
0< u_{k}\leq v_{k}\leq c<\infty \text{ for all } k\in \mathbb{N}.
\end{align}
\end{lem}
\begin{proof}
Since $a>0$ and $0\le c_-\le c_+$, then $h_k\doteq u_k-v_k$ satisfies
\begin{align*}
\left\{ \begin{aligned}
a_n\Delta_\gamma h_k - a h_k &= \Lambda_{+} - \Lambda_{-}\geq 0 \text{ in } \Omega_k, \\
h_k &=c_{-}-c_{+}\leq 0  \text{ on } \partial \Omega_k.
\end{aligned}
\right.
\end{align*}
we can apply the maximum and comparison principles to get $0< u_k \le v_k$,
 where the first inequality follows from the maximum principle since $u_k\geq 0$ and can only attain a non-positive interior minimum if it is constant. 

Now, let $c>0$ be a constant to be determined and consider the difference $\bar{v}_k\doteq v_k-c$ on $\Omega_k$. Then, 
\begin{align*}
\left\{ \begin{aligned}
a_n\Delta_\gamma \bar{v}_k - a \bar{v}_k &= - \Lambda_{+} + ac \text{ in } \Omega_k, \\
\bar{v}_k &=c_{+}-c  \text{ on } \partial \Omega_k.
\end{aligned}
\right.
\end{align*}
Then, let us choose $c$ satisfying
\begin{align*}
\infty > c\geq \max\Big\{\sup_M\frac{\Lambda_{+}}{a},c_{+}\Big\},
\end{align*}
where the first inequality holds since $a\geq a_0>0$ and $\Lambda_{+}\in C^{\infty}(M)\cap L^{\infty}(M)$ imply $\sup_M\frac{\Lambda_{+}}{a}<\infty$. This choice is independent of $v_k$ and $\Omega_k$. Therefore, for any $k\in \mathbb{N}$, we find
\begin{align}
\left\{ \begin{aligned}
a_n\Delta_\gamma \bar{v}_k - a \bar{v}_k &\geq 0 \text{ in } \Omega_k, \\
\bar{v}_k &\leq 0  \text{ on } \partial \Omega_k.
\end{aligned}
\right.
\end{align}
An application of the maximum principle then gives $\bar{v}_k\leq 0$, which is equivalent to $v_k\leq \max \left(\sup_M\frac{\Lambda_{+}}{a},c_{+}\right)$.
\end{proof}
\begin{cor}
\label{corSupersolUpperBound}
If in addition to the hypothesis of Lemma \ref{SupersolUpperBound} we assume that $\Lambda_{+}=a$, $0<c_{-}<1$ and $c_{+}\geq 1$ are fixed, and $\Lambda_{-}=c_{-}a$,  then $u_k=c_{-}$ for all $k\in \mathbb{N}$ and we get $0<c_{-}=u_k\leq v_{k}\leq c<\infty$ for all $ k\in \mathbb{N}$.
\end{cor}
\begin{proof}
Inserting these choices, which satisfy all the general requirements of Lemma \ref{SupersolUpperBound}, into the proof yields the desired result.
\end{proof}

Appealing to the above lemma, let us now show that the sequences $\{u_k\}_{k=1}^{\infty}$ and $\{v_k\}_{k=1}^{\infty}$ admit subsequences that converge to solutions of an associated problem over all of $M$. 

\begin{lem}\label{BarrierConstructions1}
The PDE problems
\begin{align}
\label{equationforu}
\begin{split}
a_n\Delta_{\gamma}u - au&=-\Lambda_{-} \text{ on } M, \text{ with } u>0,
\end{split}
\end{align}
and 
\begin{align}
\begin{split}
a_n \Delta_{\gamma}v - av&=-\Lambda_{+} \text{ on } M, \text{ with } v>0,
\end{split}
\end{align}
admit solutions $u,v\in C^\infty(M)$ satisfying $0< u\leq v\leq c$, where $c(a_0,\Lambda_{+},c_{+})$ is the same constant appearing in Lemma \ref{SupersolUpperBound}.
\end{lem}
\begin{proof}
Consider the sequence of solutions $\{u_k\}_{k=1}^{\infty}$ and $\{v_k\}_{k=1}^{\infty}$ associated with (\ref{subsolsBG.2}) and (\ref{supersol.2}). Let us now fix some $\Omega_{k'} \subset U_{k'} \subset \Omega_{k'+1}$,   consider $\{u_k\}_{k> k'}$ and appeal to the interior elliptic estimates
\begin{equation}
\label{111120201333}
\begin{aligned}
\Vert u_k\Vert_{W^{2,2}(\Omega_{k'})}
&=C(\Omega_{k'},U_{k'})\left( \| \Lambda_-\|_{L^2(U_{k'} )} +\Vert u_k\Vert_{L^2(U_{k'})} \right).\\
\end{aligned}
\end{equation}
Using Lemma \ref{SupersolUpperBound}, we know that $u_k\leq c$, with $c$ independent of $k$. Then $\Vert u_k\Vert_{L^2(U_{k'})}\leq c\mathrm{Vol}(U_{k'})^{\frac{1}{2}}$, where $k'$ is fixed. Therefore, we find
\begin{align}
\label{111120201333ter}
\Vert u_k \Vert_{W^{2,2}(\Omega_{k'})}&\leq C(\Omega_{k'}, U_{k'}) \text{ for all } k\geq k'.
\end{align} 

We can work similarly on $\{v_k\}_{k\geq k'}$ and obtain: 
\begin{align}\label{111120201333quart}
\Vert v_k \Vert_{W^{2,2}(\Omega_{k'})}&\leq C(\Omega_{k'},U_{k'}) \text{ for all } k\geq k'.
\end{align}

Doing as in \eqref{050720211611}-\eqref{050720211639} we can bootstrap \eqref{111120201333ter} and  \eqref{111120201333quart} into a uniform $W^{2,p}$ estimate on the $\Omega_{k'}$. Thus, the sequence is uniformly bounded in $W^{2,p}(\Omega_{k'})$ and  working as in Step 4 of the proof of Lemma \ref{Xagainstf} we find $u, v \in W^{2,p}_{\mathrm{loc}}(M)$  solutions of $a_n \Delta_{\gamma}u - au=- \Lambda_-$ and $a_n\Delta_{\gamma}v - av= - \Lambda_+$.  Since $a$ and $\Lambda_{\pm}$ are smooth, local regularity yields $u,v \in C^{\infty}(M)$. Furthermore, since $0\leq u_k\leq v_k\leq c$ for all $k$, we find 
\begin{align*}
0\leq u\leq v\leq c.
\end{align*}

We still need to exclude the possibility of $u= 0$. Assume there is some point $p\in M$ such that $u(p)=0$. Since $u\geq 0$, then such point is a minimum of $u$. But, it also holds that
\begin{align}
a_n\Delta_{\gamma}u(p)=au(p) -\Lambda_{-}(p)= - \Lambda_{-}(p)<0,
\end{align}
which contradicts the fact that $p$ is a minimum of $u\geq 0$, and therefore $u>0$.
\end{proof}

\begin{cor}
\label{corSupersolUpperBoundbis}
Under the conditions of Corollary \ref{corSupersolUpperBound}, one has $0< c_- \le u \le v \le c <\infty$.
\end{cor}

The functions $u$ and $v$ constructed above are good starting points for barrier functions. Let us show that one can choose  $0 < \alpha \le \beta $, $\beta \ge 1$, and  $0< \Lambda_- \le \Lambda_+$ such that $\phi_- \doteq \alpha u $ and $\phi_+\doteq  \beta \left( 1 +  v \right)$ are respectively global sub and supersolutions.  Since $\alpha \le \beta$, they will still satisfy  $0<\phi_{-}\leq \phi_{+}$.

\subsubsection{\textbf{Construction of a global supersolution}}

Setting $\phi_+ =  \beta \left(1+ v\right)$, let  $\varphi$  and $Y$ be given such that
$$\begin{aligned} & \phi_- \le \varphi \le \phi_+, \quad  \Delta_{\gamma, \mathrm{conf} } Y - \frac{n-1}{n} \nabla \tau \varphi^{\frac{2n}{n-2}}  - \omega_1\varphi^{2\frac{n + 1}{n-2}} + \omega_2 =0  \end{aligned}$$
Then
$$a_n\Delta_\gamma \phi_+ =  \beta \left(  R_\gamma + \frac{n-1}{n}\tau^2  \right)v - \beta \Lambda_+ = - \beta \left( R_\gamma + \frac{n-1}{n} \tau^2  \right) +\left( R_\gamma + \frac{n-1}{n} \tau^2  \right)  \phi_+  -  \beta \Lambda_+ ,$$ 
and
\begin{equation} \label{280720211356}
\begin{aligned}
\mathcal{H}(\phi_{+})&\doteq a_n\Delta_{\gamma}\phi_{+} - R_{\gamma}\phi_{+} - \frac{n-1}{n}\tau^2\phi_{+}^{\frac{n+2}{n-2}} + \vert \tilde{K}\vert^2_{\gamma}\phi_{+}^{-\frac{3n-2}{n-2}}+ 2\epsilon_1\phi_{+}^{\frac{n+2}{n-2}}  +2\epsilon_2\phi_{+}^{-3}\\
&+2\epsilon_3\phi_{+}^{\frac{n-6}{n-2}},\\
&\leq -\beta\Lambda_+ +  \left(\phi_+- \phi_+^{\frac{n+2}{n-2}}\right) \frac{n-1}{n}\tau^2 + \vert \tilde{K}\vert^2_{\gamma}\phi_{+}^{-\frac{3n-2}{n-2}}+ 2\epsilon_1\phi_{+}^{\frac{n+2}{n-2}}  +2\epsilon_2\phi_{+}^{-3}\\
&+2\epsilon_3\phi_{+}^{\frac{n-6}{n-2}}  \\
&\leq \beta \big( - \Lambda_+ + \left( 1+c \right) \left( 1- \beta^{\frac{4}{n-2} } \right) \frac{n-1}{n} \tau^2+\left\vert \tilde K \right\vert_\gamma^2 \beta^{-\frac{4n-4}{n-2}} + 2 \epsilon_1 (1+c)^{\frac{n+2}{n-2}} \beta^{\frac{4}{n-2}} \\
& + 2 \epsilon_2 \beta^{-4} +2 \epsilon_3 \beta^{-\frac{4}{n-2}}(1+c)^{\frac{n-6}{n-2}} \big)
 \end{aligned}\end{equation}
where $c=\max \left(\sup_M\frac{\Lambda_{+}}{a},c_{+}\right)$ is the same constant appearing in Lemma \ref{SupersolUpperBound}.
Let us now recall that
\begin{align}
\vert \tilde{K}\vert^2_\gamma\leq 2\left(\vert \pounds_{\gamma, \mathrm{conf}}Y \vert^2_{\gamma} + \vert \tilde{U}\vert^2_{\gamma} \right).
\end{align}
To obtain that $\phi_+$ is a global supersolution, i.e. that $\mathcal{H}(\phi_+) \le 0$, we need an a priori bound on solutions $Y$ of the momentum constraint with $\varphi \le \phi_+$.

\begin{proposition}
Let $(M^n,\gamma)$ be a smooth Riemannian  manifold of bounded geometry  such that $\lambda_{1,\mathrm{conf}} (M) >0$. Let $p >n$ and assume that:
\begin{equation} \nabla \tau , \omega_1, \omega_2 \in L^2 (M) \cap L^p(M). \end{equation}
Then, for all $ \varphi \le \phi_+ \le c $,  any $L^2\cap W^{1,p}_{\mathrm{loc}}$ solution of the momentum constraint 
$$\Delta_{\gamma, \mathrm{conf} } Y - \frac{n-1}{n} \nabla \tau \varphi^{\frac{2n}{n-2}} - \omega_1\varphi^{2\frac{n + 1}{n-2}} + \omega_2 =0$$
satisfies:
\begin{equation}
\label{280720211355}
\begin{aligned}\!\!\!\!\vert  \pounds_{\gamma,\mathrm{conf}}&Y \vert^2_{\gamma} \le \frac{C(n ,M,\gamma)}{\lambda_{1, \gamma, \mathrm{conf}} } \left[ \left( \sum_{j=0}^{j_{\mathrm{max}}} \| \nabla \tau \|_{L^{p_j}(M)}+ \| \nabla \tau \|_{L^{p}(M)} \right) (1+c)^{\frac{2n}{n-2} } \right. \\ & \left. + \left( \sum_{j=0}^{j_{\mathrm{max}}} \| \omega_1 \|_{L^{p_j}(M)} + \| \omega_1 \|_{L^{p}(M)}  \right)(1+c)^{2\frac{n+1}{n-2} }+ \left( \sum_{j=0}^{j_{\mathrm{max}}} \| \omega_2 \|_{L^{p_j}(M)} + \| \omega_2 \|_{L^{p}(M)}  \right) \right],
\end{aligned}
\end{equation}
where $(p_j)_{j \in \mathbb{N}\cup \{0\}}$ is the sequence defined by induction as:
$p_0=2, \, p_{j+1} = \frac{n p_j}{n-2 p_j}$, and $j_{\mathrm{max}}$ is the first integer for which $\frac{n p_{j_{\mathrm{max}}-1}}{n-2p_{j_{\mathrm{max}}-1}} \ge \frac{n}{2},$ and $p_{j_{\mathrm{max}}} = \min \left( p, \frac{n p_{j_{\mathrm{max}}-1}}{n-2p_{j_{\mathrm{max}}-1}} \right)$.
\end{proposition}
\begin{remark}
Since for any $2 \le q \le p$ we may interpolate (see, for instance, \cite[Chapter 4, p 93]{BrezisBook}):
\begin{equation}
\begin{aligned}
\label{020820211138}
\| V\|_{L^q(M)}\le  \| V\||^{\alpha}_{L^2(M)}|| V||^{1-\alpha}_{L^p(M)}\le \max \left( \| V \|_{L^2(M)} , \| V \|_{L^p(M)} \right), \text{ where }  \frac{1}{q}=\frac{\alpha}{2}+\frac{1-\alpha}{p},
\end{aligned}
\end{equation}
the right-hand side of estimate \eqref{280720211355} is then finite. 
\end{remark}
\begin{proof}
We will first use Lemma \ref{Pigola} to control the $L^2$ norm of $Y$ using the first eigenvalue of the conformal Killing Laplacian, and then proceed with a bootstrap in bounded geometry to prove the estimate. It should be noticed that elliptic estimates ensure that $Y \in W^{1,p}_{\mathrm{loc}} \implies Y \in W^{2,p}_{\mathrm{loc}}$.

From Lemma \ref{Pigola}, we can deduce that our solution $Y$ satisfies: 
\begin{equation}\label{Pigola.2}
\begin{aligned}
\Vert \pounds_{\gamma,\mathrm{conf}} Y \Vert^2_{L^2(M)}
&\le C  \|  Y \|_{L^2(M)}  \left(  \| \nabla \tau \|_{L^2(M)} (1+c)^{\frac{2n}{n-2} }+ \|\omega_1 \|_{L^2(M)} (1+c)^{2 \frac{n+1}{n-1} } + \| \omega_2 \|_{L^2(M)} \right).
\end{aligned}
\end{equation}
We can then proceed as in the proof of Lemma \ref{050720211748} and write:
\begin{equation}
\label{050720211551bis}
\left\|  Y  \right\|_{L^2(M )}^2 \le \frac{1}{ \lambda_{1, \gamma, \mathrm{conf} } }  \left\| \pounds_{\gamma, \mathrm{conf} } Y  \right\|_{L^2(M)}^2,
\end{equation}
which with \eqref{Pigola.2} yields:
\begin{equation}
\label{280720211228}
\|  Y \|_{L^2(M)}  \le \frac{C(n)}{\lambda_{1, \gamma, \mathrm{conf} } }\left(  \| \nabla \tau \|_{L^2(M)} (1+c)^{\frac{2n}{n-2} }+ \|\omega_1 \|_{L^2(M)} (1+c)^{2 \frac{n+1}{n-1} } + \| \omega_2 \|_{L^2(M)} \right).
\end{equation}

We will now prove by induction that  for all $i \le j_{\mathrm{max}}$, one has
\begin{equation} \label{050720211612bis} \begin{aligned} \left\| Y \right\|_{W^{2,p_i} (M) } &\le \frac{C(n, \gamma,M )}{\lambda_{1, \gamma, \mathrm{conf}} } \left[ \left( \sum_{j=0}^{i} \| \nabla \tau \|_{L^{p_j}(M)} \right) (1+c)^{\frac{2n}{n-2} } +\right.  \\ & \left. \left( \sum_{j=0}^{i} \| \omega_1 \|_{L^{p_j}(M)} \right)(1+c)^{2\frac{n+1}{n-1} } +   \left( \sum_{j=0}^{i} \| \omega_2 \|_{L^{p_j}(M)} \right) \right]. \end{aligned}\end{equation}
\begin{itemize}
\item when $i=0$, $p_i=2$, and injecting the $L^2$ estimate \eqref{280720211228} into  Shubin's elliptic regularity estimates in \emph{bounded geometry} (see Lemma \ref{regularity.1}) this ensures that:  
$$\|  Y \|_{W^{2,2}(M)}  \le \frac{C(n, M, \gamma)}{\lambda_{1, \gamma, \mathrm{conf} } } \left(  \| \nabla \tau \|_{L^2(M)} (1+c)^{\frac{2n}{n-2} }+ \|\omega_1 \|_{L^2(M)} (1+c)^{2 \frac{n+1}{n-1} } + \| \omega_2 \|_{L^2(M)} \right).$$
\item
Assuming the result for $i< j_{\mathrm{max}}$, the Sobolev embeddings $W^{2,p_j}\subset L^{p_{j+1}}$ (true in \emph{bounded geometry}, see  \cite[Corollary 3.19]{MR1481970}, or \cite[last paragraph p. 68]{Shubin}) and the inductive hypothesis ensures that $Y \in L^{p_{i+1}}(M)$ with:
\begin{equation}
\label{280720211312}
\begin{aligned}
\|Y \|_{L^{p_{i+1}}(M)} &\le \frac{C(n, M, \gamma )}{\lambda_{1, \gamma, \mathrm{conf}} } \left[ \left( \sum_{j=0}^{i} \| \nabla \tau \|_{L^{p_j}(M)} \right) (1+c)^{\frac{2n}{n-2}}  +  \right. \\ & \left.  \left( \sum_{j=0}^{i} \| \omega_1 \|_{L^{p_j}(M)} \right)(1+c)^{2\frac{n+1}{n-1} } +   \left( \sum_{j=0}^{i} \| \omega_2 \|_{L^{p_j}(M)} \right)  \right],
\end{aligned}
\end{equation}
Shubin's elliptic regularity estimates in bounded geometry then yield:
$$\|  Y \|_{W^{2,p_{i+1}}(M)}  \le C(n, M, \gamma) \left(  \|  \Delta_{\gamma, \mathrm{conf} } Y \|_{L^{p_{i+1} } (M)}  + \| Y  \|_{L^{p_{i+1}}(M) } \right).$$
Injecting the estimate on $\Delta_{\gamma, \mathrm{conf} } Y$ and \eqref{280720211312} into the above yields the $p_{i+1}$ estimate.
\end{itemize}
 Estimate \eqref{050720211612bis} then stands true for $j_{\mathrm{max}}$.  Then: 
\begin{itemize}
\item
If $p_{j_{\mathrm{max}}} > \frac{n}{2}$, Sobolev embeddings $W^{2,p_j}\subset C^{0}$  ensure that:
\begin{equation}
\label{280720211312bis}
\begin{aligned}
\!\!\!\!\!\!\!\!\!\!\!\!\!\!\!\!\|Y \|_{L^{\infty}(M)} &\le \frac{C(n, M ,\gamma )}{\lambda_{1, \gamma, \mathrm{conf}} } \left[ \left( \sum_{j=0}^{j_{\mathrm{max}}} \| \nabla \tau \|_{L^{p_j}(M)} \right) (1+c)^{\frac{2n}{n-2}}   +  \left( \sum_{j=0}^{j_{\mathrm{max}}} \| \omega_1 \|_{L^{p_j}(M)} \right)(1+c)^{2\frac{n+1}{n-1} } +    \sum_{j=0}^{j_{\mathrm{max}}} \| \omega_2 \|_{L^{p_j}(M)}   \right].
\end{aligned}
\end{equation}
Since in addition $Y \in L^2(M)$, one can conclude that $Y \in L^p(M)$:
\begin{align}\label{280720211330}
\begin{aligned}
\| Y \|_{L^p(M)} &\le \left( \int_M \vert Y \vert^{p-2} \vert Y \vert^2 dV_\gamma \right)^{\frac{1}{p} } \le \| Y \|_{L^\infty(M) }^{1- \frac{2}{p} } \| Y \|_{L^2(M)}^{\frac{2}{p}} \\
& \le \frac{C(n, M ,\gamma )}{\lambda_{1, \gamma, \mathrm{conf}} }  \left[ \left( \sum_{j=0}^{j_{\mathrm{max}}} \| \nabla \tau \|_{L^{p_j}(M)} \right) (1+c)^{\frac{2n}{n-2}}   +   \left( \sum_{j=0}^{j_{\mathrm{max}}} \| \omega_1 \|_{L^{p_j}(M)} \right)(1+c)^{2\frac{n+1}{n-1} } + \right.\\ & \left.   \left( \sum_{j=0}^{j_{\mathrm{max}}} \| \omega_2 \|_{L^{p_j}(M)} \right)  \right],
\end{aligned}
\end{align}
where we used \eqref{280720211228}, \eqref{280720211312bis} and the non optimal but simplifying estimate:
$$\begin{aligned}  &\| \nabla \tau \|_{L^2(M)} (1+c)^{\frac{2n}{n-2} }+ \|\omega_1 \|_{L^2(M)} (1+c)^{2 \frac{n+1}{n-1} } + \| \omega_2 \|_{L^2(M)}  \le \left( \sum_{j=0}^{j_{\mathrm{max}}} \| \nabla \tau \|_{L^{p_j}(M)} \right) (1+c)^{\frac{2n}{n-2}}   +  \\& \left( \sum_{j=0}^{j_{\mathrm{max}}} \| \omega_1 \|_{L^{p_j}(M)} \right)(1+c)^{2\frac{n+1}{n-1} } +   \left( \sum_{j=0}^{j_{\mathrm{max}}} \| \omega_2 \|_{L^{p_j}(M)} \right). \end{aligned} $$
Applying once more Shubin's elliptic estimates yields a $W^{2,p}$ estimate which we translate into the proper $L^\infty$ control on $\pounds_{\gamma, \mathrm{conf} } Y$ thanks to Sobolev embeddings, which concludes the proof.
\item 
If $p_{j_{\mathrm{max}}} = \frac{n}{2}$, let us consider the $j_{\mathrm{max}} -1$ estimate, and use the Sobolev embedding $W^{2,p_{j_{\mathrm{max}}-1}}\subset L^{p_{j_{\mathrm{max}}}}$. Using estimate \eqref{020820211138} (and dealing with the maximum by injecting the $L^2$ estimate as in \eqref{280720211330}), one can obtain   a $L^{p_{j_{\mathrm{max}}} -\varepsilon}$ estimate for a small $\varepsilon>0$.  The induction proof ensures that we obtain a control in $W^{2, \tilde p_{j_{\mathrm{max}}}}$  for $\tilde p_{j_{\mathrm{max}}} = \min \left( p,  \frac{n ( p_{j_{\mathrm{max}}}-\varepsilon)}{n-2 ( p_{j_{\mathrm{max}}}-\varepsilon)} \right)>  p_{j_{\mathrm{max}}} \ge \frac{n}{2}$. We are then back in the first case, which concludes the proof. We do not change the notations on the right-hand side of the inequality for simplicity. Since all the $L^{p_j}$ and $L^{\tilde p_j}$ estimates are obtained  thanks  to \eqref{020820211138}, this  bears no impact on the final result.
\end{itemize}
\end{proof}

\begin{remark}
The bounded geometry hypothesis intervenes above to grant Sobolev embeddings and elliptic estimates. Other hypotheses preserving these two properties are likely to lead to similar constructions.
\end{remark}

To lighten notations, let us write 
\begin{equation}
\label{280720211347}
\begin{aligned}
M_{\nabla \tau} &= 2 \frac{C(n ,M,\gamma)}{\lambda_{1, \gamma, \mathrm{conf}} }  \left( \sum_{j=0}^{j_{\mathrm{max}}} \| \nabla \tau \|_{L^{p_j}(M)}+ \| \nabla \tau \|_{L^{p}(M)} \right) \\
M_{\omega_1} &= 2 \frac{C(n ,M,\gamma)}{\lambda_{1, \gamma, \mathrm{conf}} } \left( \sum_{j=0}^{j_{\mathrm{max}}} \| \omega_1 \|_{L^{p_j}(M)}+ \| \omega_1 \|_{L^{p}(M)} \right) \\
M_{\omega_2} &= 2 \frac{C(n ,M,\gamma)}{\lambda_{1, \gamma, \mathrm{conf}} } \left( \sum_{j=0}^{j_{\mathrm{max}}} \| \omega_2 \|_{L^{p_j}(M)}+ \| \omega_2 \|_{L^{p}(M)} \right)
\end{aligned}
\end{equation}

Injecting \eqref{280720211355} and $\beta \ge 1$ into  \eqref{280720211356} yields:
\begin{align*}
\mathcal{H}(\phi_{+}) &\leq \beta \left( - \Lambda_+ + \left( 1+c \right) \left( 1- \beta^{\frac{4}{n-2} } \right) \frac{n-1}{n} \tau^2+ (M_{\nabla \tau} + M_{\omega_1} + M_{\omega_2}) \left(1+c \right)^{\frac{4n}{n-2}} \beta^{ \frac{4}{n-2}} \right. \\ & \left.+2 \left\vert  U \right\vert^2  \beta^{- \frac{4n-4}{n-2} }+ 2 \epsilon_1 (1+c)^{\frac{n+2}{n-2}} \beta^{\frac{4}{n-2}}  + 2 \epsilon_2 \beta^{-4} +2 \epsilon_3 \beta^{-\frac{4}{n-2}} (1+c)^{\frac{n-6}{n-2} } \right).
\end{align*}

At this point it is worth noticing that $c = \max \left( \sup \frac{\Lambda_+}{a}, c_+ \right)$ depends on both parameters $\Lambda_+$ and $a$ (and thus $\tau$). To make $c$ independant of $\tau$, we will choose  $\Lambda_+ =a$. Setting $m = 1+c \ge 1$ implies:
$$\begin{aligned}
\mathcal{H}(\phi_{+}) &\leq \beta \left( -a + m \left( 1- \beta^{\frac{4}{n-2} } \right) \frac{n-1}{n} \tau^2+ (M_{\nabla \tau} + M_{\omega_1} + M_{\omega_2}) m^{\frac{4n}{n-2}} \beta^{ \frac{4}{n-2}} \right. \\ & \left.+2  \left\vert  U \right\vert^2  \beta^{- \frac{4n-4}{n-2} }+ 2 \epsilon_1 m^{\frac{n+2}{n-2}} \beta^{\frac{4}{n-2}}  + 2 \epsilon_2 \beta^{-4} +2 \epsilon_3 \beta^{-\frac{4}{n-2}} m^{\frac{n-6}{n-2} } \right).
\end{aligned}$$

Taking $\beta=1$ then yields:
$$\begin{aligned}
\mathcal{H}(\phi_{+}) 
&\le - \frac{n-1}{n} \tau^2 -  R_\gamma+  (M_{\nabla \tau} + M_{\omega_1} + M_{\omega_2}) m^{\frac{4n}{n-2}}  +2  \left\vert  U \right\vert^2  + 2 \epsilon_1 m^{\frac{n+2}{n-2}}  + 2 \epsilon_2  +2 \epsilon_3 m^{\frac{n-6}{n-2} }.
\end{aligned}$$
Then, there exists  a constant $C(n,M,\gamma,  \lambda_{1, \gamma, \mathrm{conf} },c_+)$ such that if
\begin{equation}
\label{280720211430}
\begin{aligned}
&\left\vert R_\gamma \right\vert + \sum_{j=0}^{j_{\mathrm{max}}} \| \nabla \tau \|_{L^{p_s}(M)}+ \| \nabla \tau \|_{L^{p}(M)}   + \sum_{j=0}^{j_{\mathrm{max}}} \| \omega_1 \|_{L^{p_s}(M)}+ \| \omega_1 \|_{L^{p}(M)}  + \\&   \sum_{j=0}^{j_{\mathrm{max}}} \| \omega_2 \|_{L^{p_s}(M)}+ \| \omega_2 \|_{L^{p}(M)} + \vert  U \vert +  \epsilon_1 + \epsilon_2 + \epsilon_3  \le C \tau^2,
\end{aligned}
\end{equation}
 $\phi_+$ is a global supersolution.

Using \eqref{020820211138}, one can simplify: $\| \nabla \tau \|_{L^{p_s}(M)}\le \max \left( \| \nabla \tau \|_{L^{2}(M)}, \| \nabla \tau \|_{L^{p}(M)}\right)$ and similarly $\| \omega_i \|_{L^{p_s}(M)} \le  \max \left( \| \omega_i \|_{L^{2}(M)}, \| \omega_i \|_{L^{p}(M)} \right)$, as featured in the final theorems. We have thus obtained:
\begin{lem}
\label{phi+isasupersolution}
Let $(M^n, \gamma)$ be a smooth Riemannian manifold of bounded geometry such that $\lambda_{1, \mathrm{conf}}>0$. There exists a constant $C(n,M,\gamma,  \lambda_{1, \gamma, \mathrm{conf} },c_+)$ such that if
\begin{equation}
\label{280720211430bis}
\begin{aligned}
&\left\vert R_\gamma \right\vert + \max \left( \| \nabla \tau \|_{L^{2}(M)}, \| \nabla \tau \|_{L^{p}(M)}\right)  + \max \left( \| \omega_1 \|_{L^{2}(M)}, \| \omega_1 \|_{L^{p}(M)} \right) + \\& \max \left( \| \omega_2 \|_{L^{2}(M)}, \| \omega_2 \|_{L^{p}(M)} \right) +\vert  U \vert+ \epsilon_1 + \epsilon_2 + \epsilon_3  \le C \tau^2,
\end{aligned}
\end{equation}
then $\phi_+ =1+v$, with $v$ as in Lemma \ref{BarrierConstructions1}, taken with $\Lambda_+=a \ge a_0>0$,  is a global supersolution to \eqref{EMSystem-CompleteManifolds}. 
\end{lem}


\subsubsection{\textbf{Construction of a global subsolution}} 

Let us now construct a global subsolution to be paired with the above $\phi_{+}$:

\begin{lem}\label{phi-isasubsolution}
Let $(M^n, \gamma)$ be a smooth Riemannian manifold of bounded geometry such that $\lambda_{1, \mathrm{conf}}>0$. Then:
\begin{enumerate}
\item if $n \le 6$ and $\epsilon_2 + \epsilon_3 >0$, there exists $0< \alpha<1$ such that $\phi_- = \alpha u$ is a global subsolution to  \eqref{EMSystem-CompleteManifolds}.
\item if $n>6$ and $\epsilon_2>0$, there exists $0< \alpha<1$ such that $\phi_- = \alpha u$ is a global subsolution to  \eqref{EMSystem-CompleteManifolds}.
\end{enumerate} 
Moreover,  if in any item above there is a uniform
lower bound $\varepsilon_0 > 0$, there is a choice $u= c_-<1$ in \eqref{equationforu} with  $\Lambda_- =ac_-$ which  gives a constant global subsolution $\phi_-$ compatible with the  supersolution $\phi_+$ given in Lemma \ref{phi+isasupersolution}.
\end{lem}

\begin{proof}
If the hypotheses of item 1 hold, let us then start by choosing $\alpha>0$ small enough so that $0<\phi_{-}\leq 1$, and thus $\phi^{-3}_{-},\phi^{\frac{n-6}{n-2}}_{-}\geq 1$. Then,
\begin{align*}
\mathcal{H}(\phi_-) &=a_n \Delta_\gamma \phi_{-} -  R_\gamma\phi_{-} - \frac{n-1}{n} \tau^2 \phi_-^{\frac{n+2}{n-2} } +  \left\vert \tilde K(X) \right\vert^2 \phi_-^{-\frac{3n-2}{n-2}} + 2\epsilon_1  \phi_-^{\frac{n+2}{n-2}} + 2  \epsilon_2 \phi_-^{-3} + 2  \epsilon_3 \phi_-^{\frac{n-6}{n-2} } \\ 
&\ge \alpha \frac{n-1}{n} \tau^2 u - \alpha\Lambda_{-}  - \frac{n-1}{n} \tau^2 \alpha^{\frac{n+2}{n-2} } u^{\frac{n+2}{n-2} } +  2 \epsilon_2  + 2  \epsilon_3  \\
&\ge \alpha u \frac{n-1}{n} \tau^2 \left( 1- \left( \alpha u \right)^{\frac{4}{n-2} }\right) +  2  \epsilon_2  + 2  \epsilon_3 - \alpha\Lambda_{-}
\end{align*}
From our choice of $\alpha$ we find that $1- \left( \alpha u \right)^{\frac{4}{n-2}}=1- \phi_{-}^{\frac{4}{n-2}}\geq 0$. Then, if $\epsilon_2+\epsilon_3>0$ on $M$, then there is a choice of $\Lambda_{-}>0$ satisfying
\begin{align}
\label{290720211028}
2  \epsilon_2  + 2  \epsilon_3 \geq \alpha\Lambda_{-}>0.
\end{align}

After fixing such a choice we have that $\mathcal{H}(\phi_-)\geq 0$, and therefore $\phi_{-}$ is  a global subsolution.

When the hypotheses of item 2 hold, then the reasoning is very similar. The difference lies in the estimates for the $\phi^{\frac{n-6}{n-2}}_{-} \ge 0$ terms. The condition thus falls entirely on $\epsilon_2$ and one can find an admissible $\Lambda_-$ and $\alpha$ if and only if 
\begin{equation} \label{290720211028bis}\epsilon_2>0. 
\end{equation}

Finally, in any of the above cases, if $\Lambda_{-}\doteq c_{-}a\leq \Lambda_{+}=a$ (as chosen when constructing $\phi_{+}$), then $u=c_{-}$ satisfies \eqref{equationforu} and $\phi_{-}\doteq \alpha c_{-}\leq 1\leq \phi_{+}$. Finally if $\epsilon_0>0$ exists providing a positive lower bound for the energy densities, then the above analysis and a sufficiently small choice of $\alpha>0$ grant $\phi_{-}$ is a global subsolution.
\end{proof}

\begin{remark}
Should one keep the cosmological constant in \eqref{ECE}, it would induce another term of critical exponent. For instance, in a physical scenario with $\Lambda>0$, this cosmological term can be grouped with $\epsilon_1$, and one can apply the analysis of Lemma \ref{phi-isasubsolution}.
\end{remark}

\subsection{Vacuum global subsolutions}

A mathematical issue with the construction of barrier functions in Section \ref{subsecbarrfunct} is that \eqref{290720211028} and \eqref{290720211028bis} do not allow vacuum solutions. However restricting, this is physically reasonable if our main physical motivation for this analysis are cosmological open universes\footnote{This is specially the case for electromagnetic sources present in the cosmic microwave background radiation.}.  In any case, to counterbalance this, one can notice that  the supersolution does allow for $\epsilon_1, \epsilon_2, \epsilon_3 =0$, and is bounded from below (in fact with our construction $\phi_+ \ge 1$).   A subsolution $\phi_- \le 1$ built with another method would then yield  compatible barrier functions.

\subsubsection*{Reminders on Yamabe-type equations}
We now build a bounded subsolution following ideas from P. Mastrolia, M. Rigoli and A. Setti's \cite{MR2962687} on a priori estimates and existence conditions for Yamabe-type equations:
\begin{equation}
\label{Yamabetypee}
\Delta u + a(x) u- b(x)  u^\sigma = 0, \text{ with } \sigma >1.
\end{equation} 
We will use two of their theorems that we assemble in the result below, for convenience:

\begin{theorem}
\label{theoartisanal1}
Let $(M^n, g)$ be a smooth complete Riemannian manifold.
Let $a(x)$, $b(x) \in C^{0,\alpha}_{ \mathrm{loc}} (M)$ for some $0< \alpha \le 1$. Assume that:
\begin{equation} \label{h1} b(x) \ge 0, \, b(x)>0 \text{ outside a compact set},\end{equation}
 \begin{equation} \label{h2} \lambda_1^L (B_0) >0  \text{ where } B_0 = \{ x \in M \, : \, b(x)=0 \} \text{ and } L= \Delta_g +a(x), \end{equation}
\begin{equation} \label{h3} a(x)\le A, \quad  B \le b(x)\end{equation}
\begin{equation} \label{h4}\mathrm{Ric} \ge -(n-1) H^2 (1+r^2)\end{equation}
\begin{equation} \label{h5} \lambda_1^L (M) <0.\end{equation}
where $A,B$ are two positive constants, $H$ is a constant, $r(x)=d_g(p,x)$ denotes the distance function to a fixed point $p\in M$, and $\lambda_1^L(B_0)$ denotes the first eigenvalue of $L$ over $B_0$ as defined by \eqref{definitionvapsurOmega}\footnote{See \cite[Theorem 4.4, Theorem 6.7]{MR2962687} for further details}.
Then the equation
\eqref{Yamabetypee}
possesses a positive solution satisfying 
$$0<u(x) \le C,$$ for a constant $C>0$.
\end{theorem}
\begin{proof}
Hypotheses \eqref{h1}, \eqref{h2}, \eqref{h5} ensure \cite[Theorem 6.7]{MR2962687} applies: there exist positive solutions to \eqref{Yamabetypee}. On the other hand hypotheses \eqref{h4}  and  \eqref{h3} mean that we can apply  \cite[Theorem 4.4]{MR2962687} with $\alpha = \beta =0$ to obtain the uniform upper bound.
\end{proof}

\begin{remark}
\label{290720211158}
A quick way to ensure hypothesis \eqref{h2} is to impose  $b>0$ on the whole of $M$, which indeed implies $B_0=\emptyset$ and thus by definition $\lambda_1^L (B_0) = + \infty >0$ (we refer the reader  to the discussion  in   \cite[p.159]{MR2962687} as well as the proofs of Theorems 6.2 and 6.7 for more details).
\end{remark}

We now look for a global subsolution of \eqref{systoncomplphiX} from a bounded solution $\varphi_-$ of an equation  of Yamabe type, as in \eqref{Yamabetypee},
with $a$ and $b$ to be chosen according to our needs.  We propose our subsolutions to be of the form $\phi_- = \kappa \varphi_-$, such that 
$$\begin{aligned}
\mathcal{H}( \phi_-)&\doteq a_n \Delta_\gamma \phi_-  - R_\gamma \phi_- + \left\vert \tilde K (X)\right\vert^2 \phi_-^{- \frac{3n-2}{n-2}} -\frac{n-1}{n}  \tau^2 \phi_-^{\frac{n+2}{n-2}} + 2\epsilon_1 \phi_-^{\frac{n+2}{n-2}} +2 \epsilon_2 \phi_-^{-3}+2 \epsilon_3 \phi_-^{\frac{n-6}{n-2}} \ge 0.
\end{aligned}$$

Setting $a= -  R_\gamma$, $b =\frac{n-1}{n}\tau^2$ in (\ref{Yamabetypee}) and if $\varphi_-$ is a positive solution of 
\begin{equation} \label{solRtau} a_n \Delta_\gamma \varphi_- -R_\gamma \varphi_- - \frac{n-1}{n} \tau^2 \varphi_-^{\frac{n+2}{n-2} }=0,
\end{equation}
then for any $X:$
$$\begin{aligned}  \mathcal{H}(\phi_-) &\ge a_n \Delta_\gamma \phi_- - R_\gamma \phi_- -\frac{n-1}{n}\tau^2 \phi_-^{\frac{n+2}{n-2} } \\
&\ge a_n \kappa \Delta_\gamma \varphi_- -  R_\gamma \kappa \varphi_- - \frac{n-1}{n} \tau^2 \kappa^{ \frac{n+2}{n-2} } \varphi_-^{\frac{n+2}{n-2}} \\
&\ge \frac{n-1}{n}\tau^2 \varphi_-^{\frac{n+2}{n-2} } \kappa \left( 1- \kappa^{\frac{4}{n-2}} \right) \ge 0
\end{aligned}$$
for $\kappa \le 1$. Thus if there exists a positive bounded solution to \eqref{solRtau},  one can find a strictly positive subsolution to \eqref{systoncomplphiX} compatible with the supersolution $\phi_+$ defined in Lemma \ref{phi+isasupersolution}.

\begin{theorem}
\label{theosetup1}
Let $(M^n, \gamma)$ be a smooth, complete manifold, $n\ge 3$. We assume that 
$\tau \in C^{0,  \alpha}_{ \mathrm{loc} } (M)$ for some $0< \alpha \le 1$. We further assume 
\begin{equation}\label{hl1}
\begin{aligned}
&\mathrm{Ric} \ge -(n-1) H^2 (1+r^2),  \quad
R_\gamma \ge -A, 
\end{aligned}
\end{equation}
\begin{equation}
\label{hl1bis}
\vert \tau \vert \ge B>0 \text{ outside a compact set,} 
\end{equation}
\begin{equation}
\label{hl2}
\begin{aligned}
&\lambda^{ -a_n \Delta_\gamma +  R_\gamma}_1(B_0) >0, \quad 
\lambda^{-a_n \Delta_\gamma +  R_\gamma}_1 (M)< 0,
\end{aligned}
\end{equation}
$\text{ where } B_0 = \{ x \in M \, : \, \tau(x) = 0 \}.$
Then  for any $m>0$  the Lichnerowicz equation  \eqref{systoncomplphiX} admits a strictly positive subsolution $\phi_-$ satisfying
$\phi_- \le m$.
\end{theorem}
\begin{proof}
We apply  Theorem \ref{theoartisanal1} with the preceding analysis.
\end{proof}

This result is enough for our goals, but to highlight the variety of constructions of barrier functions, we offer an alternative choice for $a$ and $b$ which will allow one to consider the case $R = 0$ and $\tau = 0$ on a compact $K \subset M$ with $\tau^2>0$ on $M\backslash K$ and thus to shift the constraint onto the magnetic datum $\epsilon_3$.

Assume $R_\gamma\le 0$, $a= 2  \epsilon_3$, $b =\frac{n-1}{n} \tau^2$ and that $\varphi_-$ is a positive solution of 
\begin{equation} \label{solRtaubis} a_n \Delta_\gamma \varphi_-+2 \epsilon_3 \varphi_- - \frac{n-1}{n}\tau^2 \varphi_-^{\frac{n+2}{n-2} }=0,
\end{equation}
and let $\phi_- \doteq \kappa \varphi_-$.
Then, if in addition $R_\gamma\le 0$
$$\begin{aligned} 
\mathcal{H}(\phi_-) &\ge a_n\Delta_\gamma \phi_- +2 \epsilon_3\phi_-^{\frac{n-6}{n-2}}- \frac{n-1}{n} \tau^2 \phi_-^{\frac{n+2}{n-2} } \\
&\ge \frac{n-1}{n}\tau^2 \varphi_-^{\frac{n+2}{n-2} } \kappa \left( 1- \kappa^{\frac{4}{n-2}} \right)  +2 \epsilon_3 \kappa  \varphi_-\left(\left(\kappa \varphi_-\right)^{-\frac{4}{n-2}} -1 \right) \ge 0
\end{aligned}$$
for $\kappa \le \min \left(1, \frac{1}{\sup \varphi_-} \right)$. Thus, if there exists a positive solution to \eqref{solRtaubis} which is bounded from above,  one can find a strictly positive subsolution to \eqref{systoncomplphiX}, as small as required.  Applying Theorem \ref{theoartisanal1} then yields an equivalent to Theorem \ref{theosetup1}, replacing the bound on the scalar curvature, and the operator $-a_n\Delta_\gamma+ R_\gamma$ by $-a_n\Delta_\gamma +2\epsilon_3$ in the spectral hypotheses, and with an additional control $\epsilon_3\le A$ for a positive constant $A$.

\subsection{Existence results}

We can now build global barrier functions and obtain existence results for equation \eqref{systoncomplphiX} on a complete manifold of bounded geometry. Our first one results entirely from the construction in Section \ref{subsecbarrfunct}.

\begin{theorem}
Let $(M^n, \gamma)$ be a smooth complete Riemannian manifold of bounded geometry,  $n\ge3$, $p>n$. We make the following assumptions:
\begin{equation} \label{HYP1} R_\gamma, \epsilon_1, \epsilon_2, \epsilon_3, \vert U \vert^2,  \tau^2 \in L^p_{\mathrm{loc}}(M) \text{ and }  \omega, \nabla \tau \in L^{2}(M) \cap L^p(M). \end{equation}
\begin{equation}
\label{HYP2}
 \lambda_{1,\mathrm{conf}} >0,
\end{equation}
\begin{equation}
\label{HYP3}
a \doteq R_\gamma + \frac{n-1}{n} \tau^2 \in L^\infty (M), \,  \, a \ge a_0 >0. 
\end{equation}
Assume further that:
\begin{equation}
\label{HYP4}
\left\{ \begin{aligned}
&\epsilon_2 + \epsilon_3>0  \text{ if } n \le 6\\
& \epsilon_2 >0 \text{ if } n >6.
\end{aligned} \right.
\end{equation}

Then, there exists $C(n,M, \gamma, \lambda_{1, \mathrm{conf}})$  such that if
\begin{equation}
\label{HYP5}
\begin{aligned}
\left\vert  R_\gamma \right\vert +& \max \left( \| d\tau \|_{L^{2}(M)}, \| d\tau \|_{L^{p}(M)}\right)  + \max \left( \| \omega_1 \|_{L^{2}(M)},\| \omega_1 \|_{L^{p}(M)}\right) \\
&+ \max\left(\| \omega_2 \|_{L^{2}(M)}, \| \omega_2 \|_{L^{p}(M)} \right) +\vert U \vert+ \epsilon_1 + \epsilon_2 + \epsilon_3  \le C \tau^2,
\end{aligned}
\end{equation}
then  \eqref{systoncomplphiX} admits a $W^{2,p}_{\mathrm{loc}}$ solution. Moreover, if in either condition of \eqref{HYP4} there is a uniform lower bound $\varepsilon_0>0$, then the physical metric $g= \phi^{\frac{4}{n-2}} \gamma$ is also complete.
\end{theorem}
\begin{proof}
Hypotheses \eqref{HYP1}, \eqref{HYP2}  allow one to use Theorem \ref{050720211748bis} and conclude that if there exists global barrier functions, system  \eqref{systoncomplphiX} admits a $W^{2,p}_{\mathrm{loc}}$ solution. Then, hypothesis \eqref{HYP5} \emph{in bounded geometry} corresponds to our construction of a global supersolution with $\Lambda_+ = a$ (and $c_+$ arbitrary) (see Lemma \ref{phi+isasupersolution}), while \eqref{HYP4} allows for the construction of a compatible subsolution (see Lemma \ref{phi-isasubsolution}) with  $\Lambda_- \le \frac{1}{2} \left( \epsilon_2 + \epsilon_3 \right)$ or $\frac{1}{2} \epsilon_3$ (while \eqref{HYP5} ensures that $\Lambda_- \le \Lambda_+$ if $C(n,M, \gamma, \lambda_{1, \mathrm{conf}})\ge 1$).

In the end, the global barrier functions ensure the existence of a solution to the constraint equations.

If, in addition there is a uniform lower bound $\varepsilon_0>0$ in \eqref{HYP4}, $\Lambda_- = c_-a$ is admissible for $c_-$ small enough  by \eqref{290720211028}, and thus thanks to Corollary \ref{corSupersolUpperBoundbis} the conformal factor $\phi$ is uniformly bounded from below and above by strictly positive constants. Then 
the distance function for $g$ will satisfy $c_1d_{\gamma}(p,q)\leq d_{g}(p,q)\leq c_2d_{\gamma}(p,q)$ showing that $g$ is also complete.
\end{proof}

We can deal with the vacuum case under different hypotheses:
\begin{theorem}
Let $(M, \gamma)$ be a smooth complete Riemannian manifold of bounded geometry, let $n\ge 3$ be its dimension and $p>n$. We make the following assumptions:
\begin{equation} \label{HYP1b} R_\gamma, \tau \in C^{0,\alpha}_{\mathrm{loc}} (M) \cap  L^p_{\mathrm{loc}}(M), \epsilon_1, \epsilon_2, \epsilon_3,\vert U\vert^2,  \tau^2 \in L^p_{\mathrm{loc}}(M) \text{ and  } \omega, \nabla \tau \in L^{2}(M) \cap L^p(M). \end{equation}
\begin{equation}
\label{HYP2b}
 \lambda_{1,\mathrm{conf}} >0,
\end{equation}
\begin{equation}
\label{HYP3b}
a \doteq R_\gamma + \frac{n-1}{n} \tau^2 \in L^\infty (M), \,  \, a \ge a_0 >0. 
\end{equation}
Assume further that:
\begin{equation}
\label{HYP4b}
 \begin{aligned}
& \mathrm{Ric}_{\gamma} \ge -(n-1) H^2(1+r^2),\\
& R_\gamma \ge -A, \\
& \vert \tau \vert \ge B >0 \text{ outside a compact set},
& \lambda_1^{-a_n\Delta_\gamma + R_\gamma}(B_0)>0 \\
& \lambda_1^{-a_n\Delta_\gamma +  R_\gamma}(M)<0,
\end{aligned} 
\end{equation}
where $B_0= \{ x \in M \, : \, \tau(x)=0 \}.$
Then, there exists $C(n,M, \gamma, \lambda_{1, \mathrm{conf}})$  such that if
\begin{equation}
\label{HYP5b}
\begin{aligned}
\left\vert R_\gamma \right\vert +& \max \left( \| d\tau \|_{L^{2}(M)}, \| d\tau \|_{L^{p}(M)}\right)  + \max \left( \| \omega_1 \|_{L^{2}(M)},\| \omega_1 \|_{L^{p}(M)}\right) \\
&+ \max\left(\| \omega_2 \|_{L^{2}(M)}, \| \omega_2 \|_{L^{p}(M)} \right) +\vert U \vert+ \epsilon_1 + \epsilon_2 + \epsilon_3  \le C \tau^2,
\end{aligned}
\end{equation}
then  \eqref{systoncomplphiX} admits a $W^{2,p}_{\mathrm{loc}}$ solution.
\end{theorem}
\begin{proof}


This follows from Theorem \ref{050720211748bis} with the existence of barrier functions obtained by Lemma \ref{phi-isasubsolution} and  Theorem \ref{theosetup1}.
\end{proof}

Here we cannot a priori grant that the solution is complete. An interesting program would be to analyse what conditions are required to obtain a lower bound on the associated subsolution, and if the bounded geometry context is enough for the Yamabe type subsolution to be uniformly bounded from below.

Another relevant question lies with the uniqueness of the solutions to the Einstein Constraint Equations. Since our construction scheme relies on a Schauder fixed point theorem, several solutions may coexist on compacts, compounded by the diagonal extraction scheme which may lead to different solutions contained between $\phi_-$ and $\phi_+$ with different extractions. In a fixed model at infinity (for instance AE or AH) it would be interesting to consider whether restricting  the possible range at infinity (by requiring that $\phi_{\pm}$ converge toward the same limit at infinity) would provide further information concerning uniqueness.

\begin{appendices}

\section{Linear analysis of some elliptic boundary value problems}
\label{appendixlinearanalysisellipticboundaryvalueproblem}

Let us consider a compact manifold $M^n$, $n\geq 3$, with smooth boundary $\Sigma$ and a Riemannian metric $g\in W^{2,q}(M)$ with $q>n$, and the operator
\begin{align}\label{MixedBVPRough}
\begin{split}
\mathcal{P}_1:W^{2,q}(M)&\to L^q(M)\times W^{2-\frac{1}{q},q}(\Sigma),\\
u&\mapsto (\Delta_gu,B_0u),
\end{split}
\end{align}
where $B_0:W^{2,q}(M)\to W^{2-\frac{1}{q},q}(\Sigma)$ is the trace map, $u\mapsto u|_{\Sigma}$, which is a bounded map. Then, the following estimate holds:
\begin{proposition}
\label{propsemifradholm}
Let $g$ be a $W^{2,q}$-Riemannian metric on $M$ compact, $q>n$. Then, there is a constant $C=C(M,g,q)>0$ such that the following estimate holds for all $u\in W^{2,q}(M)$:
\begin{align}\label{MaxwellsEstimate.1}
\Vert u\Vert_{W^{2,q}(M)}\leq C\left( \Vert \Delta_gu\Vert_{L^q(M)} + \Vert B_0u\Vert_{W^{2-\frac{1}{q},q}(\Sigma)} + \Vert u\Vert_{L^q(M)} \right).
\end{align}
Moreover, the map $\mathcal{P}_1$ given by (\ref{MixedBVPRough}) is a semi-Fredholm map.
\end{proposition}
\begin{proof}
Under our conditions (\ref{MaxwellsEstimate.1}) directly follows from the same computations as in \cite[Proposition 4]{Maxwell1}. Once this estimate is established, then the semi-Fredholm claim follows directly since it is equivalent to the validity of the estimate (see, for instance, \cite[Chapter 2, Lemma 5.1]{LionsMagenes1}). 
\end{proof}

\begin{cor}\label{RoughLapBelMBVPFred}
Let $g$ be a $W^{2,q}$-Riemannian metric on $M$ compact, $q>n$. Then, the map $\mathcal{P}_1$ given by (\ref{MixedBVPRough}) is a Fredholm map of index zero.
\end{cor}
\begin{proof}
First, notice that if $g\in C^{\infty}(M)$ then $\mathcal{P}_1$ is an isomorphism when acting on $W^{2,2}(M)\to L^2(M)\times H^{\frac{3}{2}}(\Sigma)$ by very standard arguments, such as those of \cite[Theorems 8.9 and Theorem 8.12]{bibellipticpartialdifferentialequations}. Moreover, by \cite[Chapter 5, Theorem 5.4.5]{TriebelBook}, the Fredholm index of $\mathcal{P}_1:W^{2,p}(M)\to L^p(M)\times W^{2-\frac{1}{p},p}(\Sigma)$ is independent of $1<p<\infty$, and thus it must be zero for all such $p$, since we know it is for $p=2$. Finally, in the case where $g\in W^{2,q}(M)$, approximating it by smooth metrics yields a sequence for Fredholm operators of index $0$ converging towards $\mathcal{P}_1$. Since $\mathcal{P}_1$ is semi-Fredholm by Proposition \ref{propsemifradholm} and the index of a semi-Fredholm operator is locally constant by \cite[Chapter XIX, Thm 19.1.5]{Hormander3}, then the index of $\mathcal{P}_1$ is zero as well, and thus $\mathcal{P}_1$ is Fredholm.

\end{proof}


Let us now highlight that $\mathcal{P}_1$ satisfies the following maximum principle:
\begin{proposition}\label{MaxPrinciplesLaplace}
Let $g$ be a $W^{2,q}$-Riemannian metric on $M$ compact, $q>n$, and let $a\in L^q(M)$ be such that $a\geq 0$ \emph{a.e}. If $u\in W^{2,q}(M)$ satisfies $\Delta_gu-au\leq 0$ on $M$ and $B_0u\geq 0$ along $\Sigma$, then $u\geq 0$. Moreover, if $B_0u>0$, then $u>0$.
\end{proposition}
\begin{proof}
Let $\epsilon>0$ and consider $u^{-}(p)\doteq \min\{0,u(p)+\epsilon\}\in W^{1,q}(M)$. Since $B_0u^{-}=0$, then actually $u^{-}\in W^{1,q}_0(M)$ by \cite[Theorem 7.53]{Adams}.\footnote{Notice that $W^{1,q}_0(M)$ denotes the closure of $C^{\infty}_0(\overset{\circ}{M})$ with respect to the $W^{1,p}(M)$-norm, where $\overset{\circ}{M}$ denotes the interior of $M$.} Moreover, since $q>n\geq 3$, then $u^{-}\in W^{1,q}(M)\hookrightarrow W^{1,2}(M)\hookrightarrow L^{q'}(M)$, where the last inclusion holds by the Sobolev inequality, since $q'\leq \frac{2n}{n-2}\Longleftrightarrow q\geq \frac{2n}{n+2}$, while $q>n\geq \frac{2n}{n+2}$ for all $n\geq 3$. Then $u^{-}(\Delta_gu-au)\in L^1(M)$ and
\begin{align}\label{PoissonInjectivity.1}
u^{-}(\Delta_gu-au)\geq 0,
\end{align}
which we intend to integrate over $M$. In order to do this, using $u^{-}\in W_0^{1,2}(M)$ and $u\in W^{2,q}(M)$, consider sequences $\{u^{-}_k\}_{k=1}^{\infty}\subset C^{\infty}_0(\overset{\circ}{M})$ and $\{u_k\}_{k=1}^{\infty}\subset C^{\infty}(M)$ such that
\begin{align*}
u^{-}_k\xrightarrow[k\to\infty]{W^{1,2}(M)} u^{-} \text{ and } u_k\xrightarrow[k\to\infty]{W^{2,q}(M)} u,
\end{align*}
which allows one to compute
\begin{align*}
\int_Mu^{-}\Delta_gu dV_g&=\lim_{k\to\infty}\int_Mu^{-}_k\Delta_gu_k dV_g=-\lim_{k\to\infty}\int_M g(\nabla u^{-}_k,\nabla u_k) dV_g=-\int_M g(\nabla u^{-},\nabla u) dV_g,
\end{align*}
where in the integration by parts we have appealed to the compact support of $u^{-}_k$ in $\overset{\circ}{M}$ to avoid boundary terms, and limits can be justified easily via triangle and Hölder inequalities. Thus, we find:
\begin{align*}
\int_M(u^{-}(\Delta_gu-au))dV_g=-\int_M (g(\nabla u^{-},\nabla u) + a u^{-}u) dV_g.
\end{align*}
Consider the open subset of $M$ with $u^{-}\neq 0$, given by $\Omega^{-}=\{ p\in M\: :\: u^{-}(p)<0 \}$. Then, $u^{-}(p)=u(p)+\epsilon<0$ and $\nabla u^{-}=\nabla u$ on $\Omega^{-}$. Also, $u(p)=u^{-}(p)-\epsilon\leq u^{-}(p)$ on $\Omega^{-}$, and thus $(u^{-}u)(p)\geq (u^{-}(p))^2$ on $\Omega^{-}$. Thus, 
\begin{align*}
0\leq \int_M(u^{-}(\Delta_gu-au))dV_g\leq -\int_M (|\nabla u^{-}|_g^2 + a (u^{-})^2 )  dV_g\leq 0,
\end{align*}
where in the first inequality we have used (\ref{PoissonInjectivity.1}). The above clearly implies that 
\begin{align*}
\int_M\left(|\nabla u^{-}|^2_g+a(u^{-})^2\right)dV_g=0.
\end{align*}
Thus, $\Vert \nabla u^{-}\Vert_{L^2(M)}=0$, which together with $B_0u^{-}=0$ implies $u^{-}\equiv 0$. That is, $u\geq -\epsilon$ for all $\epsilon>0$, which establishes the result. Finally, if $B_0u>0$, the possibility of an interior minimum is not allowed by an application of the Harnack inequality, just as in \cite[Lemma 4]{Maxwell1}.
\end{proof}

Finally, let us show that $\mathcal{P}$ belongs to slightly broader family of operators for which the boundary value problem (\ref{MixedBVPRough}) is uniquely solvable:
\begin{theorem}\label{ShiftedRoughLapBelIso}
Let $g$ be a $W^{2,q}$-Riemannian metric on $M$ compact, $q>n$, and let $a\in L^q(M)$ be such that $a\geq 0$ \emph{a.e}. Then, the operator $\mathcal{P}_{1,a}:W^{2,q}(M)\to L^q(M)\times W^{2-\frac{1}{q},q}(\Sigma)$, given by $\mathcal{P}_{1,a}(u)=(\Delta_gu-au,B_0u)$, is an isomorphism.
\end{theorem}
\begin{proof}
Since $\mathcal{P}_{1,a}$ is a compact perturbation of $\mathcal{P}_1$, then it is still a Fredholm operator of index zero and we need only show injectivity. Thus, consider $u\in W^{2,q}(M)$ such that $\mathcal{P}_{1,a}(u)\equiv 0$ and notice that Proposition \ref{MaxPrinciplesLaplace} applied both to $u$ and $-u$ gives that $u\equiv 0$ and establishes the injectivity.
\end{proof}

\medskip
Let us now consider the case of the conformal Laplacian operator with Dirichlet boundary conditions and analyse its mapping properties as above. That is, still considering a compact manifold $M^n$, $n\geq 3$, with smooth boundary $\Sigma$ and a Riemannian metric $g\in W^{2,q}(M)$ with $q>n$, we now consider the operator
\begin{align}\label{BVPConfLapRough}
\begin{split}
L:W^{2,q}(M)&\to L^q(M)\times W^{2-\frac{1}{q},q}(\Sigma),\\
X&\mapsto (\Delta_{g,\mathrm{conf}}X,B_0X),
\end{split}
\end{align}
Our goal is to understand the mapping properties of (\ref{BVPConfLapRough}), and in particular to establish this is a Fredholm operator of index zero under the above conditions. The strategy to do this follows that of Corollary \ref{RoughLapBelMBVPFred}, and thus one first needs to have these properties established for the case when $g\in C^{\infty}(M)$. In turn, to understand this last case, the first step is to provide these claims in the $L^2$-setting. Thus, whenever $g\in C^{\infty}(M)$, given $1<p<\infty$, it will be convenient to adopt the following notation which highlights the domain of our operator:
\begin{align}\label{BVPConfLapSmooth}
\begin{split}
\mathcal{L}_p:W^{2,p}(M)&\to L^p(M)\times W^{2-\frac{1}{p},p}(\Sigma)\\
X&\mapsto (\Delta_{g,\mathrm{conf}}X,B_0X)
\end{split}
\end{align}
Let us first consider the case $p=2$, where we start by recalling the following standard result:
\begin{theorem}\label{GlobalRegularity}
Let $(M,g)$ be a compact smooth Riemannian manifold with smooth boundary $\Sigma$. If $X\in W^{2,2}(TM)$, then the following implication holds:
\begin{align}\label{GlobalRegularity.1}
\text{ if } k\in\mathbb{N} \text{ and } \mathcal{L}_2X\in W^{k,2}(M;TM)\times W^{k+\frac{3}{2},2}(\Sigma;TM)\Longrightarrow X\in W^{k+2,2}(M;TM),
\end{align}
and there is a constant $C>0$ (depending on $k$) such that
\begin{align}\label{GlobalRegularity.2}
\Vert X\Vert_{W^{k+2,2}(M)}\leq C\left( \Vert \Delta_{g,\mathrm{conf}}X\Vert_{W^{k,2}(M)} + \Vert B_0u\Vert_{W^{k+\frac{3}{2},2}(\Sigma)} + \Vert u\Vert_{L^2(M)} \right).
\end{align}
Moreover, the map $\mathcal{L}_2:W^{2,2}(M)\to L^2(M)\times W^{\frac{3}{2},2}(\Sigma)$ is Fredholm.
\end{theorem}
\begin{proof}
Since the boundary value problem associated with $L$ is regular elliptic (Dirichlet conditions always are) \eqref{GlobalRegularity.1}, \eqref{GlobalRegularity.2} directly follow from, for example, \cite[Chapter 5, Proposition 11.2]{TaylorBook1}. The Fredholm claim follows from \cite[Chapter 5, Proposition 11.16]{TaylorBook1}.
\end{proof}

From the above theorem, for $g\in C^{\infty}(M)$ we have
\begin{align}
\mathrm{Im}(\mathcal{L}_2)=\mathrm{Ker}^{\perp}(\mathcal{L}_2^{*}).
\end{align}
where $\mathcal{L}_2^{*}:L^2(M;TM)\times W^{-\frac{3}{2},2}(\Sigma;TM)\to \left(W^{2,2}(M;TM)\right)'$ is defined by duality as
\begin{align}
[\mathcal{L}_2^{*}(Y,Z)](X)\doteq \langle Y,\Delta_{g, \mathrm{conf}}X\rangle_{L^2(M,dV_g)} + \langle Z,B_0X\rangle_{W^{-\frac{3}{2},2}(\Sigma)\times W^{\frac{3}{2},2}(\Sigma)},
\end{align}
for all $(Y,Z)\in L^2(M;TM)\times W^{-\frac{3}{2},2}(\Sigma;TM)$  and  $X\in W^{2,2}(M;TM)$. A key tool to understand $\mathrm{Ker}(\mathcal{L}_2^{*})$ is the analysis of the formal adjoint problem to (\ref{BVPConfLapSmooth}), which we introduce in terms of the associated Green's identity for $\mathcal{L}_2$, given by
\begin{align}\label{GreensID}
\int_{M}\left( \langle\Delta_{g,\mathrm{conf}}X,Y\rangle_g - \langle\Delta_{g,\mathrm{conf}}Y,X\rangle_g \right)dV_g=\int_{\Sigma}\left( \langle\pounds_{g,\mathrm{conf}}X(\nu,\cdot),Y\rangle_g  -\langle\pounds_{g,\mathrm{conf}}Y(\nu,\cdot),X\rangle_g \right)d\omega_g,
\end{align}
for all $X,Y\in \Gamma(TM)$. 
\begin{definition}
The formal adjoint problem to (\ref{BVPConfLapSmooth}) when $p=2$ is defined to be the boundary value problem defined by the operator $\mathcal{L}_2:C^{\infty}(M)\to C^{\infty}(M)\times C^{\infty}(\Sigma)$. We moreover define the spaces:
\begin{align}\label{Nspace}
N\doteq \{ Y\in C^{\infty}(M)\: :\: \Delta_{g,\mathrm{conf}}Y=0 \text{ and } B_0Y=0\}=\mathrm{Ker}(\mathcal{L}_2).
\end{align}
and
\begin{align}
\mathcal{H}\doteq \{ Z\in W^{2,2}(M;TM)\: :\: \langle Z,Y\rangle_{L^2(M,dV_g)}=0 \: \forall\: Y\in N \}.
\end{align}
\end{definition}
The above definition highlights that, due to (\ref{GreensID}), the boundary value problem associated with $\mathcal{L}_2$ is formally self-adjoint. 

\begin{proposition}\label{RmkGiqcuad}
The space $N$ defined by (\ref{Nspace}) is trivial. That is, $N=\{0\}$.
\end{proposition}
\begin{proof}
This result follows from the proof of \cite[Lemma 6.7]{GicquaudHyperbolic}, which we briefly explain below.

The first step is to notice that if $Y\in N$, then integration by parts of $\langle Y, \Delta_{g,\mathrm{conf}}Y\rangle_g=0$ shows that $\pounds_{g,conf}Y=0$ on $M$ and thus $N$ consists of (smooth) CKFs which vanish along $\Sigma$. The second step is to notice these conditions imply that $\nabla Y|_{\Sigma}=0$ for any $Y\in N$. This follows by considering an arbitrary point $p\in \Sigma$ and a coordinate system $\{U,(y^i,y^n)\}_{i=1}^{n-1}$ around $p$, adapted to the boundary so that $U\cap\Sigma$ agrees with $y^n=0$, and such that $g(\partial_{y^{\alpha}},\partial_{y^{\beta}})|_{p}=\delta_{\alpha\beta}$, $\alpha,\beta=1,\cdots,n$. Since $Y|_{\Sigma}=0$, it directly follows that $\nabla_{\partial_{y^i}}Y|_{p}=0$ for any $i=1,\cdots,n-1$, while in the case of the $y^n$-direction, one can follow  \cite[Lemma 6.7]{GicquaudHyperbolic} to show that $\pounds_{g,conf}Y|_{p}=0$ implies $\nabla_{\partial_{y^n}}Y|_{p}=0$ as well. All this then implies that both $Y,\nabla Y|_{\Sigma}=0$ for any $Y\in N$. One may then construct a collar neighbourhood $\mathcal{C}$ of $M$ extending $M$. Denote the extended (open) manifold by $\tilde{M}$ and equip it with a smooth Riemannian metric $\tilde{g}$ such that $\tilde{g}|_{M}=g$. Given $Y\in N$, extend it to $\tilde{M}$ by zero to produce a vector field $\tilde{Y}\in C^{1}(\tilde{M})$, which by construction satisfies $\pounds_{\tilde{g},conf}\tilde{Y}=0$ and is thus a weak solution to $\Delta_{\tilde{g},\mathrm{conf}}\tilde{Y}=0$, granting by interior regularity that $\tilde{Y}\in C^{\infty}(\tilde{M})$. Since $\tilde{Y}|_{\mathcal{C}}\equiv 0$, one concludes that $\tilde{Y}\equiv 0$ appealing to the proof of \cite[Theorem 4]{Maxwell1}, where a unique continuation result for $W^{2,q}$-CKFs of $W^{2,q}$-metric is established for $q>n$.
\end{proof}

In this context, the following holds:
\begin{proposition}\label{PropLionsExistence}
Let $(M,g)$ be a compact smooth Riemannian manifold with smooth boundary $\Sigma$. If $f\in C^{\infty}(M)$, then there exists a solution $X\in C^{\infty}(M)$ to 
\begin{align}\label{BVProblem.1}
\begin{split}
\Delta_{g,\mathrm{conf}}X&=f \text{ in } M,\\
B_0X&=0 \text{ along } \Sigma,
\end{split}
\end{align}
for all $f\in C^{\infty}(M)$.
\end{proposition}
\begin{proof}
By an adaptation \emph{mutatis mutandis} of \cite[Proposition 5.2]{LionsMagenes1} from the case of scalar valued functions to our setting, one has that there exists a solution $X\in C^{\infty}(M)$ to (\ref{BVProblem.1}) iff $f\in \mathcal{H}$. But Proposition \ref{RmkGiqcuad} grants that $N=\{0\}$ and thus $\mathcal{H}=W^{2,2}(M)$, and thus $f\in \mathcal{H}$ $\forall f\in C^{\infty}(M)$.
\end{proof}

From this, we can establish the following result:
\begin{proposition}\label{PropTrueAdjKer}
Let $(M,g)$ be a compact smooth Riemannian manifold with smooth boundary $\Sigma$ and consider the operator $\mathcal{L}_2$. Then, $\mathrm{Ker}(\mathcal{L}_2^{*})=\{(Y,\pounds_{g,\mathrm{conf}}Y(\nu,\cdot))\in L^2(M)\times W^{-\frac{3}{2},2}(\Sigma)\: :\: Y\in N\}=\{0\}$.
\end{proposition}
\begin{proof}
If $(v,\psi)\in \mathrm{Ker}(\mathcal{L}_2^{*})$, 
by definition one must have
\begin{align}\label{TrueAdjKer.1}
\langle v,\Delta_{g,\mathrm{conf}}u\rangle_{L^2(M,dV_g)}+\langle \psi,B_0u\rangle_{W^{-\frac{3}{2},2}(\Sigma)\times W^{\frac{3}{2},2}(\Sigma)}=0 \:\: \forall\: u\in W^{2,2}(M),
\end{align}
which in particular implies that
\begin{align}\label{TrueAdjKer.2}
\langle v,\Delta_{g,\mathrm{conf}}u\rangle_{L^2(M,dV_g)}=0 \:\: \forall\: u\in W^{2,2}_B(M),
\end{align}
where $W^{2,2}_B(M)\doteq \{u\in W^{2,2}(M) \: :\: B_0u=0\}$. 

Adaptating  \emph{mutatis mutandis} the proof  \cite[Proposition 5.4]{LionsMagenes1}, one can first show that $v$ is smooth and $(v, \psi=\pounds_{g, \mathrm{conf}}v(\nu,.))$, and then that $v\in N$. For reference, one may notice that \eqref{TrueAdjKer.1} and \eqref{TrueAdjKer.2} correspond respectively  to \cite[(5.51)]{LionsMagenes1} and  \cite[(5.52)]{LionsMagenes1}. All this shows that $\mathrm{Ker}(\mathcal{L}_2^{*})=\{(Y,\pounds_{g,\mathrm{conf}}Y(\nu,\cdot))\in L^2(M)\times W^{-\frac{3}{2},2}(\Sigma)\: :\: Y\in N\}$, which combined with Proposition \ref{RmkGiqcuad} finishes the proof.
\end{proof}

The above proposition directly implies:
\begin{cor}
Let $(M,g)$ be a compact smooth Riemannian manifold with smooth boundary $\Sigma$ and consider the operator $\mathcal{L}_2$. Then, $\mathcal{L}_2$ has Fredholm index equal to zero, and is actually an isomorphism
\end{cor}
\begin{proof}
Since $\mathcal{L}_2$ is Fredholm by Theorem \ref{GlobalRegularity}, we know that $\left(\mathrm{Coker}(\mathcal{L}_2)\right)'\cong \mathrm{Ker}(\mathcal{L}_2^{*})$ and $\mathrm{Index}(\mathcal{L}_2)=\mathrm{dim}(\mathrm{Ker}(\mathcal{L}_2)) - \mathrm{dim}(\mathrm{Ker}(\mathcal{L}_2^{*}))=0$, where the last equality follows from Propositions \ref{RmkGiqcuad} and \ref{PropTrueAdjKer}. Moreover, $\mathrm{Im}(\mathcal{L}_2)=\mathrm{Ker}^{\perp}(\mathcal{L}_2^{*})$, which shows that $\mathcal{L}_2$ is surjective since $\mathrm{Ker}(\mathcal{L}_2^{*})=\{0\}$. Since $\mathcal{L}_2$ is also injective by Proposition \ref{RmkGiqcuad}, this establishes the isomorphism claim.
\end{proof}

Having the above results for $\mathcal{L}_2$, let us now analyse the case of $\mathcal{L}_p$, still for $g\in C^{\infty}(M)$. 

\begin{proposition}\label{ConfLapFredholmSmoothCase}
Let $g$ be a smooth Riemannian metric on $M$ compact. Then, the map $\mathcal{L}_p$ is a Fredholm map for any $1<p<\infty$ and its index is independent of $p$ given by $\mathrm{Index}(\mathcal{L}_p)=\mathrm{Index}(\mathcal{L}_2)=0$.
 In particular, $\mathcal{L}_p$ is an isomorphism for any such $1<p<\infty$.
\end{proposition}
\begin{proof}
First of all, in the case that $g\in C^{\infty}(M)$, the elliptic estimates (\ref{BVPConfLapRough}) hold for $\mathcal{L}_p$ in their $L^p$-version for all $1<p<\infty$, and thus $\mathcal{L}_p$ is semi-Fredholm.\footnote{This follows by the classical freezing of coefficients technique since the associated estimates hold for the frozen coefficient operators, as outlined for instance in \cite[Proposition 4]{Maxwell1}.} Also, appealing to a classical difference quotients method one obtains an elliptic regularity statement showing that $\mathrm{Ker}(\mathcal{L}_p)=\mathrm{Ker}(\mathcal{L}_2)$ is independent of $p$, and is thus always trivial by Proposition \ref{RmkGiqcuad}. Thus, being $\mathcal{L}_p:W^{2,p}(M)\to L^p(M)\times W^{2-\frac{1}{p},p}(\Sigma)$ Fredholm and injective, there is an estimate of the form\footnote{See, for instance, the proof of the estimate \cite[(19.1.2)]{Hormander3}.}
\begin{align}\label{InjectivityEstimate}
\Vert u\Vert_{W^{2,p}(M)}\leq C\left(\Vert \Delta_{g,conf}u\Vert_{L^p(M)} + \Vert u\Vert_{W^{2-\frac{1}{p},p}(\Sigma)}\right), \:\: \forall\: u\in W^{2,p}(M),
\end{align}
for a fixed given constant $C>0$. Then, given $f\in L^p(M)$ and a sequence $\{f_j\}_{j=1}^{\infty}\subset C^{\infty}(M)$ such that $f_j\xrightarrow[]{L^p(M)}f$, by Proposition \ref{PropLionsExistence} we know there exists $\{u_j\}_{j=1}^{\infty}\subset C^{\infty}(M)$ such that $\mathcal{L}_pu_j=(f_j,0)$. Then, by (\ref{InjectivityEstimate}), we have:
\begin{align*}
\Vert u_j-u_l\Vert_{W^{2,p}(M)}\leq C \Vert f_j-f_l\Vert_{L^p(M)},
\end{align*} 
and therefore we see that $\{u_j\}_{j=1}^{\infty}$ is convergent in $W^{2,p}(M)$. Denoting the limit by $u\in W^{2,p}(M)$, we find $\mathcal{L}_pu=(f,0)$. 

Finally, if $(f,\varphi)\in L^p(M)\times W^{2-\frac{1}{p},p}(\Sigma)$, consider the problem $\mathcal{L}_p\tilde{u}=(f-\mathcal{L}_p\mathcal{E}\varphi,0)\in L^p(M)\times W^{2-\frac{1}{p},p}(\Sigma)$, where $\mathcal{E}:W^{2-\frac{1}{p},p}(\Sigma)\to W^{2,p}(M)$ denotes a continuous extension operator. By the previous analysis we know that there exists a unique solution $\tilde{u}\in W^{2,p}(M)$ to this problem, and then $u\doteq \tilde{u}+\mathcal{E}\varphi\in W^{2,p}(M)$ solves $\mathcal{L}_pu=(f,\varphi)\in L^p(M)\times W^{2-\frac{1}{p},p}(\Sigma)$, since by construction $B_0\tilde{u}=0$ and $B_0\mathcal{E}\varphi=\varphi$. This shows that $\mathcal{L}_p:W^{2,p}(M)\to L^p(M)\times W^{2-\frac{1}{p},p}(\Sigma)$ is indeed surjective and thus an isomorphism.

\end{proof}

Having now dealt with the case of smooth metrics, let us now treat the following case for lower regularity coefficients:
\begin{proposition}\label{CKLFredholmProp}
Let $g$ be a $W^{2,q}$-Riemannian metric on $M^n$ compact, $q>n$. Then, there is a constant $C=C(M,g,q)>0$ such that the following estimate holds for all $u\in W^{2,q}(M)$:
\begin{align}\label{MaxwellsEstimate.2}
\Vert X\Vert_{W^{2,q}(M)}\leq C\left( \Vert \Delta_{g,conf}X\Vert_{L^q(M)} + \Vert B_0X\Vert_{W^{2-\frac{1}{q},q}(\Sigma)} + \Vert X\Vert_{L^q(M)} \right).
\end{align}
Moreover, the map $L$ given by (\ref{BVPConfLapRough}) is a Fredholm map of index zero.
\end{proposition}
\begin{proof}
Under our hypotheses, (\ref{MaxwellsEstimate.2}) directly follows from the same computations as in \cite[Proposition 4]{Maxwell1}, and then $L$ is semi-Fredholm, for instance, due to \cite[Chapter 2, Lemma 5.1]{LionsMagenes1}. Finally, to obtain the Fredholm and index claims we follow the same strategy as in Corollary \ref{RoughLapBelMBVPFred} approximating $g$ in $W^{2,q}(M)$ by smooth metrics so that $L$ is approximated in the operator norm by operators $\{\mathcal{L}^{(j)}_p\}_{j=1}^{\infty}$ to which Proposition \ref{ConfLapFredholmSmoothCase} applies. Then, just as in Corollary  \ref{RoughLapBelMBVPFred}, the stability properties of Fredholm operators give both claims.
\end{proof}

Finally, we can prove the following isomorphism result for $\Delta_{g,\mathrm{conf}}$:
\begin{theorem}\label{CKLIsoThm}
Let $g$ be a $W^{2,q}$-Riemannian metric on $M^n$, a compact manifold with smooth boundary $\Sigma$, $q>n$. Then, $L:W^{2,q}(M)\to L^q(M)\times W^{2-\frac{2}{q},q}(\Sigma)$ is an isomorphism.
\end{theorem}
\begin{proof}
Since by Proposition \ref{CKLFredholmProp} $L$ is Fredholm of index zero, then if it is injective it must be surjective as well. But following the same ideas as in the proof of Proposition \ref{RmkGiqcuad}, if $X\in \mathrm{Ker}(L)$, integrating by parts the equation $\langle X,\Delta_{g,\mathrm{conf}}X\rangle_g=0$ and using that $B_0X=0$ grants that $X$ is a $W^{2,q}(M)$-CKF of $g$. But then the same arguments as in proof of Proposition \ref{RmkGiqcuad} apply to show that $X\equiv 0$, and hence $L$ is injective, establishing the proof.


\end{proof}

\section{Manifolds of bounded geometry}\label{BoundedGeometry}

In order to apply elliptic estimates to complete manifolds, we will need to work in the bounded geometry context.  To suit these needs, we here recall notions of bounded geometry taken from Shubin's \cite{Shubin} (see also \cite[Exposé V]{Shubin2}).

\begin{definition}
We will say $(M, \gamma)$ is a (smooth) manifold of bounded geometry if:
\begin{enumerate}
\item $r_{inj} >0,$
\item $\left\vert \nabla^k \mathrm{Riem}_\gamma \right\vert_\gamma \le C_k $ for all $k \in \mathbb{N} $,
\end{enumerate}
where $r_{inj}$ stands for the injectivity radius of $(M, \gamma)$ and $C_k$ for constants depending on $k$.
\end{definition}

Given a point $x\in M$ and $r\in (0,r_{inj})$ we have normal coordinate systems given by $\exp_x:B_r(x)\to U_{x,r}\subset \mathbb{R}^n$. Such coordinate systems will be called \textbf{canonical}. The second condition above guarantees that the transition matrices (together with their derivatives up to any order) between such coordinate systems are bounded. 

Let $E\to M$ be a vector bundle over $M$. We say that \textbf{$E$ is a bundle of bounded geometry} if trivializations of $E$ over canonical coordinate systems $U,U'$, with $U\cap U'\neq \emptyset$, give rise to transition functions $g_{UU'}$ such that $\partial^{\alpha}_yg_{UU'}(y)$ are bounded by constants $C_{\alpha}$ which do not depend on the pair $U,U'$. Tensor bundles over manifolds of bounded geometry are bundles of bounded geometry.

In the above context we can find countable coverings with finite multiplicity following  from \cite[Lemma 1.2]{Shubin}.

\begin{lem}
\label{coveringwithboundedmultiplicity}
Let $(M,\gamma)$ be a manifold of bounded geometry.
Then there exists $\epsilon_0>0$ with $\epsilon_0 < \frac{1}{3} r_{\mathrm{inj}}$ such that if $\epsilon \in (0, \epsilon_0)$ then there exists a countable covering of $M$ by balls of radius $\epsilon$ $X = \bigcup B(x_i, \epsilon)$ such that the covering of $M$ by the balls $B(x_i, 2\epsilon)$ has a finite multiplicity.
\end{lem}
On these kinds of coverings one can then define bounded partitions of unity \cite[Lemma 1.3]{Shubin}.
\begin{lem}\label{Partitionofunity}
Let $(M,\gamma)$ be a manifold of bounded geometry and fix $\epsilon<\frac{r}{2}$ with $r\in (0,r_{inj})$. Then, there exists a sequence of points $\{x_i\}_{i=1}^{\infty}\subset M$ such that $M=\cup_{i}B_{\epsilon}(x_i)$ and a partition of unity $1=\sum_{i=1}^{\infty}\varphi_i$ on $M$ such that:\\
\noindent 1) $\varphi_i\geq 0$, $\varphi_i\in C^{\infty}_0(M)$ with $\mathrm{supp}(\varphi_i)\subset B_{2\epsilon}(x_i)$;\\
\noindent 2) $\vert \partial^{\alpha}_y\varphi_i(y) \vert \leq C_{\alpha}$,
for every multiindex $\alpha$ in canonical coordinates uniformly with respect to $i$ (\textit{i.e} with the constants $C_{\alpha}$ which does not depend on $i$).
\end{lem}

Using such a partition of unity, we define the Sobolev spaces $W^{k,p}$, with $k\in \mathbb{N}_0$ and $1\leq p<\infty$ as the closure of $C^{\infty}_0$ with respect to the norm
\begin{align}
\label{Wspdef}
\Vert u\Vert^p_{W^{k,p}(M)}\doteq \sum_{i=1}^{\infty}\Vert \varphi_iu\Vert^p_{W^{k,p}(B_{2\epsilon}(x_i))},
\end{align}
where the spaces $W^{k,2}(M)$ have a Hilbert structure. 
It is interesting to consider whether this definition depends on the covering and the subordinate partition of unity. We refer the interested reader to \cite[Definition 5.10, Theorem 5.11]{https://doi.org/10.1002/mana.201300007} for a discussion on when coverings and partitions are compatible. In particular, geodesic refinements are compatible.

In addition, in the case $k=0$ one can check that the norm defined above is an equivalent norm to the one introduced by considering the space $L^p(M,dV_g)$ as induced by the volume measure related to $(M,g)$.
 We will use both characterisations at different moments, appealing to the one most useful for our pourposes. Also, we notice that the usual Sobolev embeddings and multiplication properties hold in this context (see  \cite[Corollary 3.19]{MR1481970}, or \cite[last paragraph p. 68]{Shubin}). Due to the abundant use of this fact throughout this paper, we will single them out in the following proposition.
\begin{proposition}
\label{SobEstBoundGeom}
Let $(M, \gamma)$ be a complete manifold of bounded geometry. If $1<p<\infty$ and $k>0$ an integer, then:
$$\begin{aligned}
&W^{k,p}(M) \hookrightarrow L^{\frac{np}{n-kp}}(M) \text{ if } k < \frac{n}{p} \\
&W^{k,p}(M) \hookrightarrow C^{k-\left\lfloor \frac{n}{p}\right\rfloor-1, \alpha}(M) \text{ if } k> \frac{n}{p}.
\end{aligned}$$
\end{proposition}
One could write a short proof of the previous proposition appealing to (\ref{Wspdef}), the corresponding local embeddings, the finite multiplicity of our special covering provided by Lemma \ref{Partitionofunity} and embeddings for summable sequence spaces $\ell^p$. For instance, regarding the first of the above inequalities, the classical embeddings apply on each $B_{2\epsilon}(x_i)$, with a uniform constant in $i$.  Equality \eqref{Wspdef} combined with an embedding $\ell^p \subset \ell^{\frac{np}{n-kp}}$ for the sequence $\Vert \varphi_iu\Vert_{W^{k,p}(B_{2\epsilon}(x_i))}$ allows one to recover the desired inequality. Such a sketch highlights the pivotal role played by the non trivial partition offered by Lemma \ref{Partitionofunity}.

Let us now consider $E,F\to M$ be two vector bundles of bounded geometry; $A:C^{\infty}(M;E)\to C^{\infty}(M;F)$ a second order differential operator with smooth coefficients. We will call it \textbf{$C^{\infty}$-bounded} if in any canonical coordinate system $A$ is written  in the form
\begin{align}\label{ellipticops.1}
A=\sum_{\vert \alpha \vert \leq 2} a_{\alpha}(y)\partial^{\alpha}_y,
\end{align}
with $ a_{\alpha}\in \mathrm{Hom}(E,F)$ satisfying uniform estimates of the form $\vert \partial^{\beta}_{y}a_{\alpha}(y)\vert \leq C_{\beta}$ for any multiindex $\beta$ and where $C_{\beta}$ does not depend on the canonical coordinate system. 


We use in  this paper  a regularity result for elliptic operators, extracted again from \cite[Lemma 1.4]{Shubin}. 
\begin{lem}\label{regularity.1}
Let $A$ be a $C^{\infty}$-bounded linear uniformly elliptic operator of order $2$ acting between two tensor bundles over $M$. For any $1<p<\infty$, if $u\in W^{2,p}_{\mathrm{loc}}(M;E)\cap L^p (M;E)$ and $A u\in L^p(M;E)$, then $u\in W^{2,p}(M;E)$ and there exists a constant $C>0$ such that
\begin{align}
\Vert u\Vert_{W^{2,p}(M)}\leq C\left(\Vert Au\Vert_{L^{p}(M)} + \Vert u\Vert_{L^p(M)}\right).
\end{align}
\end{lem}

\begin{proof}
This is the application of Lemma 1.4 of \cite{Shubin} for $s=2$, $m=2$ and $t=0$.

\end{proof}


Two operators are of particular interest in this work: $\Delta_\gamma$ and $\Delta_{\gamma, \mathrm{conf}}$. We will spell out the regularity results for these two operators for convenience.
\begin{lem}\label{regularitydelta}
Let $(M, \gamma)$ be a smooth manifold of bounded geometry and let $A$ stand for either $\Delta_{\gamma}$ or $\Delta_{\gamma,\mathrm{conf}}$. For any $1<p<\infty$, if $u\in W^{2,p}_{\mathrm{loc}}(M)\cap L^p (M)$ and $A u\in L^p(M;E)$, then $u\in W^{2,p}(M)$ and there exists a constant $C>0$ such that
\begin{align}
\Vert u\Vert_{W^{2,p}(M)}\leq C\left(\Vert Au\Vert_{L^{p}(M)} + \Vert u\Vert_{L^p(M)}\right).
\end{align}
\end{lem}
\begin{proof}
Bounded geometry hypotheses ensure that both $\Delta_\gamma$ and $\Delta_{\gamma, \mathrm{conf}}$ are bounded, while direct computations of their principal symbols ensure they are uniformly elliptic.
\end{proof}

\end{appendices}


\bigskip
{\bf Acknowledgments: } The authors would like to thank the CAPES-COFECUB, CAPES-PNPD and FUNCAP for their financial support as well as Bruno Premoselli and Stefano Pigola for their comments on a preliminary version, as well as the anonymous reviewer who highlighted   \cite[Lemma 6.7]{GicquaudHyperbolic} and directed us toward the use of Schauder's fixed point theorem in Sections 2 and 3. The authors would also thank the Alexander von Humboldt Foundation and the Potsdam University for  financial support. 

\bigskip
{\bf Data Availability Statement: } Data sharing not applicable to this article as no datasets were generated or analysed during the current study.

\bigskip
{\bf Conflict of Interest Statement: } On behalf of all authors, the corresponding author states that there is no conflict of interest.

\bibliographystyle{alpha}
\bibliography{bibliography}

@article{DiltsIsenberg16,
  title = {{Existence and blowup results for asymptotically Euclidean initial data sets generated by the conformal method}},
  author = {Dilts, J. and Isenberg, J.},
  journal = {Phys. Rev. D},
  volume = {94},
  issue = {10},
  pages = {104046},
  numpages = {14},
  year = {2016},
  month = {Nov},
  publisher = {American Physical Society},
  doi = {10.1103/PhysRevD.94.104046},
  url = {https://link.aps.org/doi/10.1103/PhysRevD.94.104046}
}

@article{GicquaudHyperbolic,
title = {{De l'équation de prescription de courbure scalaire aux équations de contrainte en relativité générale sur une variété asymptotiquement hyperbolique}},
journal = {Journal de Mathématiques Pures et Appliquées},
volume = {94},
number = {2},
pages = {200-227},
year = {2010},
issn = {0021-7824},
doi = {https://doi.org/10.1016/j.matpur.2010.03.011},
url = {https://www.sciencedirect.com/science/article/pii/S0021782410000449},
author = {Gicquaud, R.}
}

@book{Adams,
 author = {Adams, R.},
 title = {{S}obolev {S}paces},
 fseries = {{P}ure and {A}pplied {M}athematics},
 series = {},
 issn = {},
 volume = {},
 isbn = {},
 year = {1975},
 publisher = {{A}cademic {P}ress, {I}nc., {C}ambridge},
 language = {},
 zbMATH = {},
 Zbl = {}
}

@book{TaylorBook1,
 author = {Taylor, M. E.},
 title = {{Partial Differential Equations. 1: {Basic} Theory}},
 fseries = {Applied Mathematical Sciences},
 series = {Appl. Math. Sci.},
 issn = {},
 volume = {115},
 isbn = {},
 year = {1996},
 publisher = {Berlin: Springer-Verlag},
 language = {},
 zbMATH = {884165},
 Zbl = {0869.35002}
}

@book{TriebelBook,
 author = {Triebel, H.},
 title = {{Interpolation Theory, Function Spaces, Differential Operators.}},
 edition = {2nd rev. a. enl. ed.},
 isbn = {3-335-00420-5},
 year = {1995},
 publisher = {Leipzig: Barth},
 language = {English},
 zbMATH = {762608},
 Zbl = {0830.46028}
}

@book{LionsMagenes1,
 author = {Lions, J. L. and Magenes, E.},
 title = {{Non-homogeneous boundary value problems and applications. {Vol}. {I}. {Translated} from the {French} by {P}. {Kenneth}}},
 fseries = {Grundlehren der Mathematischen Wissenschaften},
 series = {Grundlehren Math. Wiss.},
 issn = {0072-7830},
 volume = {181},
 year = {1972},
 publisher = {Springer, Cham},
 language = {English},
 zbMATH = {3353865},
 Zbl = {0223.35039}
}

@book{BrezisBook,
 author = {Brezis, H.},
 title = {Functional {A}nalysis, {Sobolev} {S}paces and {P}artial {D}ifferential {E}quations},
 fseries = {Universitext},
 series = {Universitext},
 issn = {},
 isbn = {},
 year = {2011},
 publisher = {New York, NY: Springer},
 language = {},
 zbMATH = {5633610},
 Zbl = {1220.46002}
}

@book{TaylorBook3,
 author = {Taylor, M. E.},
 title = {{Partial Differential Equations. 3: Nonlinear Equations}},
 fseries = {Applied Mathematical Sciences},
 series = {Appl. Math. Sci.},
 issn = {},
 isbn = {},
 year = {2023},
 publisher = {Berlin: Springer-Verlag},
 language = {},
}

@BOOK{Hormander3,
    Author = {L. H\"ormander},
 Title = {The analysis of linear partial differential operators. III: Pseudo-differential Operators},
 FJournal = {Classics in Mathematics},
 Journal = {Class. Math.},
 ISSN = {1431-0821},
 ISBN = {978-3-540-49938-1},
 Pages = {viii + 525},
 Year = {2007},
 Publisher = {Berlin: Springer},
 Language = {English},
}

@ARTICLE{York1,
    author = {York Jr., J.},
    title = {{Conformally invariant orthogonal decomposition of symmetric tensors on Riemannian manifolds and the initial‐value problem of general relativity}},
    journal = {Journal of Mathematical Physics},
    volume = {14},
    number = {4},
    pages = {456-464},
    year = {1973},
    doi = {10.1063/1.1666338},
    url = {https://doi.org/10.1063/1.1666338 },
    mrnumber = {},
    zbl = {0259.53014},
}

@ARTICLE{GunPig,
    author = {Güneysu, B. and Pigola, S.},
    title = {${L}^p$-interpolation inequalities and global {S}obolev regularity results},
    journal = { Annali di Matematica},
    volume = {198},
    pages = {83–96},
    year = {2019},
    doi = {10.1007/s10231-018-0763-7},
    url = {https://doi.org/10.1007/s10231-018-0763-7 },
    mrnumber = {},
}

@book {MR1481970,
    AUTHOR = {Hebey, E.},
     TITLE = {Sobolev spaces on {R}iemannian manifolds},
    SERIES = {Lecture Notes in Mathematics},
    VOLUME = {1635},
 PUBLISHER = {Springer-Verlag, Berlin},
      YEAR = {1996},
     PAGES = {x+116},
      ISBN = {3-540-61722-1},
   MRCLASS = {46E35 (58D15)},
  MRNUMBER = {1481970},
MRREVIEWER = {H. Triebel},
       DOI = {10.1007/BFb0092907},
       URL = {https://doi.org/10.1007/BFb0092907},
}

@ARTICLE{Gun,
    author = {Güneysu, B.},
    title = {{Sequences of {L}aplacian Cut-Off Functions}},
    journal = { J Geom Anal},
    volume = {26},
    pages = {171–184},
    year = {2016},
    doi = {10.1007/s12220-014-9543-9},
    url = {https://doi.org/10.1007/s12220-014-9543-9 },
    mrnumber = {},
}

@article{SHUBIN200192,
title = {{Essential Self-Adjointness for Semi-bounded Magnetic {S}chrödinger Operators on Non-compact Manifolds}},
journal = {Journal of Functional Analysis},
volume = {186},
number = {1},
pages = {92-116},
year = {2001},
issn = {0022-1236},
doi = {https://doi.org/10.1006/jfan.2001.3778},
url = {https://www.sciencedirect.com/science/article/pii/S0022123601937784},
author = { Shubin, M.}
}

@ARTICLE{York2,
    title = {{Initial - value problem of general relativity. I. General formulation and physical interpretation}},
    author = {O Murchadha, N. and York Jr., J.},
    journal = {Phys. Rev. D},
    volume = {10},
    issue = {2},
    pages = {428--436},
    numpages = {0},
    year = {1974},
    month = {7},
    publisher = {American Physical Society},
    doi = {10.1103/PhysRevD.10.428},
    url = {https://link.aps.org/doi/10.1103/PhysRevD.10.428},
    mrnumber = {0406318},
    zbl = {},
}

@ARTICLE{CB-ConformalMethod.1,
    author = {Choquet Fourès-Bruhat, Y. },
    title = {{Sur le problème des conditions initiales}},
    volume = {245},
    journal = {C. R. Acad. Sci},
    number = {},
    publisher = {},
    pages = {},
    year = {1957},
    doi = {},
    url = {},
    mrnumber = {0090448},
    zbl = {},
}

@ARTICLE{CB-WellPosedness1,
    author = {Fourès-Bruhat, Y.},
    title = {Théorème d'existence pour certains systèmes d'équations aux dérivées partielles non linéaires},
    volume = {88},
    journal = {Acta Mathematica},
    number = {none},
    publisher = {Institut Mittag-Leffler},
    pages = {141--225},
    year = {1952},
    doi = {10.1007/BF02392131},
    url = {https://doi.org/10.1007/BF02392131},
    mrnumber = {},
    zbl = {},
}

@incollection{CB-WellPosedness2,
    author = {Choquet-Bruhat, Y.},
    title = {The {C}auchy Problem},
    booktitle = {Gravitation: {A}n {I}ntroduction to {C}urrent {R}esearch},
    editor = {L. Witten},
    address = {New York},
    Publisher = {J. Wiley},
    pages = {},
    year = {1962},
    mrnumber = {0143626},
    zbl = {0658.53078},
}

@article{Gicquaud_2025,
   title={{What Uniqueness for the Holst-Nagy-Tsogtgerel–Maxwell Solutions to the Einstein Conformal Constraint Equations?}},
   volume={68},
   ISSN={1572-9060},
   url={http://dx.doi.org/10.1007/s10455-025-10000-9},
   DOI={10.1007/s10455-025-10000-9},
   number={1},
   journal={Annals of Global Analysis and Geometry},
   publisher={Springer Science and Business Media LLC},
   author={Gicquaud, R.},
   year={2025},
   month=jun }

@article{Gicquaud_2022,
   title={{Existence of solutions to the Lichnerowicz equation: A new proof}},
   volume={63},
   ISSN={1089-7658},
   url={http://dx.doi.org/10.1063/5.0050920},
   DOI={10.1063/5.0050920},
   number={2},
   journal={Journal of Mathematical Physics},
   publisher={AIP Publishing},
   author={Gicquaud, R.},
   year={2022},
   month=feb }

@article{gicquaud2019prescribednonpositivescalar,
      title={{Prescribed non positive scalar curvature on asymptotically hyperbolic manifolds with application to the Lichnerowicz equation}}, 
      author={Gicquaud, R.},
	doi = {10.4310/CAG.2023.v31.n8.a6},
    url = {https://dx.doi.org/10.4310/CAG.2023.v31.n8.a6},
    year = {2023},
    volume = {31},
    number = {8},
    pages = {2039-2087},
    journal = {Communications in Analysis and Geometry},
}

@misc{maxwell2015yamabeclassificationprescribedscalar,
      title={Yamabe Classification and Prescribed Scalar Curvature in the Asymptotically Euclidean Setting}, 
      author={Maxwell, D. and Dilts, J.},
      year={2015},
      eprint={1503.04172},
      archivePrefix={arXiv},
      primaryClass={math.DG},
      url={https://arxiv.org/abs/1503.04172}, 
}

@ARTICLE{CB2004,
    doi = {10.1088/0264-9381/21/3/009},
    url = {https://doi.org/10.1088/0264-9381/21/3/009},
    year = {2004},
    month = {1},
    publisher = {IOP Publishing},
    volume = {21},
    number = {3},
    pages = {S127--S151},
    author = {Choquet-Bruhat, Y.},
    title = {Einstein constraints on compact $n$-dimensional manifolds},
    journal = {Classical and Quantum Gravity},
    mrnumber = {2053003},
    zbl = {1040.83004},
}

@article{CBIY,
  title = {Einstein constraints on asymptotically {E}uclidean manifolds},
  author = {Choquet-Bruhat, Y. and Isenberg, J. and York, J.},
  journal = {Phys. Rev. D},
  volume = {61},
  issue = {8},
  pages = {084034},
  numpages = {20},
  year = {2000},
  month = {Mar},
  publisher = {American Physical Society},
  doi = {10.1103/PhysRevD.61.084034},
  url = {https://link.aps.org/doi/10.1103/PhysRevD.61.084034}
}

@ARTICLE{CBIP,
    doi = {10.1088/0264-9381/24/4/004},
    url = {https://doi.org/10.1088/0264-9381/24/4/004},
    year = {2007},
    month = {1},
    publisher = {IOP Publishing},
    volume = {24},
    number = {4},
    pages = {809--828},
    author = {Choquet-Bruhat, Y. and Isenberg, J. and Pollack, D.},
    title = {The constraint equations for the {E}instein-scalar field system on compact manifolds},
    journal = {Classical and Quantum Gravity},
    mrnumber = {2297268},
    zbl = {1111.83002},
}

@ARTICLE{Maxwell1,
    author = {Maxwell, D.},
    title = {{Solutions of the {E}instein Constraint Equations with Apparent Horizon Boundaries}},
    journal = {Commun. Math. Phys.},
    fjournal = {Communications in Mathematical Physics},
    volume = {253},
    year = {2005},
    pages = {561-583},
    issn = {},
    mrclass = {},
    mrnumber = {},
    mrreviewer = {},
    doi = {},
    url = {},
    zbl = {1065.83011},
}

@ARTICLE{MaxwellRoughAE,
    author = {Maxwell, D.},
    title = {Rough solutions of the {E}instein constraint equations},
    journal = {J. reine angew. Math.},
    fjournal = {Journal für die reine und angewandte Mathematik},
    volume = {590},
    year = {2006},
    pages = {1-29},
    issn = {},
    doi = {},
    url = {},
    zbl = {1088.83004},
}

@ARTICLE{ChruscielMazzeo1,
    author = {Chru{ś}ciel, P. and Mazzeo, R.},
    title = {{Initial Data Sets with Ends of Cylindrical Type: I. The {L}ichnerowicz Equation}},
    journal = {Ann. Henri Poincaré},
    fjournal = {Annales Henri Poincaré},
    volume = {16},
    year = {2015},
    pages = {1231-1266},
    issn = {},
    mrclass = {},
    mrnumber = {},
    mrreviewer = {},
    doi = {10.1007/s00023-014-0339-z},
    url = {},
    zbl = {1314.83011},
}

@ARTICLE{ChruscielMazzeo2,
    author = {Chru{ś}ciel, P. and Mazzeo, R. and  Samuel, P.},
    title = {{Initial Data Sets with Ends of Cylindrical Type: II. The Vector Constraint Equation}},
    journal = {Adv. Theor. Math. Phys. },
    fjournal = {Advances in Theoretical and Mathematical Physics},
    volume = {17},
    year = {2013},
    pages = {829-865},
    issn = {},
    mrclass = {},
    mrnumber = {},
    mrreviewer = {},
    doi = {},
    url = {},
    zbl = {},
}

@ARTICLE{ChruscielCMCHyperbolic,
    author = { Lars, A. and Chru{ś}ciel, P. and  Helmut, F.},
    title = {On the regularity of solutions to the {Y}amabe equation and the existence of smooth hyperboloidal initial data for {E}instein's field equations},
    journal = {Comm. Math. Phys.},
    fjournal = {Communications in Mathematical Physics},
    volume = {149},
    year = {1992},
    pages = {587-612},
    issn = {},
    mrclass = {},
    mrnumber = {},
    mrreviewer = {},
    doi = {},
    url = {},
    zbl = {},
}

@article{IsenbergCMCHyperbolic,
	doi = {10.1088/0264-9381/33/11/115015},
	url = {https://doi.org/10.1088/0264-9381/33/11/115015},
	year = 2016,
	month = {may},
	publisher = {{IOP} Publishing},
	volume = {33},
	number = {11},
	pages = {115015},
	author = {Allen, P. and Isenberg, J. and Lee, J. and Allen, I.},
	title = {The shear-free condition and constant-mean-curvature hyperboloidal initial data},
	journal = {Classical and Quantum Gravity},
}

@ARTICLE{AlbanseRigoli.1,
    author = {Albanese, G. and Rigoli, M.},
    doi = {10.1515/anona-2015-0106},
    url = {https://doi.org/10.1515/anona-2015-0106},
    title = {Lichnerowicz-type equations on complete manifolds},
    journal = {Advances in Nonlinear Analysis},
    number = {3},
    volume = {5},
    year = {2016},
    pages = {223--250},
    mrnumber = {3530525},
    zbl = {1349.35372},
}

@ARTICLE{AlbanseRigoli.2,
    title = {Lichnerowicz-type equations with sign-changing nonlinearities on complete manifolds with boundary},
    journal = {Journal of Differential Equations},
    volume = {263},
    number = {11},
    pages = {7475-7495},
    year = {2017},
    issn = {0022-0396},
    doi = {10.1016/j.jde.2017.08.010},
    url = {https://www.sciencedirect.com/science/article/pii/S0022039617304035},
    author = {Albanese, G. and Rigoli, M.},
    mrnumber = {3705689},
    zbl = {1377.58012},
}

@ARTICLE{MaxwellFarCMC,
    author = {Maxwell, D.},
    title = {A class of solutions of the vacuum {E}instein constraint equations with freely specified mean curvature},
    journal = {Math. Res. Lett.},
    fjournal = {Mathematical Research Letters},
    volume = {16},
    year = {2009},
    pages = {627--645},
    issn = {},
    mrclass = {},
    mrnumber = {},
    mrreviewer = {},
    doi = {},
    url = {},
    zbl = {1187.83022},
}

@ARTICLE{MaxwellRoughClosed,
    author = {Maxwell, D.},
    title = {Rough solutions of the {E}instein constraint equations on Compact Manifolds},
    journal = {J. Hyper. Differ. Eqns},
    fjournal = {Journal of Hyperbolic Differential Equations},
    volume = {2},
    year = {2005},
    pages = {521-546},
    issn = {},
    zbl = {1076.58021},
}

@ARTICLE{Valcu,
    author = {Vâlcu, C.},
    title = {{The Constraint Equations in the Presence of a Scalar Field: The Case of the Conformal Method with Volumetric Drift}},
    journal = {Commun. Math. Phys.},
    fjournal = {Communications in Mathematical Physics},
    volume = {373},
    year = {2020},
    pages = {525--569},
    issn = {},
    mrclass = {},
    mrnumber = {},
    mrreviewer = {},
    doi = {},
    url = {},
    zbl = {},
}

@ARTICLE{Premoselli1,
    author = {Premoselli, B.},
    title = {{The {E}instein-Scalar Field Constraint System in the Positive Case}},
    journal = {Commun. Math. Phys.},
    fjournal = {Communications in Mathematical Physics},
    volume = {326},
    year = {2014},
    pages = {543--557},
    issn = {},
    mrclass = {},
    mrnumber = {},
    mrreviewer = {},
    doi = {},
    url = {},
    zbl = {1285.83007},
}

@ARTICLE{Premoselli2,
    author = {Premoselli, B.},
    title = {{Effective multiplicity for the {E}instein-scalar field {L}ichnerowicz equation}},
    journal = {Calc. Var.},
    fjournal = {Calculus of Variations and Partiala Differential Equations},
    volume = {53},
    year = {2015},
    pages = {29-64},
    issn = {},
    mrclass = {},
    mrnumber = {},
    mrreviewer = {},
    doi = {},
    url = {},
    zbl = {1321.83013},
}

@ARTICLE{Nguyen,
    author = {Nguyen, T.},
    title = {{Applications of Fixed Point Theorems to the Vacuum {E}instein Constraint Equations with Non-Constant Mean Curvature}},
    journal = {Ann. Henri Poincaré},
    fjournal = {Annales Henri Poincaré},
    volume = {17},
    year = {2016},
    pages = {2237--2263},
    issn = {},
    mrclass = {},
    mrnumber = {},
    mrreviewer = {},
    doi = {},
    url = {},
    zbl = {1345.83008},
}

@ARTICLE{DiltsIsenberMazzeoAE,
    author = {Dilts, J. and Isenberg, J. and Mazzeo, R. and Meier, C.},
    title = {Non-{C}{M}{C} {S}olutions of the {E}instein {C}onstraint {E}quations on {A}symptotically {E}uclidean {M}anifolds},
    journal = {Classical and Quantum Gravity},
    fjournal = {Class. Quantum Grav.},
    volume = {31},
    year = {2014},
    pages = {065001},
    issn = {},
    mrclass = {},
    mrnumber = {},
    mrreviewer = {},
    doi = {},
    url = {},
    zbl = {1292.83009},
}

@ARTICLE{Holst1,
    author = {Holst, M. and Nagy, G. and Tsogtgerel, G.},
    title = {Rough solutions of the Einstein constraint equations on Closed Manifolds without Near {C}{M}{C} Conditions},
    journal = {Comm. Math. Phys.},
    fjournal = {Communications in Mathematical Physics},
    volume = {288},
    year = {2009},
    pages = {547-613},
    issn = {},
    mrclass = {},
    mrnumber = {},
    mrreviewer = {},
    doi = {},
    url = {},
    zbl = {1175.83010},
}

@ARTICLE{HolstLichCompact,
    author = {Holst, M. and Tsogtgerel, G.},
    title = {{The {L}ichnerowicz Equation on Compact Manifolds with Boundary}},
    journal = {Class. Quantum Grav.},
    fjournal = {Classical and Quantum Gravity},
    volume = {30},
    year = {2013},
    pages = {205011},
    issn = {},
    mrclass = {},
    mrnumber = {},
    mrreviewer = {},
    doi = {},
    url = {},
    zbl = {1276.83007},
}

@ARTICLE{HolstAE,
    author = {Holst, M. and Meier, C.},
    title = {Non-{C}{M}{C} solutions to the {E}instein constraint equations on asymptotically Euclidean manifolds with apparent horizon boundaries},
    journal = {Class. Quantum Grav.},
    fjournal = {Classical and Quantum Gravity},
    volume = {32},
    year = {2014},
    pages = {025006},
    issn = {},
    mrclass = {},
    mrnumber = {},
    mrreviewer = {},
    doi = {},
    url = {},
    zbl = {1307.83002},
}

@ARTICLE{HolstFarCMCWithBoundary,
    author = {Holst, M. and Meier, C. and Tsogtgerel, G.},
    title = {Non-{C}{M}{C} Solutions of the {E}instein Constraint Equations on Compact Manifolds with Apparent Horizon Boundaries},
    journal = {Comm. Math. Phys.},
    fjournal = {Communications in Mathematical Physics},
    volume = {357},
    year = {2018},
    pages = {467--517},
    issn = {},
    mrclass = {},
    mrnumber = {},
    mrreviewer = {},
    doi = {},
    url = {},
    zbl = {1390.83020},
}

@ARTICLE{Gicquaud1,
    author = {Gicquaud, R. and Sakovich, A.},
    title = {{A Large Class of Non-Constant Mean Curvature Solutions of the {E}instein Constraint Equations on an Asymptotically Hyperbolic Manifold}},
    volume = {310},
    journal = {Communications in Mathematical Physics},
    number = {},
    publisher = {},
    pages = {705--763},
    year = {2012},
    doi = {10.1007/s00220-012-1420-4},
    url = {},
    mrnumber = {2891872},
    zbl = {1247.83010},
}

@article{Gicquaud2,
	doi = {10.1088/0264-9381/31/19/195014},
	url = {https://doi.org/10.1088/0264-9381/31/19/195014},
	year = 2014,
	month = {sep},
	publisher = {{IOP} Publishing},
	volume = {31},
	number = {19},
	pages = {195014},
	author = {Gicquaud, R. and Qu{\^{o}}c Anh Ng{\^{o}}},
	title = {{A new point of view on the solutions to the {E}instein constraint equations with arbitrary mean curvature and small {TT}-tensor}},
	journal = {Classical and Quantum Gravity},
}

@ARTICLE{MaxwellModelProblem,
    author = {Maxwell, D.},
    title = {A Model Problem for Conformal Parameterizations of the {E}instein Constraint Equations},
    journal = {Comm. Math. Phy.},
    fjournal = {Communications in Mathematical Physics},
    volume = {302},
    year = {2011},
    pages = {697--736},
    zbl = {1215.53064},
}

@ARTICLE{Maxwell-KasnerST,
    doi = {10.1007/s00023-014-0386-5},
    url = {},
    year = {2015},
    month = {},
    publisher = {},
    volume = {16},
    number = {},
    pages = {2919-2954},
    author = {Maxwell, D.},
    title = {Conformal Parameterizations of Slices of Flat {K}asner Spacetimes},
    journal = {Ann. Henri Poincaré},
    mrnumber = {},
    zbl = {},
}

@ARTICLE{MaxwellDriftModel,
    doi = {10.4310/CAG.2021.v29.n1.a7},
    url = {},
    year = {2021},
    month = {},
    publisher = {},
    volume = {29},
    number = {1},
    pages = {207-281},
    author = {Maxwell, D.},
    title = {Initial data in general relativity described by expansion, conformal deformation and drift},
    journal = {Comm. Anal. Geom.},
    mrnumber = {4234983},
    zbl = {07333646},
}

@article{IsenbergMurchadha,
    year = {2004},
    publisher = {},
    volume = {21},
    number = {3},
    pages = {S233},
    author = {Isenberg, J. and O Murchadha, N.},
    title = {Non-{C}{M}{C} conformal data sets which do not produce solutions of the {E}instein constraint equations},
    journal = {Classical and Quantum Gravity},
    mrnumber = {2053007},
    zbl = {1042.83007},
}

@article{Walsh-nonuniq,
	doi = {10.1088/0264-9381/24/8/002},
	url = {https://doi.org/10.1088/0264-9381/24/8/002},
	year = 2007,
	month = {mar},
	publisher = {{IOP} Publishing},
	volume = {24},
	number = {8},
	pages = {1911--1925},
	author = {Walsh, D. M. },
	title = {Non-uniqueness in conformal formulations of the {E}instein constraints},
	journal = {Classical and Quantum Gravity},
}

@article{YamabeEqBG,
	doi = {},
	url = {},
	year = 2013,
	month = {},
	publisher = {},
	volume = {21},
	number = {5},
	pages = {957--978},
	author = {Grosse, N.},
	title = {The {Y}amabe equation on manifolds of bounded geometry},
	journal = {Comm. Annal. Geom.},
}

@article{BoundaryValueRegBG,
author = {Grosse, N. and Nistor, V.},
title = {Uniform {S}hapiro-{L}opatinski Conditions and Boundary Value Problems on Manifolds with Bounded Geometry},
journal = {Potential Analysis},
volume = {53},
number = {},
pages = {407-447},
doi = {https://doi.org/10.1007/s11118-019-09774-y},
url = {},
eprint = {},
}

@article{LaplacianBG,
author = {Ammann, B. and Grosse, N. and Nistor, V.},
title = {Well-posedness of the {L}aplacian on manifolds with boundary and bounded geometry},
journal = {Mathematische Nachrichten},
volume = {292},
number = {6},
pages = {1213-1237},
doi = {https://doi.org/10.1002/mana.201700408},
url = {https://onlinelibrary.wiley.com/doi/abs/10.1002/mana.201700408},
eprint = {https://onlinelibrary.wiley.com/doi/pdf/10.1002/mana.201700408},
}

@article{Shubin2,
     author = {Shubin, M. },
     title = {Weak {B}loch property and weight estimates for elliptic operators},
     journal = {S\'eminaire \'Equations aux d\'eriv\'ees partielles (Polytechnique) dit aussi "S\'eminaire Goulaouic-Schwartz"},
     note = {talk:5},
     publisher = {Ecole Polytechnique, Centre de Math\'ematiques},
     year = {1989-1990},
     zbl = {0716.35052},
     mrnumber = {1073180},
     language = {en},
     url = {http://www.numdam.org/item/SEDP_1989-1990____A5_0/}
}

@article{GreeneBG,
     author = {Greene, R.},
     title = {Complete metrics of bounded curvature on noncompact manifolds},
     journal = {Arch. Math},
     note = {},
     publisher = {},
     year = {1978},
     zbl = {},
     mrnumber = {},
     language = {},
     url = {}
}

@BOOK{RingstromBook,
 Author = {Ringstrom, H.},
 Title = {{The Cauchy Problem in General Relativity}},
 FJournal = {ESI Lectures in Mathematics and Physics},
 Journal = {ESI Lect. Math. Phys.},
 ISBN = {978-3-03719-053-1/pbk},
 Pages = {xiii + 294},
 Year = {2009},
 Publisher = {Zurich: European Mathematical Society (EMS)},
 Language = {English},
 MSC2010 = {83-02 83C05 35Q75 83A05 83C25 83C75 83C15 83F05},
 Zbl = {1169.83003}
}

@ARTICLE{GerochSplittingThm,
    author = {Geroch, R.},
    title = {{Domain of Dependence}},
    journal = {Journal of Mathematical Physics},
    volume = {11},
    number = {2},
    pages = {437-449},
    year = {1970},
    doi = {10.1063/1.1665157},
    url = { https://doi.org/10.1063/1.1665157 },
    mrnumber = {},
    zbl = {0189.27602},
}

@ARTICLE{SanchezSplittingThm,
    author = {Bernal, A. and Sánchez, M.},
    title = {On {S}mooth {C}auchy {H}ypersurfaces and {G}eroch’s {S}plitting {T}heorem},
    journal = {Communications in Mathematical Physics},
    volume = {243},
    number = {},
    pages = {461-470},
    doi = {10.1007/s00220-003-0982-6},
    url = {},
    year = {2003},
    mrnumber = {2029362},
    zbl = {1085.53060},
}

@ARTICLE{SanchezSplittingThm2,
    author = {Bernal, A. and Sánchez, M.},
    title = {Smoothness of {T}ime {F}unctions and the {M}etric {S}plitting of {G}lobally {H}yperbolic {S}pacetimes},
    journal = {Communications in Mathematical Physics},
    volume = {257},
    number = {},
    pages = {43-50},
    doi = {10.1007/s00220-005-1346-1},
    year = {2005},
    mrnumber = {2163568},
    zbl = {1081.53059},
}

@ARTICLE{IsenbergCMC,
    doi = {10.1088/0264-9381/12/9/013},
    url = {https://doi.org/10.1088/0264-9381/12/9/013},
    year = {1995},
    month = {9},
    publisher = {IOP Publishing},
    volume = {12},
    number = {9},
    pages = {2249--2274},
    author = {Isenberg, J.},
    title = {Constant mean curvature solutions of the {E}instein constraint equations on closed manifolds},
    journal = {Classical and Quantum Gravity},
    mrnumber = {1353772},
    zbl = {0840.53056},
}

@ARTICLE{IsenbergMoncrief,
    year = {1996},
    publisher = {},
    volume = {13},
    number = {7},
    pages = {1819},
    author = {Isenberg, J. and Moncrief, V.},
    title = {A set of nonconstant mean curvature solutions of the {E}instein constraint equations on closed manifolds},
    journal = {Classical and Quantum Gravity},
    mrnumber = {1400943},
    zbl = {0860.53056},
}

@book {MR2962687,
    AUTHOR = {Mastrolia, P. and Rigoli, M. and Setti, A.},
     TITLE = {{Yamabe-type Equations on Complete, Noncompact Manifolds}},
    SERIES = {Progress in Mathematics},
    VOLUME = {302},
 PUBLISHER = {Birkh\"{a}user/Springer Basel AG, Basel},
      YEAR = {2012},
     PAGES = {viii+256},
      ISBN = {978-3-0348-0375-5},
   MRCLASS = {58-02 (35J61 35R01 53-02 53C20 58J05)},
  MRNUMBER = {2962687},
MRREVIEWER = {David L. Finn},
       DOI = {10.1007/978-3-0348-0376-2},
       URL = {https://doi.org/10.1007/978-3-0348-0376-2},
}

@book {MR2473363,
    AUTHOR = {Choquet-Bruhat, Y.},
     TITLE = {General {R}elativity and the {E}instein {E}quations},
    SERIES = {Oxford Mathematical Monographs},
 PUBLISHER = {Oxford University Press, Oxford},
      YEAR = {2009},
     PAGES = {xxvi+785},
      ISBN = {978-0-19-923072-3},
   MRCLASS = {83-02 (35Q75 83C05)},
  MRNUMBER = {2473363},
MRREVIEWER = {Hans-Peter K\"{u}nzle},
}

@article{GunPig2,
title = {The {C}alderón–{Z}ygmund inequality and {S}obolev spaces on noncompact {R}iemannian manifolds},
journal = {Advances in Mathematics},
volume = {281},
pages = {353-393},
year = {2015},
issn = {0001-8708},
doi = {https://doi.org/10.1016/j.aim.2015.03.027},
url = {https://www.sciencedirect.com/science/article/pii/S0001870815001838},
author = {Güneysu, B. and Pigola, S.}
}

@article{arxivrodrigo,
	author = {Avalos, R. and Lira, J.},
	title = {Einstein-{T}ype {E}lliptic {S}ystems },
	journal = {Annales Henri Poincaré},
	    VOLUME = {23},
	    NUMBER = {2},
	      ISSN = {1424-0661},
	   MRCLASS = {83C75 (83C05)},
	  MRNUMBER = {2030884},
	MRREVIEWER = {Norbert Noutchegueme},
	doi={10.1007/s00023-022-01180-2},
	year = {2022},
       URL = {https://doi.org/10.1007/s00023-022-01180-2},
}

@article{https://doi.org/10.1002/mana.201300007,
author = {Grosse, N. and Schneider, C.},
title = {Sobolev spaces on {R}iemannian manifolds with bounded geometry: General coordinates and traces},
journal = {Mathematische Nachrichten},
volume = {286},
number = {16},
pages = {1586-1613},
doi = {https://doi.org/10.1002/mana.201300007},
url = {https://onlinelibrary.wiley.com/doi/abs/10.1002/mana.201300007},
eprint = {https://onlinelibrary.wiley.com/doi/pdf/10.1002/mana.201300007},
year = {2013}
}

@book{bibellipticpartialdifferentialequations,
	author = {Gilbarg, D. and Trudinger, N.},
	isbn = {3-540-41160-7},
	mrclass = {35-02 (35Jxx)},
	mrnumber = {1814364},
	note = {Reprint of the 1998 edition},
	pages = {xiv+517},
	publisher = {Springer-Verlag, Berlin},
	series = {Classics in Mathematics},
	title = {{Elliptic Partial Differential Equations of Second Order}},
	year = {2001}
}

@incollection {Shubin,
    AUTHOR = {Shubin, M.},
     TITLE = {Spectral theory of elliptic operators on non-compact manifolds},
      NOTE = {M\'{e}thodes semi-classiques, Vol. 1 (Nantes, 1991)},
   JOURNAL = {Ast\'{e}risque},
  FJOURNAL = {Ast\'{e}risque},
    NUMBER = {207},
      YEAR = {1992},
     PAGES = {5, 35--108},
      ISSN = {0303-1179},
   MRCLASS = {58G25 (35P05 47F05)},
  MRNUMBER = {1205177},
MRREVIEWER = {M. S. Agranovich},
}


\end{document}